\documentclass[final]{ws-ijnt}
\usepackage{cite}
\usepackage{xcolor}
\usepackage{amsmath,amssymb}
\usepackage{bbm}
\usepackage{tikz-cd}
\usepackage{mathtools}
\usepackage{mathrsfs}
\usepackage[verbose,hypertexnames=false]{hyperref}
\hypersetup{colorlinks=false,allbordercolors=blue,pdfborderstyle={/S/U/W 1}}

\newcommand\restr[2]{{
  \left.\kern-\nulldelimiterspace 
  #1 
  \vphantom{\big|} 
  \right|_{#2} 
  }}
   \newcommand{\KKv}{\KK_v}
  \newcommand{\OO}{\mathcal{O}}
  \newcommand{\LL}{{\mathbb{L}}}
\newcommand{\KK}{{\mathbb{K}}}
\newcommand{\QQ}{{\mathbb{Q}}}
\newcommand{\NN}{{\mathbb{N}}}
\newcommand{\ZZ}{{\mathbb{Z}}}
\newcommand{\CC}{{\mathbb{C}}}
\newcommand{\A}{{\mathbb{A}}}
\newcommand{\RR}{{\mathbb{R}}}

\newcommand{\DD}{{\mathbb{D}}}
\newcommand{\GL}{{\mathrm{GL}}}
\newcommand{\re}{{\mathrm{Re}(s)}}

\newcommand{\fm}{\mathfrak{m}}

\newcommand{\HH}{\mathbb{H}}
\newcommand{\VV}{\mathcal{V}}
\newcommand{\ra}{\rightarrow}

\newcommand{\Sp}{\mathrm{Sp}}

\newcommand{\bs}{\backslash}

\newcommand{\tr}{\mathrm{Tr}}
\newcommand{\Kb}{{\mathbb{K}_v}}
\newcommand{\DDb}{\mathrm{D}}
\newcommand{\bpm}{\begin{pmatrix}}
\newcommand{\epm}{\end{pmatrix}}

\newcommand{\ain}[3]{#1\in\{#2,\ldots,#3\}}

\newcommand{\iso}{\overset{\cong}{\rightarrow}}

\newcommand{\ddd}{\,\mathrm{d}}

\newcommand{\lig}{\mathfrak{g}}
\newcommand{\jl}{\mathrm{JL}}
\newcommand{\lj}{\mathrm{LJ}}
\newcommand{\mw}{\mathrm{MW}}

\renewcommand*{\det}{\qopname\relax o{det}}

\newcommand{\lik}{\mathfrak{k}}
\newcommand{\dv}{\,\mathrm{d}_v}
\newcommand{\SG}{{\textbf{S}_{K_f'}^{\GL'_{2n}}}}
\newcommand{\SH}{{\textbf{S}_{K_f'}^{H'_n}}}
\newcommand{\Sh}{{\textbf{S}_{K_f'}^{H'_1}}}
\newcommand{\Sg}{{\textbf{S}_{K_f'}^{\GL'_{2}}}}
\newcommand{\mE}{\mathcal{E}_\mu^\lor}
\newcommand{\gK}{\left(\mathfrak{g}'_\infty,K'_\infty\right)}
\newcommand{\tbf}{\textbf{t}}
\newcommand{\CCrit}{\mathrm{Crit}}
\newcommand{\diag}{\mathrm{diag}}
\newcommand{\ato}[2]{\genfrac{}{}{0pt}{}{#1}{#2}}
\newcommand{\Aut}{\mathrm{Aut}}
\newcommand{\mS}{\mathcal{S}}
\newcommand{\df}{\, \mathrm{d}_f}
\newcommand{\bsg}{\mathbf{S}^{\GL_{2n}'}}
\newcommand{\SL}{\mathrm{SL}}

\begin{document}
\def\wsdraftnote{}
\markboth{Johannes Droschl}{Critical values of $L$-functions of residual representations of $\GL_4$}

\catchline{}{}{}{}{}

\title{Critical values of $L$-functions of residual representations of $\GL_4$}

\author{Johannes Droschl}

\address{}

\maketitle

\begin{abstract}
  In this paper we prove rationality results for critical values for $L$-functions attached to representations in the residual spectrum of $\GL_4(\A)$. We use the Jacquet-Langlands correspondence to describe their partial $L$-functions via cuspidal automorphic representations of the group $\GL'_2(\A)$ over a quaternion algebra. Using ideas inspired by results of Grobner and Raghuram we are then able to compute the critical values as a Shalika period up to a rational multiple.
\end{abstract}



\section{Introduction}
Let $\DD$ be a division algebra over a totally real number field $\KK$, which is non-split at every place at infinity. Denote by $\Sigma_\DD$ the set of places where $\DD$ splits and by $\A=\A_\KK$ the adeles. We let
$M_{n,n}$ be the algebraic variety of $n\times n$ matrices over $\KK$ and
$\GL_{n}$ be the general linear group over $\KK$. Similarly, let $M_{n,n}'$ be the variety of $n\times n$ matrices with coefficients in $\DD$ and let $\GL'_{n,\DD}=\GL'_{n}$ be the group of invertible matrices in $M_{n,n}'$, where we see both varieties as algebraic groups over $\KK$.
In \cite{GroRag} the authors proved certain rationality results of critical values of the $L$-function of cohomological cuspidal irreducible automorphic representations of $\GL_{2n}(\A)$, which admit a Shalika model. The goal of this paper is to extend these results to non-cuspidal discrete series representations of $\GL_4\left(\A\right)$ by lifting them from cuspidal irreducible representations of $\GL'_2\left(\A\right)$ by use of the Jacquet-Langlands correspondence $\jl$, see \cite{BadRen}.

Let \[\mS \coloneqq \Delta \GL'_n\rtimes U_{(n,n)}'=\left\{\bpm h&X\\0& h\epm:\, h\in \GL_n',\, X\in M_{n,n}' \right\} \] be the Shalika subgroup of $\GL'_{2n}$. We say that an irreducible cuspidal automorphic representation $\Pi'$ of $\GL'_{2n}\left(\A\right)$ with central character $\omega$ admits a Shalika model with respect to a character $\eta$, if $\eta^n=\omega$ and if the Shalika period 
\[\mS_\psi^\eta\left(\phi\right)\left(g\right)\coloneqq \int_{Z'_{2n}\left(\A\right)\mS\left(\KK\right)\bs\mS\left(\A\right)}\phi\left(sg\right)\psi\left(\tr(X)\right)^{-1}\eta\left(\det'(h)\right)^{-1}\ddd s\neq 0\] does not vanish
for some $\phi\in\Pi'$ and $g\in \GL'_{2n}\left(\A\right)$.

In the split case, \emph{i.e.} $\DD=\KK$, it is well known that $\Pi'$ admits a Shalika model with respect to $\eta$ if and only if the twisted partial exterior square $L$-function
$L^S\left(s,\Pi',\bigwedge ^2\otimes\eta^{-1}\right)$
has a pole at $s=1$. In the non-split case there is currently no analogous theorem known, however, in the special case $n=1$ and $\DD$ a quaternion division algebra the following was proved in \cite{GanTakII}.
We recall quickly the M{\oe}glin-Waldspurger classification of discrete series representation. Namely, for $
\Sigma$ a cuspidal representation of $\GL_l(\A)$ and $k\in\NN$, one can construct a discrete series representation $\mw(\Sigma,k)$ of $\GL_{kl}(\A)$.
\begin{theorem}[{\cite[Theorem 1.3]{GanTakII}}]\label{T:G-TI}
Assume $\DD$ is a quaternion division algebra and $\Pi'$ a cuspidal irreducible automorphic representation of $\GL'_{2}\left(\A\right)$. 
If $\jl\left(\Pi'\right)$ is cuspidal and irreducible, the following assertions are equivalent.
\begin{enumerate}
    \item $\Pi'$ admits a Shalika model with respect to $\eta$.
    \item The twisted partial exterior square $L$-function $L^S\left(s,\Pi',\bigwedge^2\otimes\eta^{-1}\right)$ has a pole at $s=1$ and for all $v\in \VV_\DD$, $\Pi'_v$ is not isomorphic to a parabolically induced representation \[\lvert\det'\lvert_v^\frac{1}{2}\tau_1'\times\lvert\det'\lvert_v^{-\frac{1}{2}}\tau_2',\]
    where $\tau_i'$ are representations of $\GL'_1\left(\KK_v\right)$ with central character $\eta_v$.
\end{enumerate}
If $\jl\left(\Pi'\right)$ is not cuspidal, $\jl\left(\Pi'\right)=\mw(\Sigma,2)$ for some cuspidal irreducible representation $\Sigma$ of $\GL_2(\A)$. Then the following assertions are equivalent.
\begin{enumerate}
    \item $\Pi'$ admits a Shalika model with respect to $\eta$.
    \item The central character $\omega_\Sigma$ of $\Sigma$ equals $\eta$.
    \item The twisted partial exterior square $L$-function $L^S\left(s,\Pi',\bigwedge^2\otimes\eta^{-1}\right)$ has a pole at $s=2$.
\end{enumerate}
\end{theorem}
For the rest of the introduction assume $\DD$ is a quaternion algebra and $\Pi'$ be an irreducible cuspidal cohomological automorphic representation of $\GL_2'\left(\A\right)$ with respect to a coefficient system $E_\mu^\lor$.
Note that for a cuspidal $\Pi'$ and $\sigma\in\Aut\left(\CC\right)$ one can define the $\sigma$-twist ${}^\sigma \Pi'_f$ of the finite part of $\Pi'_f$. Following \cite{GroRag2} we extend this to a $\sigma$-twist ${}^\sigma\Pi'$ of $\Pi'$, which is a discrete series representation of $\GL_{2}'\left(\A\right)$. In \cite{GroRag2} it was shown that if moreover $\jl\left(\Pi'\right)$ is cuspidal, ${}^\sigma\Pi'$ is again cuspidal. We prove that the assumption of $\jl\left(\Pi'\right)$ being cuspidal is not necessary and extend their argument using the M{\oe}glin-Waldspurger classification to the case when $\jl\left(\Pi'\right)$ is residual.
Using the above criterion for admitting a Shalika model, we see that if $\Pi'$ admits a Shalika model then so does ${}^\sigma \Pi'$. 
Let $\QQ\left(\Pi_f'\right)$ be the field fixed by the automorphisms fixing $\Pi'_f$.
In \cite{GroRag2} it was shown that $\QQ\left(\Pi_f'\right)$ is a number field and that $\QQ\left(\Pi'_f\right)=\QQ\left(\jl\left(\Pi'\right)_f\right)$. Following \cite{JiaSunTia}, \cite{GroRag} we define a finite extension $\QQ\left(\Pi',\eta\right)$ of $\QQ\left(\Pi_f'\right)$ and a $\QQ\left(\Pi',\eta\right)$-structure on the Shalika model $\mS_{\psi_f}^{\eta_f}\left(\Pi_f'\right)$ of $\Pi'_f$.

As in \cite{GroRag} we will make use of a \emph{numerical coincidence}, which is together with \ref{T:G-TI}, \ref{T:Sun} and \ref{T:nien} the reason why we must limit ourselves to the case $\DD$ being quaternion and $n=1$.  Let $q_0$ be the lowest degree in which the $\left(\lig ',K'_\infty\right)$-cohomology of $\pi_\infty'\otimes E_\mu^\lor$ does not vanish. Then we have $q_0=\dim_\QQ\KK$ and 
\[\dim_\CC H^{q_0}\left(\lig'_\infty,K'_\infty,\Pi_\infty'\otimes E_\mu^\lor\right)=1.\]
By fixing a basis vector of this one-dimensional vector space, we can define an isomorphism
\[\Theta_{\Pi'}\colon\mS_{\psi_f}^{\eta_f}\left(\Pi_f'\right)\rightarrow H^{q_0}\left(\lig'_\infty,K'_\infty,\Pi'\otimes E_\mu^\lor\right),\] where the right hand side inherits a {$\QQ\left(\Pi',\eta\right)$-structure} from its geometric realization as automorphic cohomology.
Thus we can normalize the above isomorphism by a factor $\omega\left(\Pi_f'\right)$,
the so-called Shalika period, such that it respects the $\QQ\left(\Pi',\eta\right)$-structures of both sides. Analogously to \cite{GroRag} we compute how $\omega\left(\Pi_f'\right)$ behaves under twisting with a Hecke character $\chi$ of $\GL_1(\A)$ lifted to $\GL_2'\left(\A\right)$ via the determinant map. Let $\mathcal{G}\left(\chi_f\right)$ be the Gauss sum of $\chi_f$. Then
\[\sigma\left(\frac{\omega\left({\Pi'}_f\otimes\chi_f\right)}{\mathcal{G}\left(\chi_f\right)^4\omega\left({\Pi'}_f\right)}\right)=\frac{\omega\left({}^\sigma{\Pi'}_f\otimes {}^\sigma\chi_f\right)}{\mathcal{G}\left({}^\sigma\chi_f\right)^4\omega\left({}^\sigma{\Pi'}_f\right)}\] for $\sigma\in \Aut\left(\CC\right)$.

The next ingredient is the Shalika zeta-integral, first introduced in \cite{FriJac}, and extended to $\GL_2'\left(\A\right)$,
\[\zeta\left(s,\phi\right)\coloneqq \int_{\GL'_1\left(\A\right)}S^\eta_\psi\left(\phi\right)\left(\bpm g_1&0\\0&1\epm\right)\lvert\det\left(g_1\right)\lvert^{s-\frac{1}{2}}\ddd g_1\]
and its local analogs. As in \cite{GroRag} we fix a special vector $\xi^0_{{\Pi'}_f}\in \mS_{\psi_f}^{\eta_f}\left({\Pi'}_f\right)$ such that
\[\zeta_v\left(\frac{1}{2},\xi^0_{{\Pi'}_f}\right)=L\left(\frac{1}{2},\pi_v\right)\] if $v$ is a finite place at which $\psi$ and $\Pi'$ are unramified. By \cite{FriJac} the period integral over $H'_1=\GL'_1\times \GL'_1$ of a cusp form is precisely the Shalika zeta integral. To show the invariance of this period integral under the action of a Galois group, we first interpret it as an instance of Poincar\'e duality of the top cohomology group of the space
\[\textbf{S}_{K_f'}^{H'_1}=H'_1\left(\KK\right)\bs H'_1\left(\A\right)/\left(K'_\infty\cap {H'_1}_\infty\right)\iota^{-1}\left(K_f'\right),\]
where $\iota\colon H'_1\hookrightarrow \GL'_{2}$ is the block-diagonal embedding and $K_f'$ a small enough open compact subgroup of $\GL'_{2}\left(\A_{f}\right)$. To make the whole story work it is crucial that 
$\dim_\mathbb{R} \textbf{S}_{K_f'}^{H'_1}=q_0$, which only works if we restrict ourselves to the case $n=1$ and $\DD$ being a quaternion algebra, the aforementioned numerical coincidence.
Since we assume that $\jl\left(\Pi'\right)$ is residual, we then compute that the critical points of $L\left(s,\Pi'\right)$ are all half-integers $s=\frac{1}{2}+m ,\,m\in\mathbb{Z}$ with $-\mu_{v,2}\le m\le -\mu_{v,3}$ for all infinite places $v$.

Since we assume $\KK$ to be totally real, we show as in \cite{GroRag} that a certain representation $E_{\left(0,-w\right)}$ of $H'_1$ appears in the coefficient system $E_\mu^\lor$ of $\pi_\infty'$ if $\frac{1}{2}$ is a critical point of the $L$-function, which in turn lets us map the fixed special vector $\xi_{{\Pi'}_f}^0$ first to $H^{q_0}\left(\lig'_\infty,K'_\infty,\mS_\psi^\eta\left({\Pi'}\right)\otimes E_\mu^\lor\right)$ and then interpret it as an element of 
$H^{q_0}_c\left(\mathbf{S}_{K_f'}^{H'_1},\mE\right),$ which we then map to $H^{q_0}_c\left(\textbf{S}_{K_f'}^{H'_1},\mathcal{E}_{\left(0,-w\right)}\right)$ using the map from above, where $\mE$ and $\mathcal{E}_{\left(0,-w\right)}$ are the sheaves on $\textbf{S}_{K_f'}^{H'_1}$ associated to $E_\mu^\lor $ and $E_{\left(0,-w\right)}$. Finally, applying Poincar\'e duality to this last space, we show that the resulting number is essentially the value of the $L$-function $L(s,\Pi')$ at $s=\frac{1}{2}$. Now the final result of \cite{GroRag} for critical values of the $L$-function follows analogously in our case, namely if $s=\frac{1}{2}+m$, there exist periods $\omega\left(\Pi'_f\right)$ and $\omega\left(\pi_\infty',m\right)$ such that
\[\sigma\left(\frac{L\left(\frac{1}{2}+m,\Pi_f'\otimes\chi_f\right)}{\omega\left(\Pi_f'\right)\mathcal{G}\left(\chi_f\right)^4\omega\left(\pi_\infty',m\right)}\right)=\frac{L\left(\frac{1}{2}+m,{}^\sigma\Pi_f'\otimes{}^\sigma\chi_f\right)}{\omega\left({}^\sigma\Pi_f'\right)\mathcal{G}\left({}^\sigma\chi_f\right)^4\omega\left(\pi_\infty',m\right)}\]
for all $\sigma\in\Aut\left(\CC/\QQ\left(\Pi',\eta\right)\right)$. Let $\QQ\left(\Pi',\eta,\chi\right)$ be the compositum of $\QQ\left(\Pi',\eta\right)$ and $\QQ\left(\chi\right)$. This implies that
\[\frac{L\left(\frac{1}{2}+m,\Pi_f'\otimes\chi_f\right)}{\omega\left(\Pi_f'\right)\mathcal{G}\left(\chi_f\right)^4\omega\left(\pi_\infty',m\right)}\in \QQ\left(\Pi',\eta,\chi \right) \] and hence, proves the main result.
\begin{theorem}
Let $\Pi$ be a non-cuspidal discrete series representation of $\GL_4\left(\A\right)$ with trivial central character written as $\Pi\cong \mw(\Sigma\lvert\det\lvert^{\frac{1}{2}}\times \Sigma\lvert\det\lvert^{-\frac{1}{2}})$ via the M{\oe}glin-Waldspurger classification, where $\Sigma$ is a cuspidal irreducible representation of $\GL_2\left(\A\right)$. Assume moreover that there exists an irreducible cuspidal cohomological representation $\Pi'$ of $\GL'_2\left(\A\right)$ with $\jl\left(\Pi'\right)=\Pi$ which is cohomological with respect to coefficient system $E_\mu^\lor$. Let $\chi$ be a finite order Hecke-character of $\GL_1\left(\A\right)$ and $s=\frac{1}{2}+m$ a critical point of $L\left(s,\Pi'\right)$. Then
\[\frac{L\left(\frac{1}{2}+m,\Pi_f\otimes \chi_f\right)}{\omega\left(\Pi_f'\right)\mathcal{G}\left(\chi_f\right)^4\omega\left(\pi_\infty',m\right)}\in \QQ\left(\Pi',\chi\right).\]
\end{theorem}
\subsection*{Acknowledgements:}
I would like to thank Harald Grobner for his many helpful comments and patience as well as Binyong Sun for pointing out a mistake in an earlier version. Finally, I would also like to thank the anonymous reviewer for spotting some inconsistencies in an earlier version. This work has been supported by the research project P32333 of the Austrian Science Fund (FWF).
\section{Preliminaries}
We start by fixing our notations regarding automorphic representations
\subsection{ Adelic notation}\label{S:2.2}
Let $\KKv$ be a local non-archimedean field and $\DD_v$ be a central division algebra over $\KKv$ of degree $d_v$.
We let $\OO_v$ be the ring of integers of $\KK_v$ and fix a uniformizer $\varpi_v$, \emph{i.e.} a generator of the maximal ideal of $\OO_v$.
We extend the valuation $v$ of $\KK_v$ to a valuation $v'$ of $\DD_v$ by \[v'\left(x\right)\coloneqq \frac{1}{d_v}v\left(\mathrm{Nr}_{\DD_v/\KKv}\left(x\right)\right),\] where we denote by $\mathrm{Nr}_{\DD_v/\KKv}\colon \DD_v\ra\KKv$ the reduced norm. Define the ring of integers of $\DD_v$ as
$\OO_v'\coloneqq \{x\in \DD_v:v'\left(x\right)\ge 0\}.$
Let $\KK$ be a number field with $\dim_\QQ\KK=r$, $\OO$ its ring of integers, and let $\DD$ be a central division algebra of degree $d$ over $\KK$. Recall the set of places $\VV$, which decomposes into the finite places ${\VV_f}$ and the infinite places ${\VV_\infty}$. For $v\in \VV$, one has $\DD\otimes \KKv\cong M_{r_v,r_v}\left(\DD_v\right),$ where $\DD_v$ is a central division algebra of dimension $d_v^2$ over $\KKv$ and $d_vr_v=d$. 
If $\DD_v=\KKv$ we call $v$ a split place of $\DD$. From now on we assume that $\DD$ is non-split at all infinite places. Equivalently, $\KK$ is totally real and for $v\in {\VV_\infty}$, $\DD_v=M_{\frac{d}{2},\frac{d}{2}}\left(\HH\right)$, where $\HH$ denotes the Hamilton quaternions. We denote the places where $\DD$ is non-split by $\VV_\DD$.
Finally, let $\mathfrak{D}$ be the absolute different of $\KK$, \emph{i.e.} ${\mathfrak{D}^{-1}\coloneqq \{x\in\KK: \mathrm{Tr}_{\KK/\QQ}\left(x\OO\right)\subseteq\mathbb{Z}\}}$. 
We will also fix the standard non-trivial additive character $\psi\colon \KK\bs \A_\KK\rightarrow \CC^\times.$ Note that the finite places where $\psi$ ramifies correspond precisely to the prime ideals $\mathfrak{p}$ not dividing $\mathfrak{D}$. 
We will write from now on $\A$ for the adeles of $\KK$. 
\subsection{ The general linear group}
We will quickly introduce the reductive groups relevant to us and fix our notation regarding tori and parabolic subgroups.
Let $\mathrm{K}$ be a field and denote by $\GL_n$ the $n$-th general linear group over $\mathrm{K}$ with the usual maximal torus $T_n$ of diagonal matrices and fixed Borel subgroup $B_n$ of upper-triangular matrices, giving rise to a set of positive roots.
To each dominant weight $\mu\in X^*(T_n)$ one can associate a highest weight representation $E_\mu$ of $\GL_n(\CC)$. Recall that the parabolic subgroups over $\Kb$ containing $B_n$ are then parameterized by compositions of $n$. 
In other words, to $\alpha=\left(\alpha_1,\ldots,\alpha_k\right)$ a composition of $n$ we associate the parabolic subgroup $P_\alpha$ of $\GL_{n}$
containing the upper triangular matrices and having as a Levi-component the block-diagonal matrices $M_\alpha=\GL_{\alpha_1}\times\ldots\times \GL_{\alpha_k}$ and unipotent component $U_\alpha$. 

Let $\DDb$ be a central division algebra over $\mathrm{K}$ of degree $d$. Let $M_{n,n}$ be the variety whose $\mathrm{K}$ points are the $n\times n$ matrices with entries in $\mathrm{K}$ and let $M_{n,n}'$ be the variety whose $\mathrm{K}$ points are the $n\times n$ matrices with entries in $\DDb$. We recall the determinant $\det'\colon M_{n,n}'\rightarrow M_{1,1}$ and trace map $\tr\colon M_{n,n}'\rightarrow M_{1,1}$. We denote by $\GL_n'$ the elements with non-zero determinant in $M_{n,n}'$
and the center of $\GL'_n$ by $Z'_n$. 
Again we can assign to each composition $\alpha$ of $n$ a standard parabolic subgroup $P_\alpha$ of $\GL_n'$ defined over $\Kb$, containing the upper triangular matrices and having as a Levi-component the block-diagonal matrices $M_\alpha=\GL_{\alpha_1}'\times\ldots\times \GL_{\alpha_k}'$ and unipotent component denoted by $U_\alpha'$. Then $\overline{P_\alpha'}$ is again conjugated to $P_{\overline{\alpha}}'$. 
We extend the notions of highest weight representations of $\GL_n$ to $\GL_n'$ as follows. If $\mathrm{K}=\RR$ and $\DDb=\HH$, a representation $E_\mu$ of $\GL_n'(\RR)$ is called a highest weight representation with dominant weight $\mu\in \ZZ^{2n}$, if the corresponding complexified representation of $\GL_{2n}(\CC)$ is a highest weight representation with weight $\mu$.
Define finally $H_n=\GL_n\times \GL_n$, $H_n'=\GL_n'\times \GL_n'$. 
\subsection{ Automorphic representations}\label{S:2.3}
Let $\KK,\OO,\DD,r,d$ as in \ref{S:2.2}. We will highlight the basic properties and constructions regarding automorphic representations of $\GL_n(\A)$ and $\GL_n'(\A)$. 

For $v\in {\VV_\infty}$, let $Z'_{n,v}$ be the center of $\GL'_n\left(\KKv\right)$ and let $K'_v$ be the product of the maximal compact subgroup of $\GL'_n\left(\KKv\right)$ and the connected component of $Z'_{n,v}$, \emph{i.e.}
\[K'_v\coloneqq \mathrm{Sp}\left(\frac{nd}{2}\right)\RR_{> 0},\, K_\infty'\coloneqq \prod_{v\in {\VV_\infty}} K'_v.\]
Note that $\mathrm{Sp}(n)$ does not denote the standard algebraic symplectic group, we denote it by $\Sp_n$, rather it denotes the \emph{compact symplectic group} \[\Sp(n)\coloneqq \Sp_{2n}(\CC)\cap \mathrm{U}_n(\RR)\] or, alternatively the quaternionic unitary group. The group $\Sp(n)$ is a real Lie group of dimension $\dim_\RR \Sp(n)=n(2n+1)$, it is compact and simply connected.
Similarly, we fix for $v\in {\VV_\infty}$
\[K_v\coloneqq \mathrm{SO}_n(\RR)\RR_{>0},\, K_\infty\coloneqq \prod_{v\in {\VV_\infty}} K_v.\]
Moreover, we fix also open compact subgroups $K_v'$ of $\GL_n'\left(\KKv\right)$ for $v\in {\VV_f}$ as follows. Note that $\GL_n'\left(\KKv\right)$ consists of invertible $nr_v\times nr_v$ matrices with entries in $\DD_v$. We then let $K_v'$ be those matrices in $\GL_n'\left(\KKv\right)=\GL_{nd_v}(\DD_v)$ which have entries in $\OO_v'$. 
Denote for $v\in {\VV_\infty}$ by $\lig'_v$ the Lie algebra of $\GL'_{n}\left(\KKv\right)$ and by $\lig'_\infty,$ the Lie algebra of $\GL'_{n,\infty}\coloneqq \prod_{v\in \VV_\infty}\GL_n'(\KK_v)$.

To ensure that the periods we will consider in later sections are well defined, we will also have to fix a Haar measure on $H'_n\left(\KKv\right)$ for all $v\in {\VV_f}$. We do this by setting the volumes of the two copies of $K_v'$ in $H'_n\left(\KKv\right)$ with respect to the measures to $1$. Taking the product of those measures over all $v\in {\VV_f}$, we obtain a Haar measure $\df g_1\times \df g_2$ on $H'_n\left(\A_f\right)$. This in turn determines the volume
\[c\coloneqq \mathrm{vol}\left(Z'_{2n}\left(\KK\right)\bs Z'_{2n}\left(\A\right)/{\mathbb{R}^{r}_{> 0}}\right)=2^r\cdot \mathrm{vol}\left(\overbrace{\KK^\times\bs \A_f^\times\times\ldots\times\KK^\times\bs \A_f^\times}^{r}\right).\] Now for $v\in {\VV_\infty}$ let $\ddd_v{g_1}$ and $\ddd_v{g_2}$ be the Haar measures on the two copies of $\GL'_n\left(\KKv\right)$ such that $\mathrm{Sp}\left(\frac{nd}{2}\right)\subseteq \GL_\frac{nd}{2}\left(\HH\right)$ has volume $1$. Set \[\ddd_\infty g_{1}\coloneqq  c\prod_{v\in {\VV_\infty}}\ddd_vg_{1},\,
\ddd_\infty{g_2}\coloneqq \prod_{v\in {\VV_\infty}}\ddd_vg_{2},\]
which then gives a Haar measure $\ddd g_1\times \ddd g_2$ on $H'_n\left(\A\right)$ where $\ddd g_i\coloneqq \df g_i\ddd_\infty g_i, \,i=1,2.$
\subsection{ Discrete series}
Let us now fix our notations regarding automorphic representations, automorphic forms, and in particular, discrete series representations.
We call an irreducible $(\lig_\infty', K_\infty',\GL_n'(\A))$-subquotient $\Pi'$ of the space of automorphic forms $\mathcal{A}\left(\GL_n'\left(\KK\right)\bs \GL_n'\left(\A\right)\right)$ on $\GL_n'\left(\A\right)$ an (irreducible) automorphic representation of $\GL_n'\left(\A\right)$. We call $\Pi'$ \emph{cuspidal} if it is generated by a cusp form $\phi$, \emph{i.e.} an automorphic form $\phi$ such that
\[\int_{U\left(\KK\right)\bs U\left(\A\right)}\phi\left(ug\right)\ddd u=0,\]
for all $g\in \GL_n'(\A)$ and all non-trivial parabolic subgroups $P$ of $\GL_n'$ with Levi-decomposition $P=MU$.
Let $\omega\colon Z_n'\left(\KK\right)\bs Z_n'\left(\A\right)\rightarrow \CC^\times$ be a continous, unitary character. 
As in the introduction we denote the $L^2$-completion of the square-integrable functions on $\GL_n'\left(\KK\right)\bs \GL_n'\left(\A\right)$ with central character $\omega$ by\[L^2\left(\GL_n'\left(\KK\right)\bs \GL_n'\left(\A\right),\omega\right).\] This is a representation of $\GL_n'\left(\A\right)$ via the right regular action. If $\widetilde{\Pi}$ is an irreducible subrepresentation of $L^2\left(\GL_n'\left(\KK\right)\bs \GL_n'\left(\A\right),\omega\right)$, we will denote by $\widetilde{\Pi}^\infty$ the smooth vectors in $\widetilde{\Pi}$, \emph{cf}. \cite[Chapter 11]{Gro}. 
Moreover, the subspace of smooth, $K_\infty'$-finite vectors in $\widetilde{\Pi}$ carries the structure of a $\left(\lig_\infty', K_\infty',\GL_n'\left(\A\right)\right)$-module.
The automorphic representations which can be obtained in this way will be called \emph{discrete series representations} and every cuspidal representation is a discrete series representation. If it is clear from context, we will implicitly use the representation $\Pi'$ if we talk about $\left(\lig_\infty', K_\infty',\GL_n'\left(\A\right)\right)$-modules and the corresponding representation $\widetilde{\Pi}$ if we talk about $\GL_n'(\A)$-representations.

Coming with those two ways of looking at a discrete series representation $\Pi'$, we have two ways of writing it as a restricted tensor product, \emph{cf}. \cite[Chapter 14]{Gro}.
We again denote by $\widetilde{\Pi}$ the corresponding subrepresentation of $L^2\left(\GL_n'\left(\KK\right)\bs \GL_n'\left(\A\right),\omega\right)$. Then the smooth vectors $\widetilde{\Pi}^\infty$ admit a decomposition 
\[\widetilde{\Pi}^\infty\cong {\underset{v\in {\VV_\infty}}{\overline{\bigotimes}_\mathrm{pr}}}\widetilde{\pi}_v^\infty\otimes_\mathrm{in}\bigotimes_{v\in {\VV_f}}'\widetilde{\pi}_v^\infty,\]
where $\overline{\bigotimes}_\mathrm{pr}$ denotes taking the completed projective tensor product, $\otimes_\mathrm{in}$ denotes the inductive tensor product and $\widetilde{\pi}_v^\infty$ are $\GL_n'(\KKv)$-representations. For $v\in {\VV_\infty}$, taking $K_v'$-finite vectors gives a $(\lig_v', K_v')$-module $\Pi_v'$. 
This gives us a second decomposition
 $\Pi'\cong\bigotimes_{v\in \VV}'\Pi_v'$, which now is a restricted tensor product of $(\lig'_\infty,K_\infty')$- respectively $\GL_n'(\KKv)$-modules. 
Throughout the paper we will therefore mean $\left(\widetilde{\pi}_v\right)^\infty$ if we treat $\Pi_v'$ as a $\GL_{n}'\left(\KKv\right)$-representation.
We denote by $S_{\Pi'}\subseteq {\VV_f}$ the finite set of places where $\Pi'$ ramifies. The central character of $\Pi'$ will be denoted by $\omega=\omega_{\Pi'}$.
For $v\in {\VV_\infty}$ and $\left(\rho, W\right),\, \rho=\rho_1\otimes\ldots\otimes\rho_k$ an irreducible representation of $M_\alpha'\left(\KKv\right)$ we set \begin{equation}\begin{gathered}\label{E:ind}{}^a\mathrm{Ind}_{P_\alpha'}^{G_n'}\left(\rho\right)\coloneq
\{f\in C^\infty\left(G,W\right):f\left(mng\right)=\rho\left(m\right)f\left(g\right),\\ m\in M_\alpha'\left(\KKv\right),\, n\in U_\alpha'\left(\KKv\right),\, g\in  \GL_n'\left(\KKv\right)\}, \end{gathered}\end{equation} on which $\GL_n'\left(\KKv\right)$ acts by right translation. We equip ${}^a\mathrm{Ind}_{P_\alpha'}^{G_n'}$
with the subspace topology induced from the Fr\'echet space $C^\infty(\GL_n'(\KKv), W)$.
The space \[\mathrm{Ind}_{P_\alpha'}^{G_n'}\left(\rho\right)\coloneqq {}^a\mathrm{Ind}_{P_\alpha'}^{G_n'}\left(\rho\otimes\delta_{P_\alpha'}^{\frac{1}{2}}\right)\]
is called the normalized parabolically induced representation, where $\delta_{P_\alpha'}$ is the modular character of the group $P_\alpha'$.

If $(\Sigma, W)$ is a discrete series representation of $M_\alpha'(\A)$, with the corresponding $M_\alpha'(\A)$-representation $(\widetilde{\Sigma}, \widetilde{W})$ and $\mu$ a character of $P_\alpha'(\A)$ we define $\mathrm{Ind}_{P_\alpha'\left(\A\right)}^{\GL_n'\left(\A\right)}\left(\Sigma\otimes \mu \right)$ to be the space of smooth functions $f\colon \GL_n'(\A)\ra \widetilde{W}^\infty$ satisfying the normalized global analogue of the equivariance condition (\ref{E:ind}). The so-obtained space admits a natural topology with which the $\GL_n'(\A)$-action by right translations is continuous. It admits a decomposition
\[\mathrm{Ind}_{P_\alpha'\left(\A\right)}^{\GL_n'\left(\A\right)}\left(\Sigma\otimes \mu\right)\cong {\underset{v\in {\VV_\infty}}{\overline{\bigotimes}_\mathrm{pr}}}\mathrm{Ind}_{P_\alpha'}^{G_n'}\left((\Tilde{\Sigma}_v\otimes\mu_v)^\infty\right)\otimes_\mathrm{in} \bigotimes'_{v\in {\VV_f}} \mathrm{Ind}_{P_\alpha'}^{G_n'}\left(\Sigma_v\otimes\mu_v\right).\]
Similarly, we define for $\GL_n$ parabolic and normalized parabolic induction.

\subsubsection{ M{\oe}glin-Waldspurger classification} We also recall the following well-known description of discrete series representations known as the M{\oe}glin-Waldspurger classification.
\begin{theorem}[\cite{Bad},\cite{MoeWal}]\label{T:MWclas}
Let $k,l\in\mathbb{Z}_{\ge 1},\,n=lk$ and $\Sigma'$ be a cuspidal unitary automorphic representation of $\GL_l'\left(\A\right)$. Then the parabolically induced $\GL_n'(\A)$-representation \[\Sigma'\lvert\det'\lvert^\frac{k-1}{2}\times\ldots\times \Sigma'\lvert\det'\lvert^\frac{1-k}{2}\] admits a unique irreducible quotient, denoted by 
$\mw\left(\Sigma',k\right)$. It is a discrete series representation of $\GL_n'\left(\A\right)$ and moreover for every discrete series representation $\Pi'$ of $\GL_n'\left(\A\right)$, there exists $l,k$ and $ \Sigma'$ as above such that $\Pi'\cong \mw\left(\Sigma',k\right)$. The analogous statement for $\GL_n$ instead of $\GL_n'$ holds also true.
\end{theorem}
\subsection{ Jacquet-Langlands correspondence}\label{S:2.6}
We will now quickly recall the basic notions of the Jacquet-Langlands correspondence.
For a complete discussion see \cite{Bad}, \cite{BadRen}. Let $v\in \VV$ be a place and
recall that to each irreducible unitary representation of $\Pi_v$ of $\GL_{dn}\left(\KKv\right)$ respectively $\Pi_v'$ of $\GL'_n\left(\KKv\right)$ we can associate a trace character $\chi_{\Pi_v}$, respectively, $\chi_{\Pi_v'}$. We refer to \cite{DelKazVig} for the notion of a $d_v$-compatible representation of $\GL_n'(\KKv)$.
Let $U_{cp}'\left(\GL_{dn}\left(\KKv\right)\right)$ be the set of unitary $d_v$-compatible irreducible representations of $\GL_{dn}\left(\KKv\right)$ and let $U'\left(\GL'_n\left(\KKv\right)\right)$ be the set of unitary irreducible representations of $\GL'_n\left(\KKv\right)$. Moreover, let $U_{cp}\left(\GL_{dn}\left(\KKv\right)\right)$, respectively, $U\left(\GL'_n\left(\KKv\right)\right)$ be the set of representations of the form $\Pi\otimes \lvert\det\lvert^s$, respectively, $\Pi'\otimes \lvert\det'\lvert^s$ for $\Pi\in U_{cp}'\left(\GL_{dn}\left(\KKv\right)\right)$, respectively, $\Pi' \in U'\left(\GL'_n\left(\KKv\right)\right)$.
Then there exists a map called the \emph{local Jacquet-Langlands correspondence} \[\lj_v\colon U_{cp}\left(\GL_{dn}\left(\KKv\right)\right)\rightarrow U\left(\GL'_n\left(\KKv\right)\right)\] with the following properties, see \cite{DelKazVig}:
\begin{enumerate}
    \item If $\Pi_v=\widetilde{\pi}_v\otimes \lvert\det'\lvert^s$ with $\widetilde{\pi}_v$ a unitary $d_v$-compatible irreducible representations of $\GL_{dn}\left(\KKv\right)$, we have $\lj_v\left(\Pi_v\right)=\lj_v\left(\widetilde{\pi}_v\right)\otimes  \lvert\det'\lvert^s.$
    \item If $v$ is a split place of $\DD$, $\lj_v$ is the identity.
    \item $\lj_v$ restricted to square integrable representations is a bijection onto the square integrable representations of $\GL'_n\left(\KKv\right)$.
    \item $\lj_v$ commutes with parabolic induction.
\end{enumerate}
\label{S:propjl}Similarly, there is a global correspondence going from the unitary discrete series representations of $\GL'_n\left(\A\right)$ into the set of unitary discrete series representations of $\GL_{nd}\left(\A\right)$ which is denoted by $\jl$ and called the \emph{global Jacquet-Langlands correspondence}. It satisfies the following properties:
\begin{enumerate}
    \item $\lj_v\left(\left(\jl\left(\Pi'\right)\right)_v\right)\cong \Pi_v'$ for all $v\in \VV$.
    \item $\jl$ is injective.
    \item If $\jl\left(\Pi'\right)$ is cuspidal, then $\Pi'$ is cuspidal.
\end{enumerate}
Crucially, if $\Pi'$ is cuspidal $\jl\left(\Pi'\right)$ does not have to be cuspidal. 
\subsection{ Cohomological automorphic representation}
Let us fix our notations regarding relative Lie algebra cohomology and cohomological automorphic representation. For each irreducible $\left(\lig'_\infty,K_\infty'\right)$-module $\Pi_\infty'\cong\bigotimes_{v\in {\VV_\infty}}\Pi_v'$ we denote by $H^r\left(\lig'_\infty,K'_\infty,\Pi_\infty'\right)$ the $\gK$-cohomology of degree $r$ of $\Pi_\infty'$.
A $(\lig',K_\infty')$-module $\Pi_\infty'$ is called cohomological if there exists a highest weight representation $E_\mu$ of $\GL_n'(\RR)$ such that 
$H^r\left(\lig'_\infty, K'_\infty,\Pi_\infty'\otimes E_\mu\right)$
is nonzero for some $r$.
We call an automorphic representation $\Pi'\cong \Pi_\infty'\otimes\Pi'_f$ of $\GL_n'\left(\A\right)$ cohomological if its archimedean component $\Pi_\infty'$ is cohomological. The analogous definition can be made for $\GL_n$.
\subsubsection{ Godement-Jacquet global L-functions}
For $\Pi'$ a discrete series representation of $\GL_n'(\A)$, we define
\[L(s,\Pi')\coloneqq \prod_{v\in \VV}L(s,\Pi_v'),\]
which is well-defined for $\re>>0$ and admits an analytic continuation to a meromorphic function, \emph{cf.} \cite[Theorem 13.8]{GodJac}.
Moreover, if $\Pi'$ is cuspidal and not a unitary character of the form $\lvert\det\lvert^{it}$, the $L$-function $L(s,\Pi')$ is entire.
For $S$ a finite subset of $V$, we write $L^S(s,\Pi')\coloneqq \prod_{v\notin S}L(s,\Pi_v')$, respectively, for its analytic continuation. In particular, we set $L(s,\Pi'_f)\coloneqq  L^{{\VV_\infty}}(s,\Pi')$.

Let us also recall how the $L$-function behaves with respect to the Jacquet-Langlands correspondence, see \cite[§6]{Bad} and \cite[§19]{BadRen}. If $S$ does contain all places where $\DD$ splits, it follows immediately that $L^S(s,\Pi')=L^S(s,\jl(\Pi'))$ for any cuspidal representation $\Pi'$. Moreover, if $\Pi_v$ is an irreducible local discrete series representation of $\GL_{nd}(\KKv)$, then also  $L(s,\Pi_v)=L(s,\lj(\Pi_v))$. 
\section{ Cohomological unitary dual}
We start by recalling the classification of the cohomological irreducible unitary dual of $\GL_n\left(\HH\right)$ due to \cite{VogZuc} and explicitly described in \cite{GroRag2}. Let $\mathfrak{g}'$ be the Lie algebra of $\GL_n\left(\HH\right)$ and let $\mathfrak{k}'$ be the Lie algebra of $\Sp(n)$, which determines a Cartan involution $\theta'\left(X\right)=-\Bar{X}^T$ of $\lig'$. Moreover, let $\mathfrak{h}'$ be a maximal compact, $\theta'$-stable Cartan-algebra $\mathfrak{h}'=\mathfrak{a}'\oplus\mathfrak{t}'$, with
\[\mathfrak{t}'=\left\{\bpm ix_1&&0\\&\ddots&\\0&&ix_n\epm: x_j\in\mathbb{R}\right\}\text{ and }\mathfrak{a}'=\left\{\bpm y_1&&0\\&\ddots&\\0&&y_n\epm: y_j\in\mathbb{R}\right\}.\]
Furthermore, let $E_\lambda$ be a highest weight representation of $\GL_n\left(\HH\right)$, where $\lambda$ is a highest weight with respect to the subalgebra $\mathfrak{h}'_\CC$. 
To each composition $n=\sum_{i=0}^r n_i$ written as \[\underline{n}=[n_0,\ldots,n_r]\] with $n_0\ge 0$ and $n_i>0$ we can associate a $\theta'$-stable, parabolic subalgebra $\mathfrak{q}'_{\underline{n}}$ of $\lig'_\CC$ whose Levi-decomposition we will denote as $\mathfrak{q}'_{\underline{n}}=\mathfrak{l}'_{\underline{n}}+\mathfrak{u}'_{\underline{n}}$, \emph{cf.} \cite[Section 4]{GroRag2} for more details.
We further assume that $\restr{\lambda}{\mathfrak{a}'}=0$ and that $\lambda$ can be extended to an admissible character of $ \mathfrak{l}'_{\underline{n}}\supseteq \mathfrak{h}'_\CC$.
\begin{theorem}[{\cite[Theorem 4.9]{GroRag2}}]\label{T:dimcoh}
Let $E_\lambda$ be a self-dual highest weight representation of $\GL_n\left(\HH\right)$.
\begin{enumerate}
    \item To each ordered composition $\underline{n}=[n_0,\ldots,n_r]$ of $n$ with $n_0\ge 0, n_i>0$ one can assign an irreducible unitary representation $A_{\underline{n}}\left(\lambda\right)$ of $\GL_n\left(\HH\right)$.
    \item All such representations are cohomological with respect to $E_\lambda$ and every cohomological representation is of this form.
    \item 
    The Poincar\'e polynomial of $H^*\left(\lig',\mathrm{Sp}(n)\RR_{\ge 0}, E_\lambda\otimes A_{\underline{n}}\left(\lambda\right)\right)$ is
    $$P\left(\underline{n},X\right)=\frac{X^{\dim_\CC\left(\mathfrak{g}_\CC^-\cap\mathfrak{u}'_{\underline{n}}\right)}}{ 1+X}\prod_{i=1}^r\prod_{j=1}^{n_i}\left(1+X^{2j-1}\right)\prod_{j=1}^{n_0}\left(1+X^{4j-3}\right).$$ Here $\mathfrak{g}_\CC^-$ is the $-1$-eigenspace of $\theta'$ acting on $\lig'_\CC$.
\end{enumerate}
\end{theorem}
For later use, we compute the following.
\begin{lemma}\label{L:explicitform}
Let $\underline{n}=[n_0,n_1,\ldots,n_r]$ be a composition of $n$. Then
    \[\dim_\CC(\mathfrak{g}_\CC^-\cap\mathfrak{u}'_{\underline{n}})=\sum_{i=1}^r\binom{n_i}{2}+2\sum_{0\le i<j\le r}n_in_j.\]
\end{lemma}
\begin{proof}
    We first recall the definition of $\mathfrak{u}'_{\underline{n}}$, \emph{cf.} \cite[§4.2]{GroRag2}. Let \[x=\mathrm{diag}(\underbrace{0,\ldots,0}_{n_0},\underbrace{1,\ldots,1}_{n_1},\ldots,\underbrace{r\ldots,r}_{n_r})\in i\mathfrak{t}'\] and let $\Delta(\lig'_\CC,\mathfrak{t}'_\CC)$, respectively, $\Delta(\mathfrak{g}_\CC^-,\mathfrak{t}'_\CC)$ be the set of roots coming from $\mathfrak{t}'_\CC$. We have the explicit description
    $\Delta(\mathfrak{g}_\CC^-,\mathfrak{t}'_\CC)=\{\pm e_i\pm e_j, 1\le i<j\le n\},$
    where $e_j(H)=ix_j$ for $H=\diag(ix_1+y_1,\ldots,ix_n+y_n))\in \mathfrak{h}'$.
    Moreover,
    \[\mathfrak{u}'_{\underline{n}}=\bigoplus_{\ato{\alpha\in \Delta(\lig'_\CC,\mathfrak{t}'_\CC)}{ \alpha(x)>0}}(\lig'_\CC)_\alpha\]
    and therefore \[\mathfrak{u}'_{\underline{n}}\cap\mathfrak{g}_\CC^- =\bigoplus_{\ato{\alpha\in \Delta(\mathfrak{g}_\CC^-,\mathfrak{t}'_\CC)}{ \alpha(x)>0}}(\lig'_\CC)_\alpha.\]
    Hence $ \dim_\CC(\mathfrak{g}_\CC^-\cap\mathfrak{u}'_{\underline{n}})=\#\{\alpha\in \Delta(\mathfrak{g}_\CC^-,\mathfrak{t}'_\CC),\, \alpha(x)>0\},$
    which is easily seen to be equal to the above explicit formula.
\end{proof}
Our next step is to showing that if $\Sigma$ is a cuspidal irreducible representation of $\GL_{n}'(\A)$ and $k\in \mathbb{N}$, then $\Sigma$ is cohomological if $\mw(\Sigma,k)$ is. The author would like to thank Harald Grobner for pointing out the argument presented here. Before we start, we need to recall the following theorem.
    \begin{theorem}[{\cite[Theorem 1.8]{Sal}}]\label{T:sal}
        Let $G$ be a connected, semisimple real Lie group with finite center and Lie algebra $\lig$. Fix a maximal connected subgroup $K$ of $G$ with Lie algebra $\lik$ and moreover, let $\pi$ be an irreducible unitary smooth representation of $G$ with central character $\chi_\pi$. Finally, let $U$ be a finite-dimensional $(\lig,K)$-module admitting an infinitesimal character $\chi_{U}=\chi_{\pi^\lor}$. Then
        \[H^*(\lig,\lik,\pi\otimes U)\neq 0.\]
    \end{theorem}
We denote for a real Lie group $G$ by $Z_G$ its center and by $Z_G^0$ the connected component of the latter.
\begin{lemma}\label{L:sal2}
Let $\underline{G}$ be a connected reductive group over $\KK$, $v\in {\VV_\infty}$ and $G\coloneqq \underline{G}(\KKv)$. Let $\pi$ be an irreducible unitary representation of $G$ and $E_\lambda$ a finite dimensional highest weight representation of $G$ over $\CC$ such that $Z_G^0$, acts trivially on $E_\lambda\otimes \pi$ and $\chi_{E_\lambda}=\chi_{\pi^\lor}$.
Then \[H^*(\lig,(Z_G \cdot K)^0,\pi\otimes E_\lambda)\neq 0.\]
\end{lemma}
\begin{proof}
Note that $E_\lambda$ always admits a central character.
Recall that \[H^*(\lig,(Z_G \cdot K)^0,\pi\otimes E_\lambda)=H^*(\lig,\mathfrak{z}_G \oplus \lik,\pi\otimes E_\lambda),\] where $\mathfrak{z}_G$ is the Lie algebra of $Z_G$ and we use that $K$ has finite center. Since $\mathfrak{z}_G$ acts trivially on $\pi\otimes E_\lambda$, the Künneth formula gives a decomposition  \[H^*(\lig,\mathfrak{z}_G \oplus \lik,\pi\otimes E_\lambda)\cong \bigotimes_{a+b=*}H^a(\mathfrak{z}_G,\mathfrak{z}_G,\CC)\otimes H^b(\lig/ \mathfrak{z}_G,\lik,\pi\otimes E_\lambda)=\]\[=H^*(\lig/ \mathfrak{z}_G,\lik,\pi\otimes E_\lambda).\]The latter does not vanish by \ref{T:sal}, since both $\pi$ and $E_\lambda$ admit the right central characters and the image of $\lig/ \mathfrak{z}_G$ under the exponential map generates  the connected, semisimple real Lie group $G/Z_G$ .
\end{proof}
    Let  $\underline{G}$ be either $\GL_n$ or $\GL_n'$, $v\in {\VV_\infty}$, $G=\underline{G}(\KKv)$ and $K=K_v$ or $K_v'$. Moreover, let $\underline{P}$ be a standard parabolic subgroup of $\underline{G}$ and set $P=\underline{P}(\KKv)=L\rtimes U$.  Write  $L=M\times A^0$, where $A^0=Z_L^0$. Next, let $(\pi,V)$ be an irreducible, unitary representation of $L$. Denote now by $\mathfrak{b}_\CC$ the complexified Cartan subalgebra of the Lie algebra of $M$ coming from our fixed choice of Cartan subalgebra of $L$, \emph{i.e.} the diagonal matrices if $\underline{G}=\GL_n$ or $\mathfrak{h}'$ if $\underline{G}=\GL_n'$. Let $\mathfrak{a}_{P,\CC}^\lor$ be the complexified dual of the Lie-algebra of $A^0$ and fix $\mu\in \mathfrak{a}_{P,\CC}^\lor$. 
    We let $\mathfrak{p}_\CC$ be the complexified Lie algebra of $P$ and let $\rho$ be the half-sum of all positive roots of $\mathfrak{p}_\CC$ with respect to our fixed Cartan subalgebra. 
    Denote by $\Delta_M$ the simple roots of $M$, $W$ the Weyl-group of $G$ and
$W^P=\{w\in W:w^{-1}(\alpha)>0\text{ for all }\alpha\in \Delta_M\}.$
    We write
\[\mathrm{Ind}_{P}^G(\pi,\mu)=\{f\colon G\ra V\text{ smooth}:f(maug)=a^{\rho+\mu}\pi(m)f(g),\]\[ a\in A^0,\, m\in M,\, u\in U,\, g\in G\}\] and use the standard parametrization of infinitesimal characters, i.e. for a highest weight representation $E_\lambda$, $\chi_{\lambda+\rho}=\chi_{E_\lambda}$. 
\begin{prop}\label{T:sal3}
    If $\tau$ is a non-zero $(\lig,(Z_G \cdot K)^0)$-module, which appears as a quotient of $\mathrm{Ind}_{P}^G(\pi,\mu)$ and is cohomolocigal with respect to some highest weight representation $E_\lambda^\lor$, then $\pi$ is cohomological as a $(\mathfrak{l}, (Z_L\cdot (L\cap K)^0)$-module with respect to $E_{{w(\lambda+\rho)-\rho}}^\lor,$ where $w$ is some element of $W^P$.
\end{prop}
\begin{proof}
We notice that without loss of generality $A^0$ acts trivially on $\pi$.
Moreover, if $\chi_\pi$ denotes the infinitesimal character of $\pi$, $\mathrm{Ind}_{P}^G(\pi,\mu)$ and hence also $\tau$ have infinitesimal character $\chi_{\pi+\mu}$. On the other hand, $\tau$ is by assumption cohomological with respect to $E_\lambda^\lor$ and hence it has to have infinitesimal character $\chi_{{\lambda+\rho}}$ by \cite[Theorem I.5.3]{BorWal}.
 Therefore the infinitesimal character of $\restr{\pi}{M}$ is equal to
$\chi_{\restr{\lambda+\rho-\mu}{\mathfrak{b}_\CC}}$ and hence $\restr{\pi}{M}$ has non-vanishing cohomology with respect to $E_{\restr{w(\lambda+\rho)-\rho}{\mathfrak{b}_\CC}}^\lor$ by \ref{L:sal2}.
Now for any Konstant-representative $w\in W^P$
$\chi_{\restr{\lambda+\rho-\mu}{\mathfrak{b}_\CC}}=\chi_{\restr{w(\lambda+\rho)}{\mathfrak{b}_\CC}}.$ Note now that ${\restr{w(\lambda+\rho)-\rho}{\mathfrak{b}_\CC}}$ is a dominant weight, see \cite[III.3.2]{BorWal}, and hence the last character is equal to
$\chi_{E_{\restr{w(\lambda+\rho)-\rho}{\mathfrak{b}_\CC}}^\lor}.$
We can choose now $w$ as in \cite[III. Theorem 3.3]{BorWal} such that $Z_L^0=A^0$ acts trivially on $\pi\otimes E_{{w(\lambda+\rho)-\rho}}^\lor.$Hence $\pi$ is by \ref{L:sal2} cohomological with respect to $E_{{w(\lambda+\rho)-\rho}}^\lor$. 
\end{proof}
\begin{corollary}\label{C:conj}
Let $\Sigma$ be a cuspidal irreducible representation of $\GL_n(\A)$ or $\GL_n'(\A)$ and $k\in\mathbb{N}$. Then $\Sigma$ is cohomological if $\mw(\Sigma,k)$ is cohomological.
\end{corollary}
\begin{proof}
    Since $\mw(\Sigma,k)_v^\infty$ is the quotient of \[(\Sigma_v^\infty\lvert\det'\lvert^\frac{k-1}{2}\times\ldots\times \Sigma_v^\infty\lvert\det'\lvert^\frac{1-k}{2})_v\] for all $v\in {\VV_\infty}$, the claim follows from \ref{T:sal3}, because \[K_v'=\mathrm{Sp}(\frac{nd}{2})\RR_{>0}=(\mathrm{Sp}(\frac{nd}{2})\RR)^{0},\, K_v=\mathrm{SO}(n)\RR_{>0}=(\mathrm{O}(n)\RR)^{0}.\]
\end{proof}
\begin{lemma}\label{L:littlelem}
Assume $\Pi'$ is a cuspidal irreducible cohomological representation of $\GL_{2n}'\left(\A\right)$ such that $\jl\left(\Pi'\right)$ is not a cuspidal representation of $\GL_{2dn}\left(\A\right)$. Let $2dn=kl$ and let $\Sigma$ be a unitary cuspidal irreducible representation of $\GL_l\left(\A\right)$ such that 
$\jl\left(\Pi'\right)=\mw\left(\Sigma,k\right)$.

Then $l$ is even and $\Sigma_v$ is cohomological with respect to some highest weight representation $E_{\lambda_v}$, $\lambda_v=\left(\lambda_{v,1},\ldots,\lambda_{v,\frac{l}{2}}\right)$.
 For each $v\in {\VV_\infty}$, $\Pi_v'$ is of the form $\Pi_v'=A_{\underline{nd}}\left(\lambda_v'\right)$ for \[\underline{nd}=[0,\overbrace{k,\ldots,k}^\frac{l}{2}]\]
 and $\lambda_{v,\frac{l}{2}}=\lambda_{v,nd}'$. In particular, the lowest, respectively, highest degree in which the cohomology group $H^{q}(\lig_v',\mathrm{Sp}(nd)\RR_{\ge 0}, \Pi_v'\otimes E_{\lambda_v'})$ does not vanish is
\[q=nd\left(nd-1\right)-\frac{nd}{2}\left(k-1\right)\text{, respectively, } q=nd\left(nd-1\right)+\frac{nd}{2}\left(k+1\right)-1.\]
\end{lemma}
\begin{proof}
Fix an infinite place $v\in {\VV_\infty}$. By \cite[Theorem 5.2]{GroRag2} $\mw\left(\Sigma,k\right)$ is cohomological and thus by \ref{C:conj} so is $\Sigma$. By \cite[Theorem 18.2]{Bad}, $k\lvert d$ and hence, $l$ has to be even.
Since the archimedean component of a cohomological cuspidal irreducible unitary representation of $\GL_l\left(\A\right)$ must be tempered we may write $\Sigma_v=A_{[0,2,\ldots,2]}\left(\lambda_v\right)$ and let $\Pi_v'=A_{\underline{nd}}\left(\lambda_v'\right)$ for suitable $\underline{nd}$ and $\lambda_v$, $\lambda_v'$, with $\underline{nd}=[n_0,n_1,\ldots,n_{l'}]$, see \cite[Section 5.5]{GroRag2}.
Furthermore, $\Sigma_v$ is fully induced from representations of $\GL_2(\RR)$.
To proceed with the proof, we are quickly going to recap the construction of $A_{\underline{nd}}\left(\lambda_v'\right)$ in the proof of \cite[Theorem 5.2]{GroRag2}. 
Let \[\rho_{\mathfrak{gl}_{{m}}\left(\HH\right)}=\left(\frac{2m-1}{2},\ldots,-\frac{2m-1}{2}\right),\] respectively,
\[\rho_{\mathfrak{gl}_{{m}}\left(\CC\right)}=\left(\left(\frac{m-1}{2},\ldots,-\frac{m-1}{2}\right),\left(\frac{m-1}{2},\ldots,-\frac{m-1}{2}\right)\right)\]
the smallest algebraically integral element in the interior of the dominant Weyl chamber of $\GL_{m}\left(\HH\right)$, respectively, $\GL_{m}\left(\CC\right)$.
Define now \[\mu\coloneqq \left(\rho_{\mathfrak{gl}_{n_0}\left(\HH\right)}, \rho_{\mathfrak{gl}_{n_1}\left(\CC\right)},\ldots,\rho_{\mathfrak{gl}_{n_{l'}}\left(\CC\right)}\right)\]
and let $P'$ be a certain complex parabolic subgroup of $\GL_{dn}\left(\HH\right)$, which we will specify in a moment, and having Levi-factor
$\prod_{i=1}^{nd}\GL_1(\HH)$.
For any integer $s> 0$ and $u\in \CC$ we set $D\left(u,s\right)\coloneqq  D\left(s\right)\otimes \lvert\det\lvert^{-\frac{u}{2}},$ where $D\left(s\right)$ is the unique irreducible discrete series representation of $\mathrm{SL}_2^{\pm}\left(\RR\right)$ of lowest $\mathrm{O}\left(2\right)$-type $s+1$.
We also set \[F\left(u,s\right)\coloneqq  F\left(s\right)\otimes \lvert\det'\lvert^{-\frac{u}{2}},\] where $F\left(s\right)$ is the unique irreducible representation of $\mathrm{SL}_1\left(\HH\right)$ of dimension $s$. 
Moreover, recall the Levi decomposition ${\mathfrak{q}'_{\underline{nd}}}=\mathfrak{l}'_{\underline{nd}}+\mathfrak{u}'_{\underline{nd}}$ and 
let the weight $\rho\left(\underline{nd}\right)=\left(\rho\left(\underline{nd}\right)_1,\ldots,\rho\left(\underline{nd}\right)_{nd}\right)$ 
be the half-sum of all roots appearing in $\mathfrak{u}'_{\underline{nd}}$. Let $k_i=\lambda_i'+\rho\left(\underline{nd}\right)_i$.
We set \[\sigma=\bigotimes_{i=1}^{nd-n_0} F\left(0,k_i\right).\] 
Then $P'$ can be chosen such that $(P',\sigma, \mu)$ is a Langlands-datum and $A_{\underline{nd}}\left(\lambda_v'\right)$ is the unique irreducible quotient of the induced representation
$\mathrm{Ind}_{P'}^{\GL'_{2n}(\KKv)}(\sigma,\mu)$.
Since $\Sigma_v\lvert\det'\lvert^{\frac{k+1}{2}-j}$ is essentially tempered for every $v\in {\VV_\infty}$ and $\ain{j}{1}{k}$ and $\Sigma_v$ is cohomological, \[\Sigma_v\lvert\det'\lvert^{\frac{k+1}{2}-j}\cong\mathrm{Ind}_{P_j}^{\GL_{l}(\KKv)}(\sigma_j)\]
by \cite[Section 5.5]{GroRag2}, where$\sigma_j=\bigotimes_{i=1}^\frac{l}{2} D\left(2j-k-1,k_{i,j}\right)$
for certain $k_{i,j}\in\mathbb{Z}_{>0}$ and $P_j$ is the standard parabolic subgroup of upper triangular matrices with block size $\overbrace{\left(2,\ldots,2\right)}^\frac{l}{2}$.
Recall furthermore \[A_{\underline{nd}}\left(\lambda_v'\right)=\Pi_v'=\lj_v\left(\left(\jl\left(\Pi'\right)\right)_v\right)=\lj_v\left(\mw\left(\Sigma_v,k\right)\right).\]
By \cite[Theorem 5.2]{GroRag2} and its proof the last term is equal to the Langlands quotient of \[\mathrm{Ind}_{P}^{\GL_{nd}(\HH)}\left(\bigotimes_{j=1}^k\bigotimes_{i=1}^\frac{l}{2} F\left(0,k_{i,j}\right),\mu'\right),\]
where now $P$ is the standard parabolic subgroup of type $\overbrace{(1,\ldots,1)}^{nd}$ of $\GL_{nd}(\HH)$
and \[\mu'=\left(\overbrace{\frac{k-1}{2},\ldots,\frac{k-1}{2}}^{\frac{l}{2}},\ldots,\overbrace{-\frac{k-1}{2},\ldots,-\frac{k-1}{2}}^{\frac{l}{2}}   \right).\]
Comparing $\mu$ and $\mu'$ and
using the uniqueness of the Langlands quotient implies then that $n_0=0$ and $n_1=\ldots=n_{l'}=k$. Moreover, by  we have $\lambda_{v,\frac{l}{2}}+1=k_{\frac{l}{2},k}=k_{nd}=\lambda_{v,nd}'+1$.
From \ref{L:explicitform} we obtain
\[\dim_\CC\left(\mathfrak{g}_\CC^-\cap\mathfrak{u}'_{\underline{nd}}\right)=2\binom{nd}{ 2}-\frac{l}{2}\binom{k}{ 2}=nd\left(nd-1\right)-\frac{nd}{2}\left(k-1\right).\]
By \ref{T:dimcoh} the lowest degree of non-vanishing cohomology is $nd\left(nd-1\right)-\frac{nd}{2}\left(k-1\right)$ and the highest degree of non-vanishing cohomology is
\[\dim_\CC\left(\mathfrak{g}_\CC^-\cap\mathfrak{u}'_{\underline{nd}}\right)-1+\sum_{j=1}^\frac{l}{2}\sum_{i=1}^k \left(2i-1\right)=nd\left(nd-1\right)+\frac{nd}{2}\left(k+1\right)-1.\]
\end{proof}
\begin{remark}
To extend the ideas of \cite{GroRag} to the case $\GL'_{2n}\left(\A\right)$ we need the following numerical coincidence. Namely, it will be necessary that either the lowest or highest degree in which the cohomology group
$H^*\left(\lig_v',K_v', \Pi_v \otimes E_{\lambda_v}^\lor\right)$
does not vanish is \[q_0=\left(nd\right)^2-\left(nd\right)-1=\dim_\RR \left(H'_\frac{nd}{2}\left(\RR\right)/\left(\mathrm{Sp}\left(\frac{nd}{2}\right)\times \mathrm{Sp}\left(\frac{nd}{2}\right)\times \RR_{> 0}\right)\right).\] By \ref{L:littlelem} the only possible value for $n$ is therefore $n=1$, $d=2$ and the composition $\underline{2n}$ of $2n=2$ has to be $\underline{2n}=[0,2]$.
\end{remark}
\section{ Rational structures} Next we will recall the action of $\Aut(\CC)$ on representations.
Let $\Pi_f'$ be a representation of $\GL'_n\left(\A_f\right)$ on some complex vector space $W$ and $\sigma\in\Aut\left(\CC\right)$. We define the $\sigma$-twist ${}^\sigma\Pi'_f$ as follows, \emph{cf.} \cite{WalI}. Let $W'$ be a complex vector space which allows a $\sigma$-linear isomorphism $t\colon W'\rightarrow W$. We then set ${}^\sigma\Pi'_f\coloneqq t^{-1}\circ\Pi'_f\circ t.$ An explicit example of such a space $W'$ is the space 
$W'=W\otimes_\CC{} _\sigma\CC,$
where $_\sigma\CC$ is $\CC$ as a field but $\CC$ acts on $_\sigma\CC$ via $\sigma^{-1}$. Then 
$W'$ is a $\CC$ vector space via the right action of $\CC$ on $\CC$ and the map $t\colon W\rightarrow W\otimes_\CC {}_\sigma\CC$ is given by $w\mapsto w\otimes 1$. Similarly, we define the $\sigma$-twist $^\sigma\Pi_v'$ of a local representation $\Pi_v'$ with $v\in {\VV_f}$.
For a highest weight representation $E_\mu$ of $\GL'_{n,\infty}$, we define $\left({}^\sigma E_\mu\right)_v\coloneqq \left(E_\mu\right)_{\sigma^{-1}\circ v},$ where $v$ is seen as an embedding $\KK\hookrightarrow\CC$ and hence, $\sigma^{-1}\circ v$ defines an infinite place of $\KK$.
For $\Pi_f'$ as above let $$\mathfrak{S}\left(\Pi'_f\right)\coloneqq \{\sigma\in\Aut\left(\CC\right): {{}^\sigma\Pi_f'}\cong \Pi_f'\}$$
and let
$$\mathbb{Q}\left(\Pi_f'\right)\coloneqq \{z\in\CC:\sigma\left(z\right)=z,\, \text{for all}\, \sigma\in\mathfrak{S}\left(\Pi_f'\right)\}$$
be the rationality field of $\Pi_f'$. Analogously we define for a highest weight representation $E_\mu$ and a local representation $\Pi_v'$ the fields $\QQ\left(E_\mu\right)$ and $\QQ\left(\Pi_v'\right)$.
Moreover, if $\alpha=\left(\alpha_1,\ldots,\alpha_k\right)$ is a composition of $n$ and $\rho=\rho_1\otimes\ldots\otimes\rho_k$ an irreducible representation of $M_\alpha'\left(\KKv\right)$, $v\in {\VV_f}$, \[{}^\sigma{}^a\mathrm{Ind}_{P_\alpha'\left(\KKv\right)}^{\GL_n\left(\KKv\right)}\left(\rho\right)={}^a\mathrm{Ind}_{P_\alpha'\left(\KKv\right)}^{\GL_n\left(\KKv\right)}\left({}^\sigma\rho\right), \, \sigma\in\Aut\left(\CC\right)\]
and therefore \begin{equation}\label{E:E2}{}^\sigma\left(\rho_1\times\ldots\times\rho_k\right)=\left({}^\sigma\rho_1\times\ldots\times{}^\sigma\rho_k\right)\epsilon_\sigma^{dn-1},\, \epsilon_\sigma\coloneqq \frac{\lvert\det'\lvert^\frac{1}{2}}{\sigma\left(\lvert\det'\lvert^\frac{1}{2}\right) }\end{equation}
and similarly for the split case $\GL_n$.

Finally, we say that the representation $\Pi'_f,\, \Pi_v'$ or $E_\mu$ with underlying vector space $W$ is defined over some field $\mathbb{L}\subseteq \CC$ if there exists an $\LL$-vector space $W_\LL\subseteq W$, stable under the group action of $\GL_n'(\A_f),\, \GL_n'(\KKv)$, respectively, $\prod_{v\in {\VV_\infty}}\GL_n'(\KK)$, such that the natural map $W_\LL\otimes_\LL\CC\rightarrow W$ is an isomorphism. In this case we say $W$ admits an $\LL$-structure. Let $E_\mu$ be a highest weight representation and let $\mathbb{L}$ be a minimal field extension of $\KK$ such that $\DD$ splits over $\mathbb{L}$. Then $E_\mu$ is defined over $\LL$, see \cite[Lemma 7.1]{GroRag2} and we set
\[\QQ\left(\mu\right)\coloneq\mathbb{L}\cdot \QQ\left(E_\mu\right).\]
\begin{lemma}[{\cite[Proposition 3.2]{CloI}}]\label{L:clozelstructure}
    Let $v\in {\VV_f}$ and $\Pi_v'$ an irreducible representation of $\GL_n'(\KKv)$. Then $\Pi_v'$ admits an $\QQ(\Pi_v')$-structure.
\end{lemma}
Note that in the reference the lemma is only proven in the case $\GL_n$. However, the proof carries over analogously, since the Langlands classification via multisegments used in it is also valid for $\GL_n'$.
\begin{theorem}[{\cite[Theorem 8.1, Proposition 8.2, Theorem 8.6]{GroRag2}} ]\label{T:rat}
Let $\Pi'$ be a cuspidal irreducible representation of $\GL_n'\left(\A\right)$ and let $\mu$ be a highest weight such that $\Pi'$ is cohomological with respect to $E_\mu$. Then $\Pi_f'$ is defined over the number field \[\QQ\left(\Pi'\right)\coloneq\QQ\left(\mu\right)\QQ\left(\Pi_f'\right).\] Moreover, let $S\subseteq \VV$ be a finite set containing all places where $\Pi_f'$ ramifies. Then $\QQ\left(\Pi_f'\right)$ is the compositum of the number fields $\QQ\left(\Pi_v'\right), v\in {\VV_f}-S$.
\end{theorem}
We also have the following theorem by the same authors.
\begin{theorem}[{\cite[Proposition 7.21]{GroRag2}  }]\label{T:twistiscoh}
Let $\Pi'$ be a cuspidal irreducible representation of $\GL_n'\left(\A\right)$ and let $\mu$ be a highest weight such that $\Pi'$ is cohomological with respect to $E_\mu$. Then for all $\sigma\in\Aut\left(\CC\right)$ the representation ${}^\sigma\Pi_f'$ is the finite part of a discrete series representation ${}^\sigma\Pi'$ of $\GL'_n\left(\A\right)$ which is cohomological with respect to ${}^\sigma E_\mu$. 
Moreover, if $E_\mu$ is regular, ${}^\sigma\Pi$ is cuspidal.
\end{theorem}
\begin{definition}\label{D:oc}
We say the $\Aut\left(\CC\right)$-orbit of an cuspidal irreducible representation $\Pi'$ of either $\GL_n'\left(\A\right)$ or $\GL_n\left(\A\right)$ is cuspidal cohomological if ${}^{\sigma}\Pi'$ is cuspidal and cohomological for all $\sigma\in\Aut\left(\CC\right)$.
\end{definition}
We will now show that the regularity condition on $E_\mu$ is not needed.
\begin{prop}\label{P:cusp}
Let $\Pi'$ be a cuspidal irreducible cohomological representation of $\GL_n'\left(\A\right)$. Then $^\sigma\Pi'$ is cuspidal for all $\sigma\in \Aut\left(\CC\right)$. Moreover, ${}^\sigma\jl\left(\Pi'\right)=\jl\left({}^\sigma\Pi'\right)$
for all $\sigma\in\Aut\left(\CC\right)$. 
\end{prop}
\begin{proof}
If $\Pi\coloneq\jl\left(\Pi'\right)$ is cuspidal, the two claims are proven in \cite[Theorem 7.30]{GroRag2}. More precisely, they are proven under the assumption that $\Pi$ is so-called \emph{regular algebraic}, which by \cite[Lemma 3.14]{CloI} is equivalent to $\Pi$ being cohomological. Since $\jl$ sends cohomological representations to cohomological representations, this shows the first claims.

If $\Pi$ is not cuspidal, it is still a discrete series and we can write $\Pi=\mw\left(\Sigma,k\right)$ for some $k \ge 1$ by \ref{T:MWclas}. Let $\sigma\in\Aut\left(\CC\right)$. We will proceed by showing that $^\sigma\Pi'$ is cuspidal by induction on the $\KK$-rank of $\GL'_n$. If $n=1$, we already know that ${}^\sigma \Pi '$ is cuspidal. For $n>1$, let $\Theta,\Sigma',s$ and $t$ be such that \[\mw\left(\Theta,s\right)= {}^\sigma\Pi',\, \mw\left(\Sigma',t\right)= \jl\left(\Theta\right)\] and hence, by \cite[Theorem 18.2]{BadRen} $\jl\left({}^\sigma\Pi'\right)=\mw\left(\Sigma',st\right).$
  Note that \begin{equation}\label{E:E1}\begin{gathered}{}^\sigma \mw\left(\Sigma,k\right)_{\VV\setminus\VV_\DD}{=}{}^\sigma\jl\left(\Pi'\right)_{\VV\setminus\VV_\DD}\stackrel{(\ref{S:2.6}.(2))}{=}{}^\sigma\left(\Pi'_{\VV\setminus\VV_\DD}\right)\stackrel{(\ref{S:propjl})}{=}\\=\left({}^\sigma\Pi'\right)_{\VV\setminus\VV_\DD}=\stackrel{(\ref{S:2.6}.(2))}{=}\jl\left({}^\sigma\Pi'\right)_{\VV\setminus\VV_\DD}=\mw\left(\Sigma',st\right)_{\VV\setminus\VV_\DD}.\end{gathered}\end{equation}
 We will need the following intermediate lemma.
\begin{lemma}\label{L:twMW}
We have ${}^\sigma \mw\left(\Sigma,k\right)= \mw\left({}^\sigma \Sigma\chi_\sigma,k\right)$ for some quadratic character $\chi_\sigma$ with $^{\sigma^{-1}}\chi_\sigma=\chi_{\sigma^{-1}}^{-1}$.
\end{lemma}
\begin{proof} 
Let $v\in \VV\setminus \VV_\DD$ be a place where $\mw(\Sigma,k)$ and $\Sigma$ are unramified. By the Bernstein-Zelevinsky classification, see \cite{ZelII}, we can write $\Sigma_v=\langle\fm\rangle$ and $\mw(\Sigma,k)_v=\langle\fm'\rangle$ for some multisegments $\fm$ and $\fm'$ both consisting only of segments of length $1$ and with unramified cuspidal support. We will write from now on denote a finite length representation $\pi$ its cosocle, \emph{i.e.} its maximal, semi-simple quotient, by $\cos(\pi)$.
Moreover, $\fm$ and $\fm'$ determine each other and
\[\cos(\langle \fm\rangle\lvert\det\lvert^{\frac{k-1}{2}}\times\ldots\times \langle\fm\rangle\lvert\det\lvert^{\frac{1-k}{2}})= \langle\fm'\rangle.\]
by \ref{T:MWclas}. 
Applying \cite[Lemma 3.5(ii)]{CloI} both to $\langle\fm\rangle$ and $\langle\fm'\rangle$ yields
\[\cos({}^\sigma\langle\fm\rangle\lvert\det\lvert^{\frac{k-1}{2}}\chi_\sigma\times\ldots\times {}^\sigma\langle\fm\rangle\lvert\det\lvert^{\frac{1-k}{2}}\chi_\sigma)={}^\sigma\langle\fm'\rangle,\]
where $\chi_\sigma$ is $\epsilon_\sigma$ if both $k$ and $dn$ are odd and the trivial character otherwise.  
By \cite[Lemma 1]{LanV} $\mw\left({}^\sigma \Sigma\epsilon_\sigma,k\right)_v$ has to be the unique constituent of \[(\lvert\det\lvert^\frac{k-1}{2}\Sigma\chi_\sigma\times\ldots\times \lvert\det\lvert^\frac{1-k}{2}\Sigma\chi_\sigma)_v\]
with a $K_v'$-fixed vector for almost all places $v\in \VV_f$. 
Similarly, ${}^\sigma \mw\left(\Sigma,k\right)_v$
has to be the unique constituent of \[{}^\sigma\left(\lvert\det\lvert^\frac{k-1}{2}\Sigma\times\ldots\times \lvert\det\lvert^\frac{1-k}{2}\Sigma\right)_v\]
with a $K_v'$-fixed vector for almost all places $v\in \VV_f$.
Thus, the representations ${}^\sigma \mw\left(\Sigma,k\right)$ and $\mw\left({}^\sigma \Sigma\chi_\sigma,k\right)$ have to agree at almost all places and the claim follows then from Strong Multiplicity One, \emph{cf.} \cite[§4.4]{Bad}.
\end{proof}
Hence, it follows from (\ref{E:E1}) that $st=k$ and ${}^\sigma\Sigma\chi_\sigma=\Sigma'$ by Strong Multiplicity One. Assume now that $s>1$. By \ref{C:conj}, we know that $\Theta$ is cohomological and since $s>1$ the induction hypothesis implies that $^{\sigma^{-1}}\Theta$ is cuspidal. We thus can consider the discrete series representation $\mw\left(^{\sigma^{-1}}\Theta,t\right)$.
Finally,
\[\jl\left(^{\sigma^{-1}}\Theta\right)_{\VV\setminus\VV_\DD}{=}{}^{\sigma^{-1}}\jl\left(\Theta\right)_{\VV\setminus\VV_\DD}={}^{\sigma^{-1}}\mw\left({}^\sigma\Sigma\chi_\sigma,t\right)_{\VV\setminus\VV_\DD}\stackrel{(\ref{L:twMW})}{=}\]\[=\mw\left(\Sigma {{}^{\sigma^{-1}}\chi_\sigma}\chi_{\sigma^{-1}},t\right)_{\VV\setminus\VV_\DD}=\mw\left(\Sigma,t\right)_{\VV\setminus\VV_\DD}.\]
Therefore, $$\jl\left(\mw\left(^{\sigma^{-1}}\Theta,s\right)\right)_{\VV\setminus\VV_\DD}=\mw\left(\Sigma,k\right)_{\VV\setminus\VV_\DD}=\lvert JL\lvert\left(\Pi'\right)_{\VV\setminus\VV_\DD},$$
which implies $\mw\left(^{\sigma^{-1}}\Theta,s\right)=\Pi'$ by Strong Multiplicity One and the injectivity of $\jl$, a contradiction by \ref{T:MWclas}. Thus, $s=1$ and hence, ${}^\sigma \Pi '$ is cuspidal.
Moreover,
\[{}^\sigma\jl\left(\Pi'\right)_{\VV\setminus\VV_\DD}=  {}^\sigma \mw\left(\Sigma,k\right)_{\VV\setminus\VV_\DD}=\mw\left({}^\sigma \Sigma\epsilon_\sigma,k\right)_{\VV\setminus\VV_\DD}=\jl\left({}^\sigma\Pi\right)_{\VV\setminus\VV_\DD}\]
and the second claim follows again from Strong Multiplicity One.
\end{proof}
\section{Shalika models}
Let $U_{(n,n)}'$ and $\mS$ be the following two subgroups of $\GL_{2n}'$. We recall the Shalika subgroup\[\mS\coloneqq \Delta \GL'_n\rtimes U_{(n,n)}'=\left\{\bpm h&X\\0&h\epm: h\in \GL'_n, X\in M'_n\right\}.\] 
Let $\psi$ be the additive character fixed in \ref{S:2.2}. We extend this character to $\mS\left(\A\right)$ by setting $\psi\left(s\right)\coloneq\psi\left(\tr\left(X\right)\right),\,\eta\left(s\right)\coloneq\eta\left(\det'\left(h\right)\right)$ for $s=\bpm h&X\\0&h\epm.$ Let $\Pi'$ be a cuspidal irreducible representation of $\GL_{2n}'\left(\A\right)$ and we
assume there exists a Hecke character $\eta$ of $\GL_1\left(\A\right)$ such that for all $a\in \GL_1(\A)$ \[\eta\circ \det'\overbrace{\bpm a&&\\&\ddots&\\&&a\epm}^n=\omega\overbrace{\bpm a&&\\&\ddots&\\&&a\epm}^{2n},\] where we recall that $\omega$ is the central character of $\Pi'$. Let $S_\eta$ be the set of places where $\eta$ ramifies.
For $\phi\in \Pi'$ a cusp form and $g\in \GL_{2n}'\left(\A\right)$ we define the Shalika period integral by
\[\mS_\psi^\eta\left(\phi\right)\left(g\right)\coloneq\int_{Z'_{2n}\left(\A\right)\mS\left(\KK\right)\bs\mS\left(\A\right)}\phi\left(sg\right)\psi\left(s\right)^{-1}\eta\left(s\right)^{-1}\ddd s.\]
Note that this is well defined since \[Z'_{2n}\left(\A\right)\Delta \GL'_n\left(\KK\right)\bs \Delta \GL'_n\left(\A\right)\]
has finite measure and $U_{(n,n)}'\left(\KK\right)\bs U_{(n,n)}'\left(\A\right)$ is compact.
If there exists a $\phi$ such that $\mS_\psi^\eta\left(\phi\right)$ does not vanish for some $g\in \GL_{2n}'\left(\A\right)$, this gives a nonzero intertwining operator of $\GL_{2n}'\left(\A\right)$-representations
\[\mS_\psi^\eta\colon\Pi'\rightarrow \mathrm{Ind}_{\mS\left(\A\right)}^{\GL_{2n}'\left(\A\right)}\left(\eta\otimes \psi\right),\]
where the second space is the vector-space consisting of smooth functions with the obvious left-invariance.
In this case we say that $\Pi'$ admits a Shalika model with respect to $\eta$.
For $v\in \VV$ we define local Shalika models of $\Pi'\cong \bigotimes_{v\in \VV}'\Pi_v'$ as follows.
We also denote the local counterpart by $\mathrm{Ind}_{\mS\left(\KKv\right)}^{\GL'_{2n}\left(\KKv\right)}\left(\eta_v\otimes \psi_v\right).$
If $v$ is a finite place in $V$, we say $\Pi_v'$ admits a local Shalika model if there exists a non-zero intertwiner \[\Pi_v'\ra \mathrm{Ind}_{\mS\left(\KKv\right)}^{\GL'_{2n}\left(\KKv\right)}\left(\eta_v\otimes \psi_v\right)\] of $\GL_{2n}'\left(\KKv\right)$-representations. For $v\in {\VV_\infty}$ a priori, $\Pi_v'$ is by our conventions not an honest $\GL\left(\KKv\right)$-representation and therefore we have to
 consider $\left(\Pi_v'\right)^\infty$, the sub-space of smooth vectors in $\Pi_v'$. Then
 $\mathrm{Ind}_{\mS\left(\KKv\right)}^{\GL'_{2n}\left(\KKv\right)}\left(\eta_v\otimes \psi_v\right)$ and $(\Pi_v')^\infty$ are both Fr\'echet spaces and admit a natural, smooth $\GL_{2n}'\left(\KKv\right)$-action. 
 We say $\Pi_v'$ admits a local Shalika model with respect to $\eta_v$ if there exists a non-zero, continuous intertwining operator of $\GL_{2n}'\left(\KKv\right)$-representations 
\[\left(\Pi_v'\right)^\infty\rightarrow \mathrm{Ind}_{\mS\left(\KKv\right)}^{\GL'_n\left(\KKv\right)}\left(\eta_v\otimes \psi_v\right).\]
If $\Pi'$ admits a global Shalika model with respect to $\eta$, then so does $\Pi_v'$ with respect to $\eta_v$ for all $v\in \VV$.
Note that the reverse direction, \emph{i.e.} the existence of a local Shalika model for each $\Pi_v',\, v\in \VV$ implying the existence of a Shalika model for $\Pi'$, is not true in general, see \cite[Theorem 1.4]{GanTakII}. 
\subsection{ Existence} In the case of $\GL_n$ we have the following characterization of Shalika models.
\begin{theorem}[{\cite[Theorem 1]{JacSha}}]\label{T:JF}
Let $\Pi$ be a cuspidal irreducible representation of $\GL_{2n}\left(\A\right)$.
Then the following assertions are equivalent.
\begin{enumerate}
    \item There exists $\phi\in\Pi$ and $g\in \GL_{2n}\left(\A\right)$ such that $\mS_\psi^\eta\left(\phi\right)\left(g\right)\neq 0$.
    \item Let $S\subset V$ be a finite subset of places containing ${\VV_\infty}$ and the finite places where $\Pi$ and $\eta$ ramify. Then the twisted partial exterior square $L$-function 
    $L^S\left(s,\Pi,\bigwedge^2\otimes\eta^{-1}\right)$ has a pole at $s=1$.
\end{enumerate}
\end{theorem}
If $\DD$ does not split over $\KK$ there is no longer such a nice criterion.
In the case $n=2$ and $\DD$ a quaternion algebra we have the following criterion by \cite{GanTakII}, which is a consequence of the global theta correspondence.
\begin{theorem}[{\cite[Theorem 1.3]{GanTakII}}]\label{T:G-T}
Assume $\DD$ is a quaternion algebra, $\Pi'$ a cuspidal irreducible representation of $\GL'_{2}\left(\A\right)$ and $\eta$ the Hecke character we fixed above. Let $S\subset V$ be a finite subset of places containing ${\VV_\infty}$ and the finite places where $\Pi$ and $\eta$ ramify.
If $\jl\left(\Pi'\right)$ is cuspidal the following assertions are equivalent.
\begin{enumerate}
    \item $\Pi'$ admits a Shalika model with respect to $\eta$.
    \item The twisted partial exterior square $L$-function $L^S\left(s,\Pi',\bigwedge^2\otimes\eta^{-1}\right)$ has a pole at $s=1$ and for all $v\in \VV_\DD$ the representation $\Pi_v'$ is not of the form
$\lvert\det'\lvert_v^{\frac{1}{2}}\tau_1\times\lvert\det'\lvert_v^{-\frac{1}{2}}\tau_2,$
    where $\tau_1$ and $\tau_2$ are representations of $\GL'_1\left(\KKv\right)$ with central character $\eta_v$.
\end{enumerate}
If $\jl\left(\Pi'\right)$ is not cuspidal it is of the form $\mw\left(\Sigma,2\right)$ for some cuspidal irreducible representation $\Sigma$ of $\GL_2\left(\A\right)$. Then the following assertions are equivalent.
\begin{enumerate}
    \item $\Pi'$ admits a Shalika model with respect to $\eta$.
    \item The central character $\omega_\Sigma$ of $\Sigma$ equals $\eta$.
    \item The twisted partial exterior square $L$-function $L^S\left(s,\Pi',\bigwedge^2\otimes\eta^{-1}\right)$ has a pole at $s=2$.
\end{enumerate}
\end{theorem}
Thus, the situation is much more delicate in the case where $\jl\left(\Pi'\right)$ is cuspidal because of the second, local condition. On the other hand, if $\jl\left(\Pi'\right)$ is not cuspidal we have a priori $\eta^2=\omega_\Pi=\omega_\Sigma^2$, hence, $\eta$ and $\omega_\Sigma$ only differ by a quadratic character at most.
\subsection{ Shalika zeta-integrals}
The connection between $L$-functions and Shalika models can first be seen from the next two theorems, which are extensions of \cite[Proposition 2.3, Proposition 3.1, Proposition 3.3]{FriJac}.
\begin{theorem}\label{T:global}
Let $\Pi'$ be a cuspidal irreducible representation of $\GL_{2n}'\left(\A\right)$.
Assume $\Pi'$ admits a Shalika model with respect to $\eta$ and let $\phi\in\Pi'$ be a cusp form. Consider the integrals
\[\Psi\left(s,\phi\right)\coloneqq \int_{Z'_{2n}\left(\A\right)H_n'\left(\KK\right)\bs H_n'\left(\A\right)}\phi\left(\bpm h_1&0\\0&h_2\epm\right)\left\lvert \frac{\det'\left(h_1\right)}{\det'\left(h_2\right)}\right\lvert^{s-\frac{1}{2}}\eta\left(h_2\right)^{-1}\ddd h_1\ddd h_2,\]
\[\zeta\left(s,\phi\right)\coloneqq \int_{\GL'_n\left(\A\right)}S^\eta_\psi\left(\phi\right)\left(\bpm g_1&0\\0&1\epm\right)\lvert\det'\left(g_1\right)\lvert^{s-\frac{1}{2}}\ddd g_1.\]
Then $\Psi\left(s,\phi\right)$ converges absolutely for all $s$ and $\zeta\left(s,\phi\right)$ converges absolutely if $\re>>0$. Moreover, if $\zeta\left(s,\phi\right)$ converges absolutely,
$\Psi\left(s,\phi\right)=\zeta\left(s,\phi\right)$.
\end{theorem}
In \cite{FriJac} this statement was proven for $\DD=\KK$, and we will show in \ref{S:JF} that their proof extends with some small adjustments to the case of $\DD$ being a division algebra.
Let $\xi_\phi\in\mS_\psi^\eta\left(\Pi'\right)$ and choose an isomorphism $\mS_\psi^\eta\left(\Pi'\right)\xrightarrow{\cong}\bigotimes_{v\in \VV}\mS_{\psi_v}^{\eta_v}\left(\Pi_v'\right)$. Assume the image of $\xi_\phi$ can be written as a pure tensor
\[\xi_\phi\mapsto \bigotimes_{v\in \VV}\xi_{\phi,v}\in \bigotimes_{v\in \VV}\mS_{\psi_v}^{\eta_v}\left(\Pi_v'\right).\]
We can now consider the local version of the above integral 
\[\zeta_v\left(s,\xi_{\phi,v}\right)\coloneqq \int_{\GL'_n\left(\KKv\right)}\xi_{\phi,v}\left(\bpm g_1&0\\0&1\epm\right)\lvert\det'_v\left(g_1\right)\lvert^{s-\frac{1}{2}}\dv g_{1},\] where $\xi_{\phi,v}\in \mS_{\psi_v}^\eta\left(\Pi_v'\right)$.
The local Shalika integrals are then connected to the local $L$-factors by the following theorem.
\begin{theorem}\label{T:connL}
Let $\Pi'$ be a cuspidal irreducible representation of $\GL_{2n}'\left(\A\right)$
and assume $\Pi'$ admits a Shalika model with respect to $\eta$. Then for each place $v\in \VV$ and $\xi_{v}\in\mS_{\psi_v}^{\eta_v}\left(\Pi_v'\right)$ there exists an entire function $P\left(s,\xi_{v}\right)$, with $P\left(s,\xi_{v}\right)\in \CC[q_v^{s-\frac{1}{2}},q_v^{\frac{1}{2}-s}]$ if $v\in {\VV_f}$, such that
\[\zeta_v\left(s,\xi_{v}\right)=P\left(s,\xi_{v}\right)L\left(s,\Pi_v\right)\]
and hence, $\zeta_v\left(s,\xi_{v}\right)$ can be analytically continued to $\CC$. Moreover, for each place $v$ there exists a vector $\xi_{v}$ such that $P\left(s,\xi_{v}\right)=1$. If $v$ is a place where neither $\Pi'$ nor $\psi$ ramify this vector can be taken as the spherical vector $\xi_{\Pi_v'}$ normalized by $\xi_{\Pi_v'}\left(\mathrm{id}\right)=1$.
\end{theorem}
In the case $\KK=\DD$ the existence of such a holomorphic $P$ was proven in \cite{FriJac} and in \cite[Corollary 5.2]{LapMao} it was shown that $P$ is actually a polynomial in $\CC[q_v^{s-\frac{1}{2}},q_v^{\frac{1}{2}-s}]$.
\ref{T:global} and \ref{T:connL} imply for $\xi_{\phi,f}\cong \bigotimes_{v\in {\VV_f}}'\xi_{\phi,f}$ and $\re>>0$
\[\zeta_f\left(s,\xi_{\phi,f}\right)\coloneq\int_{\GL'_n\left(\A_f\right)}\xi_{\phi,f}\left(\bpm g_1&0\\0&1\epm\right)\lvert\det'\left(g_1\right)\lvert_f^{s-\frac{1}{2}}\df g_{1}=\]\[=\prod_{v\in {\VV_f}}P\left(s,\xi_{\phi,v}\right)L\left(s,\Pi_v\right).\]
\subsection{ \texorpdfstring{$\Aut(\CC)$}{Aut(C)}-action}Let $\Pi'$ be a cuspidal irreducible representation of $\GL_{2n}'\left(\A\right)$ and assume $\Pi'$ admits a Shalika model with respect to $\eta$. It is natural to ask whether ${}^\sigma\Pi'$ admits a Shalika model with respect to ${}^\sigma\eta$ assuming that ${}^\sigma\Pi'$ is cuspidal. In the split case it was proven in the appendix of \cite{GroRag} that if $\Pi'$ admits a Shalika model with respect to $\eta$, then ${}^\sigma\Pi'$ admits one with respect to ${}^\sigma\eta$. 
\begin{definition}\label{D:os}
We say the $\Aut\left(\CC\right)$-orbit of $\Pi'$ admits a Shalika model with respect to $\eta$ if $^\sigma\Pi'$ is cuspidal and admits a Shalika model with respect to $^\sigma\eta$ for all $\sigma\in\Aut\left(\CC\right)$.
\end{definition}
Note that the above definition has been studied in the wider context of certain distinction problems in \cite{GanRag2012}, in which an extended discuss on this phenomenon can be found.
In the case of $n=1$ and $\DD$ a quaternion algebra \ref{T:G-T} allows us to prove the following.
\begin{lemma}\label{L:admism}
Let $\DD$ be a quaternion algebra and $\Pi'$ a cuspidal irreducible cohomological representation of $\GL'_2\left(\A\right)$. If $\Pi'$ admits a Shalika model with respect to $\eta$ then ${}^\sigma \Pi'$ admits one with respect to ${}^\sigma \eta$.
\end{lemma}
\begin{proof}
Note first that by \ref{P:cusp} $^\sigma\Pi'$ is cuspidal.
Assume first that $\jl\left(\Pi'\right)$ is not cuspidal, \emph{i.e.} $\jl\left(\Pi'\right)=\mw\left(\Sigma,2\right)$ for some cuspidal irreducible representation of $\GL_2\left(\A\right)$. From \ref{C:conj}, \ref{P:cusp} and \ref{L:twMW} it follows that \[\jl\left({}^\sigma \Pi'\right)={}^\sigma\jl\left(\Pi'\right)=\mw\left({}^\sigma \Sigma,2\right).\] Since the central character $\omega_\Sigma$ of $\Sigma$ equals by assumption $\eta$, the central character of ${}^\sigma \Sigma$ equals ${}^\sigma \eta$. Thus we are done by \ref{T:G-T}.
Next assume $\jl\left(\Pi'\right)$ is cuspidal and hence, $\jl\left(\Pi'\right)$ admits a Shalika model with respect to $\eta$ by \ref{T:JF} and \ref{T:G-T}. Thus, $^\sigma\jl\left(\Pi'\right)=\jl\left(^\sigma\Pi'\right)$ admits also a Shalika model with respect to $^\sigma\eta$ by \cite[Theorem 3.6.2]{GroRag} and hence $L^S\left(s, {}^\sigma\Pi,\bigwedge^2\otimes {}^\sigma\eta^{-1}\right)$ has a pole at $s=1$. Moreover, 
if $v$ is a non-split place of $\DD$ and ${}^\sigma \Pi_v$ were of the form  \[{}^\sigma \Pi_v\cong \lvert\det'\lvert_v^\frac{1}{2}\tau_1\times\lvert\det'\lvert_v^{-\frac{1}{2}}\tau_2,\] where $\tau_i$ are representations of $\GL'_1\left(\KKv\right)$ with central character ${}^\sigma\eta$. This would lead to the contradiction, since by (\ref{E:E2}) \[\Pi_v={}^{\sigma^ {-1}}(\lvert\det'\lvert_v^\frac{1}{2}\tau_1\times\lvert\det'\lvert_v^{-\frac{1}{2}}\tau_2)=\lvert\det'\lvert_v^\frac{1}{2}{}^{\sigma^{-1}}\tau_1\times\lvert\det'\lvert_v^{-\frac{1}{2}}{}^{\sigma^{-1}}\tau_2.\]
\end{proof}
\begin{remark}We have currently no proof in the general case $n>2$ and unfortunately the methods of \cite{GanTakII} do not generalize well beyond the quaternion case. Hence, we can only conjecture the following. \end{remark}
\begin{conjecture}
Let $\Pi'$ be a cuspidal irreducible cohomological representation of $\GL_n'\left(\A\right)$ such that $\jl\left(\Pi'\right)$ is residual. If $\Pi'$ admits a Shalika model with respect to $\eta$, then so does the $\Aut\left(\CC\right)$-orbit of $\Pi'$.
\end{conjecture}
In \cite{GroRag} the authors define an action of $\Aut\left(\CC\right)$ on a given Shalika model and we will generalize this now to our setting.
Let $\psi_f$ be the finite part of the additive character $\psi$, which takes values in $\mu_\infty\subseteq \CC^\times$, the subgroup of all roots of unity of $\CC^\times$. We will associate to an element $\sigma\in\Aut\left(\CC\right)$ an element $t_\sigma\in \A^\times$ such that for all $x\in \A$ one has $\sigma\left(\psi\left(x\right)\right)=\psi\left(t_\sigma x\right).$ More explicitly, we construct $t_\sigma$ by first restricting $\sigma$ to $\QQ\left(\mu_\infty\right)$ and sending it to $\prod_p \mathbb{Z}_p^\times$ via the global symbol map of Artin reciprocity \[\Aut\left(\QQ\left(\mu_\infty\right)/\QQ\right)\xrightarrow{\cong} \widehat{\mathbb{Z}}^\times=\prod_{p \text{ prime}} \mathbb{Z}_p^\times,\] then embed the so obtained element into $\A$ via the diagonal embedding $\mathbb{Z}_p\hookrightarrow \prod_{v\lvert p}\OO_v$.
 Next we define the action of $\sigma\in\Aut\left(\CC\right)$ on the finite part $\mS_{\psi_f}^{\eta_f}\left(\Pi_f'\right)$ by sending $\xi_f$ to \[g_f\mapsto {}^\sigma\xi_f\left( g_f\right)\coloneqq \sigma\left(\xi_f\left(\tbf^{-1}_\sigma g_f\right)\right),\, g_f\in \GL_{2n}'\left(\A_f\right),\]where
$\tbf_\sigma=\mathrm{diag}\left(\overbrace{t_\sigma,\ldots,t_\sigma}^n,\overbrace{1,\ldots,1}^n\right).$
This gives a $\sigma$-linear intertwining operator
\begin{equation}\label{E:actionshalika}
\sigma^*\colon\mathrm{Ind}_{\mS\left(\A_f\right)}^{\GL'_{2n}\left(\A_f\right)}\left(\eta_f\otimes \psi_f\right)\rightarrow \mathrm{Ind}_{\mS\left(\A_f\right)}^{\GL'_{2n}\left(\A_f\right)}\left({}^\sigma \eta_f\otimes \psi_f\right),\, \xi_f\mapsto{}^\sigma\xi_f.\end{equation}
Completely analogously we define a $\sigma$-linear intertwining operator 
\[\sigma^*\colon\mathrm{Ind}_{\mS\left(\KKv\right)}^{\GL'_{2n}\left(\KKv\right)}\left(\eta_v\otimes \psi_v\right)\rightarrow \mathrm{Ind}_{\mS\left(\KKv\right)}^{\GL'_{2n}\left(\KKv\right)}\left({}^\sigma \eta_v\otimes \psi_v\right)\] for every finite place $v$, where we use $t_{\sigma,v}$ and $\tbf_{\sigma,v}$ instead of $t_{\sigma}$ and $\tbf_\sigma$.
\subsection{ Uniqueness} Let $\Pi'$ be a cuspidal irreducible cohomological representation of $\GL_{2n}'\left(\A\right)$ which admits a Shalika model with respect to $\eta$. Let $v\in {\VV_f}$ be a finite place. In order to proceed we need the local uniqueness of the Shalika model, \emph{i.e.} for every irreducible representation $\Pi_v'$ of $\GL_{2n}'\left(\KKv\right)$ the claim that
\[\dim_\CC \mathrm{Hom}_{\GL_{2n}'\left(\KKv\right)}\left(\Pi_v',\mathrm{Ind}_{\mS\left(\KKv\right)}^{\GL'_{2n}\left(\KKv\right)}\left(\eta_v\otimes \psi_v\right)\right)\le 1.\]
By Frobenius reciprocity every such map corresponds uniquely to a Shalika functional $\lambda\in \mathrm{Hom}_{\mS\left(\KKv\right)}\left(\Pi_v',\eta_v\otimes\psi_v\right).$ \begin{definition}\label{D:ol}
We say that the $\Aut\left(\CC\right)$-orbit of $\Pi'$ has a unique local Shalika model if ${}^\sigma\Pi_v'$ has a unique Shalika model for all $v\in {\VV_f}$ and $\sigma\in \Aut\left(\CC\right)$.
\end{definition}
In the split case or when $\DD$ is a quaternion algebra the following was proven in \cite{Nie}.
\begin{theorem}[{\cite[ Theorem 3.4]{Nie}}]\label{T:nien}
Let $\DD$ be a field or a quaternion algebra. If $\DD$ is quaternion, assume $\eta_v$ is trivial. Then \[\dim_\CC\mathrm{Hom}_{\GL_{2n}'\left(\KKv\right)}\left(\Pi_v',\mathrm{Ind}_{\mS\left(\KKv\right)}^{\GL'_{2n}\left(\KKv\right)}\left(\eta_v\otimes \psi_v\right)\right)\le 1.\]
\end{theorem} This is yet another reason why we will have to restrict ourselves to the case $\DD$ being quaternion in the end.
Combining \ref{P:cusp}, \ref{L:admism}, and \ref{T:nien} we have proved the following. 
\begin{theorem}\label{T:L}
Let $\Pi'$ be a cuspidal irreducible cohomological representation of $\GL_2'\left(\A\right)$ which admits a Shalika model with respect to $\eta$ and assume that $\DD$ is a quaternion algebra. Then the $\Aut\left(\CC\right)$-orbit of $\Pi'$ is cuspidal cohomological, admits a Shalika model with respect to $\eta$ and has a unique local Shalika model if $\eta$ is trivial.
\end{theorem}
For the rest of the chapter let us collect the following decorations of a cuspidal irreducible representation $\Pi'$ of $\GL\left(\A\right)$:
\begin{enumerate}\label{I:prop}
    \item $\Pi'$ is cuspidal irreducible cohomological representation of $\GL_{2n}'\left(\A\right)$.
    \item The $\Aut\left(\CC\right)$-orbit of $\Pi'$ is cuspidal cohomological and admits a Shalika model with respect to $\eta$.
    \item The $\Aut\left(\CC\right)$-orbit of $\Pi'$ has a local unique Shalika model.
\end{enumerate}
Moreover, we also fix a splitting ${}^\sigma \Pi'\cong {}^\sigma\Pi_\infty'\otimes {}^\sigma \Pi_f',\, {}^\sigma\Pi_f'\iso \bigotimes_{v\in {\VV_f}}'{}^\sigma\Pi_v'$ and a Shalika model of ${}^\sigma\Pi_v$
\[\mS_{\psi_v}^{{}^\sigma\eta_v}\colon {}^\sigma\Pi_v'\ra  \mathrm{Ind}_{\mS\left(\KKv\right)}^{\GL'_{2n}\left(\KKv\right)}\left({}^\sigma\eta_v\otimes \psi_v\right)\]
 for all $\sigma\in\Aut\left(\CC\right),\, v\in {\VV_f}$.
\begin{lemma}
For $\Pi'$ as in \ref{I:prop}, $v\in {\VV_f}$ and the action of (\ref{E:actionshalika}) we have \[\sigma^*\left(\mS_{\psi_v}^{\eta_v}\left(\Pi_v'\right)\right)=\mS_{\psi_v}^{{}^\sigma\eta_v}\left({}^\sigma\Pi_v'\right)\] for all $\sigma\in\Aut\left(\CC\right)$. For any finite extension $\mathbb{K}$ of $\QQ\left(\Pi_v',\eta_v\right)$ we have a $\mathbb{K}$-structure
\[\mS_{\psi_v}^{\eta_v}\left(\Pi_v'\right)_\mathbb{K}\coloneq\mS_{\psi_v}^{\eta_v}\left(\Pi_v'\right)^{\Aut\left(\CC/\mathbb{K}\right)}\]
on $\mS_{\psi_v}^{\eta_v}\left(\Pi_v'\right)$.
\end{lemma}
\begin{proof} For the first assertion, note that the representation ${}^\sigma\Pi_v'$ has on the on hand the unique Shalika model $\mS_{\psi_v}^{{}^\sigma\eta_v}\left({}^\sigma\Pi_v'\right)$ with respect to ${}^\sigma\eta_v$, but on the other hand, the $\sigma$-linear map
\[\Pi_v'\iso \mS_{\psi_v}^{\eta_v}\left(\Pi_v'\right)\stackrel{\sigma^*}{\hookrightarrow} \mathrm{Ind}_{\mS\left(\KKv\right)}^{\GL'_{2n}\left(\KKv\right)}\left({}^\sigma \eta_v\otimes \psi_v\right)\] gives rise to a linear map 
\[{}^\sigma\Pi_v'\hookrightarrow \mathrm{Ind}_{\mS\left(\KKv\right)}^{\GL'_{2n}\left(\KKv\right)}\left({}^\sigma \eta_v\otimes \psi_v\right).\]
Therefore, the assumed local uniqueness of the Shalika model implies that up to a scalar those two maps have to agree and hence, their image is identical.
For the second assertion, one can follow exactly the line of reasoning as in the proof of \cite[Theorem 3.1]{JiaSunTia}.
\end{proof}
We introduce the following notation. Let $v\in {\VV_f}$, $\sigma\in \Aut(\CC)$ and $f\in \CC(q_v^{s-\frac{1}{2}},q_v^{\frac{1}{2}-s}).$ We denote by $f^\sigma$ the rational function obtained by applying $\sigma$ to all coefficients of $f$ for some $\sigma\in\Aut\left(\CC\right)$, which is the same as applying $\sigma$ to the coefficients of $f$ considered as a Laurent-series. Moreover, $\sigma\left(f\left(\frac{1}{2}\right)\right)=f^\sigma\left(\frac{1}{2}\right).$
\begin{lemma}\label{L:compat}
Let $\Pi'$ be a cuspidal irreducible automorphic representation of $\GL_{2n}'\left(\A\right)$ with local representations $\Pi_v$ of $\GL'_{2n}\left(\KKv\right)$ for $v\in \VV$. Then for every finite place $v$
\[L^\sigma\left(s,\Pi_v'\right)=L\left(s,{}^\sigma\Pi_v'\right),\]
and hence, if $L(s,\Pi_v')$ has no pole at $s=\frac{1}{2}$, $L\left(\frac{1}{2},\Pi_v'\right)\in \QQ\left(\Pi_v'\right)$.
\end{lemma}
\begin{proof}
For the first claim, note that by \cite[Theorem 6.18]{Bad}, we have an explicit description of the local $L$-factors and that for $\Pi_v'$ a representation of $\GL_m'(\KKv)$, $L(s+\frac{md-1}{2},\Pi_v')\in \CC(q_v^{s},q_v^{s})$. We denote then for $f\in \CC(q_v^{s},q_v^{s})$ by ${}^\sigma f$ the coefficient-wise application of $\sigma$. Note that for $m$ even we thus have that ${}^\sigma L(s+\frac{md-1}{2},\Pi_v')=L^\sigma(s+\frac{md}{2},\Pi_v')$. One can then carry over the proof of \cite[Lemma 4.6]{CloI} \emph{mutatis mutandis} from the case $\GL_m$ to $\GL_m'$ to obtain that 
\[{}^\sigma L\left(s+\frac{md-1}{2},\Pi_v'\right)= L\left(s+\frac{md-1}{2},{}^\sigma \Pi_v'\right).\]
Thus, for $m=2n$, one has$L^\sigma(s+nd,\Pi_v')=L\left(s+nd,{}^\sigma \Pi_v'\right)$ and since $q_v^{nd}\in \QQ$, the first claim follows.
For the second claim, it is enough to observe that in this case
\[\sigma\left(L\left(\frac{1}{2},\Pi_v'\right)\right)=L{}^\sigma\left(\frac{1}{2},\Pi_v'\right)=L\left(\frac{1}{2},{}^\sigma\Pi_v'\right)=L\left(\frac{1}{2},\Pi_v'\right)\]
for any $\sigma\in \Aut(\CC/ \QQ(\Pi_v'))$.
\end{proof}
\begin{lemma}\label{L:nomvec}
For $\Pi'$ as in \ref{I:prop} and $v\in {\VV_f}$ there exists a vector 
$\xi^0_{\Pi',v}\in\mS_{\psi_v}^{\eta_v}\left(\Pi_v'\right)_{\QQ\left(\Pi',\eta\right)}$ such that
$\zeta_v\left(s,\xi^0_{\Pi',v}\right)=L\left(s,\Pi_v'\right)$
if $v\notin S_{\Pi_f',\psi}$ and $P^\sigma(s,\xi^0_{\Pi',v})=P(s,{}^\sigma\xi^0_{\Pi',v})$ for all $\sigma \in \Aut\left(\CC/\mathbb{Q}(\Pi',\eta_v)\right)$
if $v\in S_{\Pi_f',\psi}$.
\end{lemma}
\begin{proof}
We again follow the proof of \cite[Theorem 3.1]{JiaSunTia}. For $v\notin S_{\Pi_f',\psi}$
we choose $\xi_{\Pi',v}^0$ to be the normalized spherical vector of \ref{T:connL}. Note that $\xi_{\Pi',v}^0\in \mS_{\psi_v}^{\eta_v}\left(\Pi_v'\right)_{\QQ\left(\Pi',\eta_v\right)}$, since for $v\notin S_{\Pi_f',\psi}$ the normalization $\xi_{\Pi',v}^0\left(1\right)=1$ implies that ${}^\sigma\xi_{\Pi',v}^0\left(1\right)=1$ and hence, because $\sigma^*(\mS_{\psi_v}^{\eta_v}\left(\Pi_v'\right))=\mS_{\psi_v}^{\eta_v}\left(\Pi_v'\right)$, ${}^\sigma\xi_{\Pi',v}^0=\xi_{\Pi_v'}^0$.
Thus, $P(s,\xi_{\Pi'_f,v}^0)=1$ for all $v\notin S_{\Pi_f',\psi}$ and therefore $\zeta_v\left(s,\xi^0_{\Pi',v}\right)=L\left(s,\Pi_v'\right)$. For $v\in S_{\Pi'_f,\psi}$, pick any non-zero $\xi^0_{\Pi',v}\in\mS_{\psi_v}^{\eta_v}\left(\Pi_v'\right)_{\QQ\left(\Pi',\eta\right)}$ and
  recall that $P\left(s,\xi_{\Pi',v}^0\right)\in \CC[q_v^{s-\frac{1}{2}},q_v^{\frac{1}{2}-s}],$ see \ref{T:connL}, and the $L$-function $L\left(s,\Pi_v'\right)$ does not vanish at $s=\frac{1}{2}$, as it is the reciprocal of a polynomial. 
Since \[\frac{\zeta_v\left(s,\xi_{\Pi',v}^0\right)}{ L\left(s,\Pi_v'\right)}=P\left(s,\xi_{\Pi',v}^0\right),\, 
\frac{1}{ L\left(s,\Pi_v'\right)}\in\CC[q_v^{s-\frac{1}{2}},q_v^{\frac{1}{2}-s}],\]
we have $\zeta_v\left(s,\xi_{\Pi',v}^0\right)\in \CC\left(q_v^{s-\frac{1}{2}},q_v^{\frac{1}{2}-s}\right).$
From the definition of $\zeta_v\left(s,\xi_{\Pi',v}^0\right)$ it follows that the $k$-th coefficient of $q_v^{s-\frac{1}{2}}$ in $\zeta_v\left(s,\xi_{\Pi',v}^0\right)$ is \[c_k\left(\xi_{\Pi',v}^0\right)\coloneqq \int_{\substack{\GL'_n\left(\KKv\right),\\\lvert\det'\left(g_1\right)\lvert=q^{-k}}}\xi_{\Pi',v}^0\left(\bpm g_1&0\\0&1\epm\right)\dv g_1,\] which vanishes for $k<<0$ and is a finite sum, see the proof of \ref{T:connL}. Hence, by a change of variables, $c_k\left({}^\sigma\xi_{\Pi',v}^0\right)=\sigma\left(c_k\left(\xi_{\Pi',v}^0\right)\right)$ for all $\sigma\in\Aut\left(\CC/\QQ\left(\eta_v\right)\right)$. It follows that for all $s\in\CC,\,\sigma\in \Aut\left(\CC/\QQ\left(\eta_v\right)\right)$ $\zeta_v^\sigma\left(s,\xi_{\Pi',v}^0\right)=\zeta_v\left(s,{}^\sigma\xi_{\Pi',v}^0\right)$ by analytic continuation.
 \ref{L:compat} shows then that \[P^{\sigma}\left(s,\xi_{\Pi',v}^0\right)L^{\sigma}\left(s,\Pi_v'\right)=\zeta_v^\sigma\left(s,\xi_{\Pi',v}^0\right) =\zeta_v\left(s,{}^\sigma\xi_{\Pi',v}^0\right)=\]\[=P\left(s,{}^\sigma\xi_{\Pi',v}^0\right)L\left(s,{}^\sigma \Pi_v'\right)=P\left(s,{}^\sigma\xi_{\Pi',v}^0\right)L^\sigma\left(s,\Pi_v'\right)\]
for all $\sigma\in \Aut\left(\CC/\mathbb{Q}(\Pi',\eta_v)\right)$ and hence, 
$P^\sigma(s,\xi^0_{\Pi',v})=P(s,{}^\sigma\xi^0_{\Pi',v}).$
\end{proof}
 We let $\xi_{\Pi_f',0}\in \mS_{\psi_f}^{\eta_f}\left(\Pi_f'\right)$ be the image of $\bigotimes_{v\in {\VV_f}}\xi_{\Pi',v}^0$ under the fixed isomorphisms of \ref{I:prop}. 
\section{Periods}\label{I:prop2}
 In this section we will closely follow the strategy of \cite{GroRag}. Throughout the rest of the section let $\Pi'$ be an automorphic representation of $\GL_{2n}'\left(\A\right)$ as in \ref{I:prop}. Let $\mu$ be the highest weight such that $\Pi'$ is cohomological with respect to $E_\mu^\lor$ and assume that $\jl\left(\Pi'\right)=\mw\left(\Sigma,k\right)$ for some $k>1$ and $\Sigma$ a cuspidal irreducible representation of $\GL_l(\A)$ with $lk=2nd$.
Note that we also fixed a splitting isomorphism \[\Pi'\xrightarrow{\cong}\Pi_\infty'\otimes\Pi_f'\iso\bigotimes_{v\in {\VV_\infty}}\Pi_v'\otimes \Pi'_f.\]  and $\jl\left({}^\sigma\Pi'\right)={}^\sigma \mw\left(\Sigma,k\right)=\mw\left({}^\sigma\Sigma,k\right)$ by \ref{L:twMW} and \ref{P:cusp}. We also have that \[\left({}^\sigma\Pi'\right)_\infty\iso \bigotimes_{v\in {\VV_\infty}}\pi_{\sigma^{-1}\circ v}'.\]
Indeed, by \ref{T:twistiscoh} ${}^\sigma\Pi'$ is cohomological with respect to $^\sigma E_\mu^\lor$ and therefore
 \ref{T:dimcoh} and \ref{L:littlelem} show that for $v\in {\VV_\infty}$ ${}^\sigma\Pi_v'\cong A_{\underline{n'}}(\lambda_v)$, where $\lambda_v$ is determined by $\mu_{\sigma^{-1}\circ v}$ and $\underline{n'}$ is determined by $k$ and $l$. 
For $\sigma\in\Aut\left(\CC\right)$ we thus can fix a splitting isomorphism \[ \left({}^\sigma\Pi'\right)_\infty\otimes {}^\sigma\Pi_f'\iso \bigotimes_{v\in {\VV_\infty}}\pi_{\sigma^{-1}\circ v}'\otimes {}^\sigma\Pi_f'.\]
Let us give an example that satisfies all of those properties.
\begin{example}Let for a moment $\KK=\QQ$. Then in \cite[§ 6.11]{GroRag2} the following representation was constructed. Set $\Pi_\infty$ to the Langlands quotient of $F\left(1,s+2\right)\times F\left(-1,s+2\right)$ for $s$ a positive integer. This representation is cohomological with the coefficient system given by the weight vector $\left(\frac{s}{2},\frac{s}{2},-\frac{s}{2},-\frac{s}{2}\right)$. Moreover, $\Pi_\infty$ can be extended to a cuspidal irreducible automorphic representation of $\GL'_2\left(\A\right)$ with $\DD$ a quaternion algebra and is regular algebraic if $k$ is even.\end{example}
\subsection{ Orbifolds} Let $K_f'$ be a compact open subgroup of $\GL'_{2n}\left(\A_f\right)$ and denote the block diagonal embedding by $\iota\colon H'_n\hookrightarrow \GL'_{2n}.$ We set
\[\SG\coloneqq \GL'_{2n}\left(\KK\right)\bs \GL'_{2n}\left(\A\right)/K'_\infty K_f',\]\[\SH\coloneqq H'_n\left(\KK\right)\bs H'_n\left(\A\right)/\left(K'_\infty\cap H'_{n,\infty}\right)\iota^{-1}\left(K_f'\right).\]
 Let $r=\dim_\QQ\KK$ and note that if we consider $\SH$ as an orbifold, its real dimension is
\[{\dim_\RR\SH=r\left(\left(nd\right)^2-nd-1\right).}\]
\begin{lemma}
The embedding $\iota$ induces a proper map
\[\iota\colon\SH\rightarrow\SG.\]
\end{lemma}
\begin{proof}
It follows from \cite[Lemma 2.7]{Ash} that 
$ H'_n\left(\KK\right)\bs H'_n\left(\A\right)/\iota^{-1}\left(K_f'\right)\rightarrow \SG$ is proper. But this map factors as
\[ H'_n\left(\KK\right)\bs H'_n\left(\A\right)/\iota^{-1}\left(K_f'\right) \rightarrow \SH \rightarrow \SG.\]
Since the first map is surjective and the composition is proper, the second map is proper.
\end{proof}
 Next let $E_\mu^\lor$ be a highest weight representation of $\GL_{2n,\infty}'$ and consider the locally constant sheaf $\mE$ on $\SG$, whose espace \'etal\'e is
\[\GL'_{2n}(\A)/K'_\infty K_f'\times_{\GL'_{2n}\left(\KK\right)}E_\mu^\lor,\]
We consider its cohomology groups of compact support
\[H^*_c\left(\SG,\mE\right),\, H^*_c\left(\SH,\mE\right).\]
Both carry a natural structure of a module of the Hecke algebra
\[\mathcal{H}_{K_f'}^{\GL'_{2n}}=S\left(K_f'\bs \GL'_{2n}\left(\A_f\right)/K_f'\right),\,\mathcal{H}_{K_f'}^{H'_n}=S\left( \iota^{-1}(K_f')\bs H'_n\left(\A_f\right)/\iota^{-1}(K_f')\right),\]
where the product is as usual given by convolution.
 Now since $\iota$ is proper, it defines a map between compactly supported cohomology groups
\[\iota^*\colon H^*_c(\SG,\mE)\rightarrow H^*_c(\SH,\mE).\]
Recall that for all $\sigma\in\Aut(\CC)$ there exists then a
$\sigma$-linear isomorphism $\sigma\colon E_\mu^\lor\rightarrow  {}^\sigma E_\mu^\lor$ of 
$\GL_{2n}'(\KK)$-representations.
Thus, there exist natural $\sigma$-linear isomorphisms of Hecke algebra-modules
\[\sigma_{\GL_{2n}'}^*\colon H^*_c\left(\SG,\mE\right)\rightarrow H^*_c\left(\SG,{}^\sigma\mE\right),\,\sigma_{H'_n}^*\colon H^*_c\left(\SH,\mE\right)\rightarrow H^*_c\left(\SH,{}^\sigma \mE\right),\]
as well as a morphism
\[{}^\sigma\iota^*\colon H^*_c\left(\SG,{}^\sigma \mE\right)\rightarrow H^*_c\left(\SH,{}^\sigma \mE\right).\]
Then the following diagram commutes.
\begin{equation}\label{E:sumdia2}
\begin{tikzcd}
H^*_c\left(\SG,{}^\sigma \mE\right)\arrow[rr, "\iota^*"]\arrow[d, "\sigma_{\GL'_{2n}}^*"] & & H^*_c\left(\SH,\mE\right)\arrow[d, "\sigma_{H'_n}^*"]\\
H^*_c\left(\SG,{}^\sigma \mE\right)\arrow[rr, "{}^\sigma \iota^*"] & & H^*_c\left(\SH,{}^\sigma \mE\right)
\end{tikzcd}
\end{equation}
\begin{lemma}[{\cite[Lemma 7.3]{GroRag2}}]
The $\mathcal{H}_{K_f'}^{\GL'_{2n}}$-module $H^*_c\left(\SG,\mE\right)$ and the $\mathcal{H}_{K_f'}^{H'_n}$-module $H^*_c\left(\SH,\mE\right)$ are defined over $\QQ\left(\mu\right)$ by taking $\Aut(\CC/ \QQ(\mu))$-invariant vectors under the action given by above $\sigma_{\GL_{2n}'}^*$, respectively, $\sigma_{H_n'}^*$.
\end{lemma}
If $K_f'\subseteq K_f''$ consider the canonical map $\SG\rightarrow \textbf{S}_{K_f''}^{\GL_{2n}'}$. This allows us to define the space \[\bsg\coloneqq \lim_{\stackrel{\longleftarrow}{K_f'}}\SG\]
as a projective limit. Note that $\mE$ naturally extends to $\bsg$ and hence, the cohomology $H^*_c\left(\bsg,\mE\right)$ is a $\GL(\A_f)$-module.
\begin{prop}[{\cite[Proposition 7.16, Theorem 7.23]{GroRag2}}]
There exists an inclusion of the space
\[H^*_{cusp}\left(\GL_{2n}',E_\mu^\lor\right)\coloneqq \bigoplus_{\Pi'\text{ cuspidal}}H^*\left(\lig_\infty',K_\infty',\Pi_\infty'\otimes E_\mu^\lor\right)\otimes \Pi_f'\]
into $H^*_c\left(\bsg,\mE\right)$ respecting the $\GL_{2n}'\left(\A_f\right)$-action.
Write $H^*_c\left(\bsg,\mE\right)\left(\Pi_f'\right)$ for the image of  \[H^*\left(\lig_\infty',K_\infty',\Pi_\infty'\otimes E_\mu^\lor\right)\otimes \Pi_f'\] under this inclusion.

 If $K_f'$ fixes $\Pi'_f$, we obtain an inclusion $H^*_c\left(\bsg,\mE\right)\left(\Pi'_f\right)\hookrightarrow H^*_c\left(\SG,\mE\right)$ and we denote its image again by $H^*_c\left(\bsg,\mE\right)\left(\Pi'_f\right)$.
 Moreover, the isomorphism $\sigma_{\GL_{2n}'}^*$ respects the decomposition, i.e. if $\Pi'$ and ${}^\sigma\Pi'$ are both cuspidal then \[\sigma_{\GL_{2n}'}^*(H^*_c\left(\bsg,\mE\right)\left(\Pi'_f\right))=H^*_c\left(\bsg,\mE\right)\left({}^\sigma\Pi'_f\right)\]
 for $\sigma\in \Aut\left(\CC\right)$ and thus if the $\Aut(\CC)$-orbit of $\Pi'$ is cuspidal, the cohomology group $H^*_c\left(\SG,\mE\right)\left(\Pi'_f\right)$ is defined over $\QQ\left(\Pi'\right)$.
\end{prop}
\subsubsection{}\label{S:5.3} Let $q_0$ be the lowest degree in which the cohomology 
$H^{q_0}\left(\lig'_\infty,K_\infty', \Pi_\infty'\otimes E_\mu^\lor \right)$
does not vanish. Thus, by \ref{T:dimcoh}
\[\CC\cong H^{q_0}\left(\lig'_\infty,K_\infty', \mS_{\psi_\infty}^{\eta_\infty}\left(\Pi_\infty'\right)\otimes E_\mu^\lor \right)= \left(\bigwedge^{q_0}\left(\lig'_\infty,/\mathfrak{k}'_\infty\right)^*\otimes\mS_{\psi_\infty}^{\eta_\infty}\left(\Pi_\infty'\right)\otimes E_\mu^\lor\right)^{K_\infty'}.\]
We fix once and for all a generator of $H^{q_0}\left(\lig'_\infty,K_\infty', \mS_{\psi_\infty}^{\eta_\infty}\left(\Pi_\infty'\right)\otimes E_\mu^\lor\right)$
as follows.
First fix an Künneth-isomorphism \[\mathfrak{K}\colon H^{*}\left(\lig'_\infty,K_\infty', \mS_{\psi_\infty}^{\eta_\infty}\left(\Pi_\infty'\right)\otimes E_\mu^\lor\right)\iso \bigotimes_{v\in {\VV_\infty}} H^{*}\left(\lig'_v,K_v', \mS_{\psi_v}^{\eta_v}\left(\Pi_v'\right)\otimes E_{\mu_v}^\lor\right),\] which is determined by the already fixed isomorphism $\mS_{\psi_\infty}^{\eta_\infty}\left(\Pi_\infty'\right)\cong \bigotimes_{v\in {\VV_\infty}}\mS_{\psi_v}^{\eta_v}\left(\Pi_v'\right),$ 
and let $q_{0,v}$ be the lowest degree in which the cohomology 
\[H^{q_{0,v}}\left(\lig'_v,K_v', \mS_{\psi_v}^{\eta_v}\left(\Pi_v'\right)\otimes E_{\mu_v}^\lor\right)=\left(\bigwedge^{q_{0,v}}\left(\lig'_v/\mathfrak{k}'_v\right)^*\otimes\mS_{\psi_v}^{\eta_v}\left(\Pi_v'\right)\otimes E_{\mu_v}^\lor\right)^{K_v'}\] does not vanish and similarly we fix Künneth-isomorphisms $\mathfrak{K}_\sigma$ for all $\sigma\in \Aut(\CC)$.
For $v\in {\VV_\infty}$ we then choose a generator of this space of the form
\begin{equation}\label{E:periods}[\Pi_v']\coloneqq \sum_{\underline{i}=\left(i_1,\ldots,i_{q_{0,v}}\right)}\sum_{\alpha=1}^{\dim E_{\mu_v}^\lor} X_{\underline{i}}^*\otimes \xi_{v,\alpha,\underline{i}}\otimes e_\alpha^\lor, \end{equation}
where 
\begin{enumerate}
\item Pick a $\mathbb{L}$-basis $\{Y_i\}$ of $\mathfrak{h}_v'/(\mathfrak{h}_v'\cap \mathfrak{k}'_v)$.
    \item Extend $\{Y_i\}$ to a $\mathbb{L}$-basis $\{X_i\}$ of
   $\lig'_v/\mathfrak{k}'_v$, set $\{X_i^*\}$ to the corresponding dual basis of $\left(\lig'_v/\mathfrak{k}'_v\right)^*$ and $X_{\underline{i}}^*\coloneq\bigwedge_{i\in\underline{i}}X_i^*$.
    \item A $\QQ\left(\mu\right)$-basis $e^\lor_\alpha$ of $E_{\mu_v}^\lor$.
    \item For each $\alpha$ and $\underline{i}$ a vector $\xi_{v,\alpha,\underline{i}}\in \mS_{\psi_v}^{\eta_v}\left(\Pi_v'\right)$.
\end{enumerate}
We then set $[\Pi_\infty']\coloneqq \mathfrak{K}^{-1}\left(\bigotimes_{v\in {\VV_\infty}}[\Pi_v']\right).$
We further assume that the $X_i$'s are a extension of a basis of $\mathfrak{h}'_\infty/\left(\mathfrak{h}'_\infty\cap \mathfrak{k}'_\infty\right)$, where $\mathfrak{h}'_\infty$ is the Lie algebra at infinity of $H'_n\left(\A\right)$. Finally for $\sigma\in\Aut\left(\CC\right)$ we set
\[\sigma\left([\Pi_\infty']\right)\coloneqq [\left({}^\sigma\Pi'\right)_\infty]\coloneqq \mathfrak{K}_\sigma^{-1}\left(\bigotimes_{v\in {\VV_\infty}}[\pi'_{\sigma^{-1}\circ v}]\right).\]
Let $K_f'$ be an open compact subgroup of $\GL_{2n}'\left(\A_f\right)$ which fixes $\Pi'$. A choice of such a generator $[\Pi_\infty']$ fixes an isomorphism of  $\mathcal{H}_{K_f'}^{\GL'_{2n}}$-module
\[\Theta_{\Pi'}\colon\mS_{\psi_f}^{\eta_f}\left(\Pi'_f\right)\iso H_c^{q_0}\left(\bsg,\mE\right)\left(\Pi'_f\right)\] defined by 
\[\mS_{\psi_f}^{\eta_f}\left(\Pi'\right)\iso\mS_{\psi_f}^{\eta_f}\left(\Pi'\right)\otimes H^{q_0}\left(\lig'_\infty,K_\infty',\mS_{\psi_\infty}^{\eta_\infty}\left(\Pi_\infty'\right)\otimes E_\mu^\lor\right)\iso \]\[\iso H^{q_0}\left(\lig'_\infty,K_\infty',\mS_{\psi}^{\eta}\left(\Pi'\right)\otimes E_\mu^\lor\right)\iso H^{q_0}\left(\lig'_\infty,K_\infty', \Pi'\otimes E_\mu^\lor\right)\iso \]\[\iso H_c^{q_0}\left(\bsg,\mE\right)\left(\Pi'_f\right), \]
where the first isomorphism is the one induced by $[\Pi_\infty]$ and the third isomorphism is the one induced by the inverse of $\Pi'\iso\mS_{\psi}^\eta\left(\Pi'\right)$.
\begin{theorem}\label{T:2strc}
For each $\sigma\in\Aut\left(\CC/\LL\right)$ there exists a complex number \[\omega\left({}^\sigma\Pi'_f\right)=\omega\left({}^\sigma\Pi'_f,[{}^\sigma\Pi_\infty']\right)\in \CC^\times\] such that $\Theta_{^\sigma\Pi',0}\coloneq\omega\left({}^\sigma\Pi'_f\right)^{-1}\Theta_{{}^\sigma\Pi'}$ is $\Aut\left(\CC\right)$ invariant, \emph{i.e.}
\[
\begin{tikzcd}
\mS_{\psi_f}^{\eta_f}\left(\Pi'_f\right) \arrow[rr,"\Theta_{\Pi',0}"]\arrow[d,"\sigma^*"] &&H_c^{q_0}\left(\bsg,\mE\right)\left(\Pi'_f\right)\arrow[d,"\sigma_{\GL_{2n}'}^*"] \\
\mS_{\psi_f}^{{}^\sigma\eta_f}\left({}^\sigma\Pi'_f\right) \arrow[rr,"\Theta_{{}^\sigma\Pi',0}"] && H_c^{q_0}\left(\bsg,{}^\sigma\mE\right)\left({}^\sigma\Pi'_f\right)
\end{tikzcd}\]
commutes. Hence, $\Theta_{\Pi',0}$ maps the $\QQ\left(\Pi',\eta\right)$-structure of $\mS_{\psi_f}^{\eta_f}\left(\Pi'\right)$ to the $\QQ\left(\Pi',\eta\right)$-structure of $H_c^{q_0}\left(\bsg,\mE\right)\left(\Pi'_f\right)$ and $\omega\left(\Pi'_f\right)$ is well defined up to multiplication by an element of $\QQ\left(\Pi',\eta\right)$.
\end{theorem}
\begin{proof}
Since \[\Theta_{\Pi'}\colon\mS_{\psi_f}^{\eta_f}\left(\Pi'_f\right)\iso H_c^{q_0}\left(\bsg,\mE\right)\left(\Pi'_f\right)\] is a morphism of irreducible $\GL_{2n}'\left(\A_f\right)$-modules it follows from Schur's Lemma that there exists a complex number $\omega\left(\Pi'_f\right)$ such that $\Theta_{\Pi',0}=\omega\left(\Pi_f\right)^{-1}\Theta_{\Pi'}$ maps one $\QQ\left(\Pi',\eta\right)$-structure onto the other, since rational structures are unique up to homothethies, see \cite[Proposition 3.1]{CloI}. Consider now the vector $\xi^0_{\Pi'_f}$ from \ref{L:nomvec} which generates $\mS_{\psi_f}^{{}^\sigma\eta_f}(\Pi'_f)$ as a $\GL'_{2n}\left(\A_f\right)$-module. After rescaling $\omega\left({}^\sigma\Pi'_f\right)$ by an element of $\QQ\left(\Pi',\eta\right)$ we have the equality
\[\sigma_{\GL_{2n}'}^*\left(\Theta_{\Pi',0}\left(\xi^0_{\Pi'_f}\right)\right)=\Theta_{{}^\sigma\Pi',0}\left(\xi^0_{{}^\sigma\Pi'_f}\right),\]
since both sides of the equation lie in the same $\QQ(\Pi',\eta)$-structure.
Thus we proved the assertion.
\end{proof}
\subsection{ Behavior under twisting}
As in \cite{GroRag} we discuss now how the above introduced periods behave under twisting with an algebraic character $\chi=\left(\widetilde{\chi}\circ \det'\right)\cdot \lvert\det'\lvert^b$, where $\widetilde{\chi}$ is a Hecke character of $\GL_1(\A)$ of finite order and $b\in\mathbb{Z}$.
In particular, for $v\in {\VV_\infty}$ the character $\widetilde{\chi}(\det')_v$ is the trivial one.
For the rest of this section fix such a  character $\chi$. The following is an easy consequence of the respective definitions.
\begin{lemma}
The representation $\Pi'\otimes \chi$ is cohomological with respect to $(E_{\mu+b})^\lor=E_\mu^\lor\otimes\bigotimes_{v\in S_\infty} \lvert\det_v'\lvert^{-b}$. If $\Pi'$ admits a Shalika model with respect to $\eta$ then $\Pi\otimes\chi$ admits one with respect to $\chi^2\eta$ and hence, $\omega\left(\Pi'_f\otimes\chi_f\right)$ is well defined up to a multiple of $\QQ\left(\Pi',\chi,\eta\right)^\times$.
\end{lemma}
We fix a splitting isomorphism $\chi\iso\chi_\infty\otimes\chi_f\iso\bigotimes_{v\in {\VV_\infty}}\lvert\det'_v\lvert^{b}\otimes \chi_f, $ which extends to a splitting isomorphism \[\Pi'\otimes \chi\iso \Pi_\infty'\otimes\chi_\infty\otimes\Pi'_f\otimes\chi_f\iso \bigotimes_{v\in {\VV_\infty}}\Pi_v'\otimes\lvert\det'_v\lvert^b\otimes\Pi'_f\otimes\chi_f.\]
Having already fixed the generator $[\Pi_v']$ we set \[[\Pi_v'\otimes \chi_v]\coloneqq \sum_{\underline{i}=\left(i_1,\ldots.,i_{q_{0,v}}\right)}\sum_{\alpha=1}^{\dim E_{\mu_v}^\lor} X_{\underline{i}}^*\otimes \lvert\det_v'\lvert^{-b}\xi_{v,\alpha,\underline{i}}\otimes e_\alpha^\lor,\]
\[[\Pi_\infty'\otimes\chi_\infty]\coloneqq \mathfrak{K}_\chi^{-1}\left(\bigotimes_{v\in {\VV_\infty}}[\Pi_v'\otimes\chi_v]\right),\]
where $\mathfrak{K}_\chi$ is defined as the map \[\mathfrak{K}_\chi\colon H^{q_0}\left(\lig_\infty',K_\infty', \mS_{\psi_\infty}^{\chi^2_\infty\eta_\infty}\left(\Pi_\infty'\otimes\chi_\infty\right)\otimes (E_{\mu+b})^\lor\right)\ra\]\[\ra  \bigotimes_{v\in {\VV_\infty}}H^{q_0}(\lig_\infty', K_\infty',\mS_{\psi_v}^{\chi^2_v\eta_v}\left(\Pi_v'\otimes\chi_v\right)\otimes (E_{\mu_v+b})^\lor),\]
corresponding to the splitting isomorphism $\Pi_\infty'\otimes\chi_\infty\iso \bigotimes_{v\in {\VV_\infty}}\Pi_v'\otimes\lvert\det'_v\lvert^{-b}.$
Note that for $[\Pi_v']$ as in (\ref{E:periods}) we have then \[[\Pi_v'\otimes \chi_v]\coloneqq \sum_{\underline{i}=\left(i_1,\ldots.,i_{q_{0,v}}\right)}\sum_{\alpha=1}^{\dim E_{\mu_v}^\lor} X_{\underline{i}}^*\otimes \xi_{v,\alpha,\underline{i}}\lvert\det'_v\lvert^{b}\otimes e_\alpha^\lor.\]
We quickly recall the definition of the Gauss sum of $\chi_f$, see \cite[VII, §7]{WeiB}.
For $v\in {\VV_f}$ let
$\mathfrak{c}_v$ be the conductor of $\chi_v$ and choose $c\in \GL_1\left(\A_f\right)$ such that for all $v\in {\VV_f}$ one has$\mathrm{ord}_v(c_v)=-\mathrm{ord}_v(\mathfrak{c})-\mathrm{ord}_v(\mathfrak{D}).$
Set 
\[\mathcal{G}\left(\chi_v,\psi_v,c_v\right)\coloneq\int_{\OO_v^\times}\chi_v\left(u_v\right)^{-1}\psi\left(c_v^{-1}u_v\right)\dv,\] 
where $\dv$ is a Haar measure on $\OO_v^\times$ normalized such that $\OO_v^\times$ has volume $1$.
Then $\mathcal{G}\left(\chi_v,\psi_v,c_v\right)$ is nonzero for all finite places and $1$ at the places where $\chi$ and $\psi$ is unramified, see \cite[Equation 1.22]{God}. Hence, the global Gauss sum 
\[\mathcal{G}\left(\chi_f,c\right)\coloneq\prod_{v\in {\VV_f}}\mathcal{G}\left(\chi_v,\psi_v,c_v\right)\]
is well-defined. From now on we fix one such $c$ and write $\mathcal{G}\left(\chi_f\right)\coloneqq \mathcal{G}\left(\chi_f,c\right)$. If $\chi=\left(\widetilde{\chi}\circ\det'\right)\cdot \lvert\det'\lvert^b$ is a character of $\GL_{2n}'$ as above, we set
$\mathcal{G}\left(\chi_f\right)\coloneqq \mathcal{G}\left(\widetilde{\chi}_f \lvert{}\cdot{}\lvert_f^b\right).$
The periods we defined in \ref{T:2strc} behave under twisting with such a character as follows.
\begin{theorem}\label{T:twist}
Let $\Pi'$ be a cuspidal irreducible representation of $\GL_{2n}'\left(\A\right)$ as in \ref{I:prop2}.
Let $\chi=\left(\widetilde{\chi}\circ\det'\right)\cdot \lvert\det'\lvert^b$ with $\widetilde{\chi}$ a Hecke character of $\GL_1(\A)$ of finite order and $b\in\mathbb{Z}$. For each $\sigma\in\Aut\left(\CC / \mathbb{Q}(\mu)\right)$ we have  \[\sigma\left({\frac{\omega\left(\Pi'_f\otimes\chi_f\right)}{ \mathcal{G}\left(\chi_f\right)^{nd}\omega\left(\Pi'_f\right)}}\right)=\frac{\omega\left({}^\sigma\Pi'_f\otimes {}^\sigma\chi_f\right)}{\mathcal{G}\left({}^\sigma\chi_f\right)^{nd}\omega\left({}^\sigma\Pi'_f\right)}.\]
\end{theorem}
In order to prove this we need three lemmata about the following maps.
The first map is
\[
S_{\chi}\colon\mS_{\psi}^{\eta}\left(\Pi'\right)\ra \mS_{\psi}^{\eta\chi^2}\left(\Pi'\otimes \chi\right),\,
\xi\mapsto \left(g\mapsto \chi\left(\det'\left(g\right)\right)\xi\left(g\right)\right),\]
which splits under our fixed splitting isomorphisms into the two maps
\[S_{\chi_f}\colon\mS_{\psi_f}^{\eta_f}\left(\Pi'_f\right)\rightarrow \mS_{\psi_f}^{\eta_f\chi_f^2}\left(\Pi'_f\otimes \chi_f\right),\,
\xi_f\mapsto \left(g_f\mapsto \chi_f\left(\det'\left(g_f\right)\right)\xi_f\left(g_f\right)\right)\]
and
\[S_{\chi_\infty}\colon\mS_{\psi_\infty}^{\eta_\infty}\left(\Pi_\infty'\right)\rightarrow \mS_{\psi_\infty}^{\eta_\infty\chi_\infty^2}\left(\Pi_\infty'\otimes \chi_f\right),\,
\xi_\infty\mapsto \left(g_\infty\mapsto \chi_\infty\left(\det'\left( g_\infty\right)\right)\xi_\infty\left(g_\infty\right)\right).\]
Moreover, we define\[
    A_\chi\colon\Pi'\rightarrow \Pi'\otimes \chi\\
\phi\mapsto \left(g\mapsto\chi\left(\det'\left( g\right)\right)\phi\left(g\right)\right),\]
where we consider $\phi\in\Pi'$ as a cusp form.
\begin{lemma}\label{L:1}
With $S_{\chi_f}$ as above,
\[\sigma^*\circ S_{\chi_f}=\sigma\left(\left(\chi_f\left(t_\sigma\right)\right)^{-nd}\right)S_{{}^\sigma\chi_f}\circ \sigma^*=\left(\frac{\sigma\left(\mathcal{G}\left(\chi_f\right)\right)}{\mathcal{G}\left({}^\sigma \chi_f\right)}\right)^{-nd}S_{{}^\sigma\chi_f}\circ \sigma^*.\]
\end{lemma}
Let  $\mathrm{1}_{ E_\mu^\lor}$ be the identity map on $E_\mu^\lor$ and let $\left(A_{\chi}\otimes \mathrm{1}_{ E_\mu^\lor}\right)^*$ be the induced map on cohomology \[\left(A_{\chi}\otimes \mathrm{1}_{ E_\mu^\lor}\right)^*\colon H_c^{q_0}\left(\bsg,\mE\right)\left(\Pi'_f\right)\ra H_c^{q_0}\left(\bsg,\mE\right)\left(\Pi'_f\otimes \chi_f\right).\]
The proof of the following can be done exactly in the same manner as in \cite[Proposition 2.3.7]{RagSha}.
\begin{lemma}[{\cite[Proposition 2.3.7]{RagSha}}]\label{L:2}
With $A_\chi$ as above,
\[\left(A_{\chi}\otimes \mathrm{1}_{ E_\mu^\lor}\right)^*\circ \Theta_{\Pi'}=\Theta_{\Pi'\otimes \chi}\circ \mS_{\chi_f} .\]
\end{lemma}
\begin{lemma}[{\cite[Proposition 2.3.6]{RagSha}}]\label{L:3}
For any $\sigma\in \Aut\left(\CC/\QQ(\mu)\right)$ we have
\[\sigma_{\GL_{2n}'}^*\circ\left(A_\chi\otimes \mathrm{1}_{E_\mu^\lor}\right)^*=\left(A_{{}^\sigma\chi}\otimes \mathrm{1}_{{}^\sigma E_\mu^\lor}\right)^*\circ \sigma_{\GL_{2n}'}^*\]
\end{lemma}
\begin{proof}
    The proof is exactly the same as the one of \cite[Proposition 2.3.6]{RagSha}. Note that in our case it even simplifies a bit since $\sigma(\chi)={}^\sigma\chi$. Indeed, both $^\sigma\left(\widetilde{\chi}\circ\mathrm{Nrd}\right)$ and $\widetilde{\chi}\circ\mathrm{Nrd}$ are trivial at infinity by \ref{T:dimcoh} and $\lvert\det'\lvert^b={}^\sigma\lvert\det'\lvert^b=\sigma(\lvert\det'\lvert^b)$ for all $\sigma\in \Aut(\CC)$.
\end{proof}
We will first show how \ref{T:twist} follows from the above lemmata.
\begin{proof}[Proof of \ref{T:twist}]
We compute $\left(A_{{}^\sigma\chi}\otimes \mathrm{1}_{{}^\sigma E_\mu^\lor}\right)^{*} \circ \sigma_{\GL_{2n}'}^*\circ \Theta_\Pi$ in two different ways.
On the one hand
\[\left(A_{{}^\sigma\chi}\otimes \mathrm{1}_{{}^\sigma E_\mu^\lor}\right)^{*} \circ\sigma_{\GL_{2n}'}^*\circ \Theta_{\Pi'}\stackrel{(\ref{T:2strc})}{=}  \left(\frac{\sigma\left(\omega\left(\Pi'_f\right)\right)}{\omega\left({}^\sigma\Pi'_f\right)} \right)\left(A_{{}^\sigma\chi}\otimes \mathrm{1}_{{}^\sigma E_\mu^\lor}\right)^{*} \circ \Theta_{{}^\sigma\Pi'}\circ \sigma^*\stackrel{(\ref{L:2})}{=}\]
\[=\left(\frac{\sigma\left(\omega\left(\Pi'_f\right)\right)}{\omega\left({}^\sigma\Pi'_f\right)} \right)\Theta_{{}^\sigma\Pi'\otimes {}^\sigma\chi}\circ \mS_{{}^\sigma\chi_f}\circ \sigma^*.\]
But on the other hand, we see
\[\left(A_{{}^\sigma\chi}\otimes \mathrm{1}_{{}^\sigma E_\mu^\lor}\right)^{*} \circ \sigma_{\GL_{2n}'}^*\circ \Theta_{\Pi'} \stackrel{(\ref{L:3})}{=}\sigma_{\GL_{2n}'}^*\circ\left(A_{\chi}\otimes \mathrm{1}_{ E_\mu^\lor}\right)^{*} \circ  \Theta_{\Pi'}\stackrel{(\ref{L:2})}{=}\]\[=\sigma_{\GL_{2n}'}^*\circ\Theta_{\Pi'\otimes \chi}\circ S_{\chi_f}\stackrel{(\ref{T:2strc})}{=}\left(\frac{\sigma\left(\omega\left(\Pi'_f\otimes\chi_f\right)\right)}{\omega\left({}^\sigma\Pi'_f\otimes {}^\sigma\chi_f\right)}\right)\Theta_{{}^\sigma \Pi'\otimes {}^\sigma \chi}\circ \sigma^*\circ S_{\chi_f}\stackrel{(\ref{L:1})}{=}\]\[=\left(\frac{\sigma\left(\omega\left(\Pi'_f\otimes\chi_f\right)\right)}{\omega\left({}^\sigma\Pi'_f\otimes {}^\sigma\chi_f\right)}\right)\left(\frac{\sigma\left(\mathcal{G}\left(\chi_f\right)\right)}{\mathcal{G}\left({}^\sigma \chi_f\right)}\right)^{-nd}\Theta_{{}^\sigma \Pi'\otimes {}^\sigma \chi}\circ S_{{}^\sigma\chi_f}\circ \sigma^*.\]
Hence, the desired equality follows.
\end{proof}
\begin{proof}[Proof of \ref{L:1}]
The first equality follows by inserting the definition of $\sigma^*$ and noticing that the determinant of $\det'(\textbf{t}_\sigma)=t_\sigma^{nd}$.
The second equality follows from \cite[Theorem 2.4.3]{BusFro}, which states $\sigma\left(\mathcal{G}\left(\chi_f\right)\right)=\sigma\left(\chi_f\left(t_\sigma\right)\right)\mathcal{G}\left({}^\sigma \chi_f\right).$
\end{proof}
\section{Critical values of of \texorpdfstring{$L$}{L}-functions and their cohomological interpretation}
Throughout this section, we assume $ n=2$, $d=1$.
Let $\Pi'$ be a cuspidal irreducible cohomological representation of $\GL_{2}'\left(\A\right)$ as in \ref{I:prop2}
and consider the standard global $L$-function $L(s,\Pi')$ of $\Pi'$. Recall that a critical point of $L\left(s,\Pi'\right)$ is in our case a point $s_0\in\frac{1}{2}+\mathbb{Z}$ such that both $L\left(s,\Pi_\infty'\right)$ and $L\left(1-s,\Pi_\infty'^\lor\right)$ are holomorphic at $s_0$. We further assume that $\jl\left(\Pi'\right)$ is residual of the form $\jl\left(\Pi'\right)=\mw\left(\Sigma,2\right)$ for some cuspidal irreducible representation $\Sigma$ of $\GL_{2}$. To calculate the critical values of the $L$-function it suffices to consider the meromorphic contribution of the local $L$-factors from the infinite places. Let $\mu'$ be the highest weight such that $\Sigma$ is cohomological with respect to $E_{\mu'}^\lor$.
 We can compute $L$-factor of $L\left(s,\Pi_\infty'\right)$ by \ref{L:littlelem}, \cite[Theorem 5.2]{GroRag2} and \cite[Theorem 19.1.(b)]{BadRen} as follows.
 We showed in the proof of \ref{L:littlelem} that for $v$ an infinite place and $\mu_v'=(\mu_{v,1}',\ldots,\mu_{v,4}')$,
 $\Pi_v'$ has to be the unique quotient of $D(l_{1,v},-w_{1,v})\times D(l_{2,v},-w_{2,v}),$
where $l_{1,v}=\mu_{v,1}'-\mu_{v,4}'+2,\,l_{2,v}=\mu_{v,2}'-\mu_{v,3}'+2$ and $w_{1,v}=\mu_{v,1}'+\mu_{v,4}'+1,\, w_{2,v}=\mu_{v,2}'+\mu_{v,3}'-1$.
Moreover, from the proof of \cite[Theorem 5.2]{GroRag2}, see also the proof of \ref{L:littlelem}, it follows that $\mu_{v,1}'=\mu_{v,2}'\ge\mu_{v,3}'=\mu_{v,4}'$.
Thus the $L$-function $L(s,\Pi_\infty')$ is up to a non-zero scalar equal to
\[\prod_{v\in {\VV_\infty}}\prod_{i=1}^2\Gamma\left(s+\frac{w_{i,v}+l_{i,v}}{2}\right)\]
and the one of $L\left(1-s,\Pi_\infty'^\lor\right)$ is of the form
\[\prod_{v\in {\VV_\infty}}\prod_{i=1}^2\Gamma\left(1-s-\frac{-w_{i,v}+l_{i,v}}{2}\right).\]
Recall that the poles of the Gamma function lie precisely on the non-positive integers and it is non-vanishing everywhere else. 
Thus the critical points are precisely those $\frac{1}{2}+m,m\in\ZZ$ such that
\[\CCrit\left(\Pi'\right)=\{\frac{1}{2}+m:-\mu_{v,2}'\le m\le -\mu_{v,3}',v\in\VV_\infty\}.\]
Note that by \cite[Corollary 13.7, Theorem 19.1(b)]{BadRen} the above set is also the set of critical points of $L(s+\frac{1}{2},\Sigma_\infty)L(s-\frac{1}{2},\Sigma_\infty)$. For an explicit example we refer to \cite[§ 6.11]{GroRag2}.

Following \cite{GroRag} we define a map $\mathcal{T}^*$.
\begin{prop}[{\cite[Proposition 6.3.1]{GroRag}}]
Let $\Pi'$ be an irreducible representation as in \ref{I:prop2} and assume moreover that $n=1$. 
Assume $\frac{1}{2}$ is a critical point of $L(s,\Pi')$. Denote by $w_v$ the weight such that
$E_{\mu_v}\cong E_{\mu_v}^\lor\otimes \det'^{w_v}$ for each $v\in S_\infty$. Let $E_{\left(0,-w_v\right)}$ be the representation $\mathbf{1}\otimes \det^{-w_v}$ of $H'\left(\CC\right)=\GL_{2}\left(\CC\right)\times \GL_{2}\left(\CC\right)$. Then
\[\dim_\CC \mathrm{Hom}_{H'\left(\CC\right)}\left(E_{\mu_v}^\lor,E_{\left(0,-w_v\right)}\right)=1\]
for all $v\in S_\infty$.
\end{prop}
We then let $E_{0,-w}\coloneqq \bigotimes_{v\in {\VV_\infty}}E_{\left(0,-w_v\right)}$ and write $\mathcal{E}_{\left(0,-w\right)}$ for the corresponding locally constant sheaf of $\SH$.
Since $\frac{1}{2}$ is a critical point and since $\QQ\left(\mu\right)$ contains the splitting field of $\DD$, \cite[Lemma 4.8]{JiaSunTia} shows that there exists a map $\mathcal{T}=\bigotimes_{v\in S_\infty}\mathcal{T}_{\circ v}$ in above space which is defined over $\QQ(\mu)$ and we fix a choice of such a map.
Lifting this map to cohomology we obtain a morphism
\[\mathcal{T}^*\colon H^*_c\left(\SH,\mE\right)\rightarrow H^*_c\left(\SH,\mathcal{E}_{\left(0,-w\right)}\right).\]
For $\sigma\in\Aut\left(\CC\right)$ we define the twist of $\mathcal{T}$ as
\[{}^\sigma\mathcal{T}=\bigotimes_{v\in S_\infty}\mathcal{T}_{\sigma^{-1}\circ v}\]
and denote the corresponding morphism on the cohomology by ${}^\sigma\mathcal{T}^*$.
Since $\mathcal{T}$ is defined over $\QQ(\mu)$, we therefore obtain for all $\sigma\in\Aut(\CC/\QQ(\mu))$ that ${}^\sigma\mathcal{T}^*=\mathcal{T}^*$.
We then have the following commutative diagram for all $\sigma\in\Aut(\CC/\QQ(\mu))$.
\begin{equation}\label{E:sumdia}
\begin{tikzcd}
H^*_c\left(\SH,\mE\right)\arrow[d,"\sigma_{H'_n}^*"]\arrow[rr,"\mathcal{T}^*"]&&H^*_c\left(\SH,\mathcal{E}_{\left(0,-w\right)}\right)\arrow[d,"\sigma_{H'_n}^*"]\\
H^*_c\left(\SH,{}^\sigma\mE\right)\arrow[rr,"{}^\sigma\mathcal{T}^*"]&&H^*_c\left(\SH,{}^\sigma\mathcal{E}_{\left(0,-w\right)}\right)
\end{tikzcd}
\end{equation}
The next step consists of translating the computation of a critical point of $L(s,\Pi')$ into an instance of Poincar\'e duality. But in order to apply 
 Poincar\'e duality, the highest or lowest degree in which $H^*\left(\lig'_\infty,K_\infty',\Pi'\otimes E_\mu^\lor\right)$ vanishes has to equal the dimension $\dim_\RR\SH$, which implies $n=1,d=2$ and $k=2$. Indeed, \ref{L:littlelem} implies that
\[{r\left(\left(nd\right)^2-nd-1\right)=r\left(\left(nd\right)^2-nd-\frac{nd}{2}\left(k-1\right)\right)}\text{ or }\]\[{r\left(\left(nd\right)^2-nd-1\right)=r\left(\left(nd\right)^2-nd+\frac{nd}{2}\left(k-1\right)+1\right)}\]
and only the first equation can be satisfied, which leads to the above restriction.
\subsection{ Poincaré duality}\label{I:prop3} We therefore let $\Pi'$ be a representation as in \ref{I:prop2} and assume moreover that $\DD$ is quaternion and $n=1$. By \ref{T:L} the respective conditions on the $\Aut(\CC)$-orbit on $\Pi'$ hold unconditionally in this case and we set $q_0\coloneqq r$, which is the real dimension of $\Sh$ and the lowest degree in which $H^*(\lig'_\infty,K_\infty',\Pi'\otimes E_\mu^\lor)$ does not vanish. 

If $\Pi'$ is as above and $\frac{1}{2}$ is a critical point of $L(s,\Pi')$, 
we choose $K_f'$ small enough so that $\iota^{-1}\left(K_f'\right)$ can be written as a product $K_{f,1}\times K_{f,2}$, $\eta_f$ is trivial on $\det'\left(K_{f,2}\right)$ and $K_f'$ fixes $\Pi'_f$. 
Let $\mathcal{C}$ be the set of connected components of $\Sh$. Using \cite[Theorem 5.1]{BorII} we see that $\mathcal{C}$ is finite and each $C\in\mathcal{C}$ is a quotient of $H'_{1,\infty}/\left(K_\infty'\cap H_{1,\infty}'\right)$ by a discrete subgroup of $H_1'\left(\KK\right)$. Recall that the $Y_i$'s from in \ref{S:5.3} give a basis of $\mathfrak{h}'_\infty/\left(\mathfrak{h}'_\infty\cap \mathfrak{k}'_\infty\right)$. Thus, they give an orientation on  $H'_{1,\infty}/\left(K_\infty'\cap H_{1,\infty}'\right)$, since $\mathfrak{h}'_\infty/\left(\mathfrak{h}'_\infty\cap \mathfrak{k}'_\infty\right)$ is parallelizable and therefore on each $C\in \mathcal{C}$ we can now consider $\mathbf{1}\times \eta^{-1}$ as a global section of $\mathcal{E}_{\left(0,-w\right)}$ and denote the corresponding cohomology class as 
\[[\eta]\in H^{q_0}_c\left(\Sh,\mathcal{E}_{\left(0,-w\right)}\right).\]

Poincar\'e duality on each connected component of $\Sh$ gives rise to a surjective map 
\[H_c^{q_0}\left(\Sh,\mathcal{E}_{\left(0,-w\right)}\right)\rightarrow \CC,\,\theta \mapsto {\int_\Sh} \theta\wedge[\eta]\coloneqq \sum_{C\in \mathcal{C}}{\int_C} \theta\wedge[\eta].\]
The following is an immediate consequence of the equivariance properties we proved in the above sections.
\begin{lemma}\label{L:indeed}
    This map commutes with twisting by an automorphism $\sigma\in\Aut(\CC)$, \emph{i.e.}
\[\sigma\left({\int_\Sh} \theta\wedge[\eta]\right)={\int_\Sh} \sigma_{H'_1}^*\left(\theta\right)\wedge[{}^\sigma\eta].\]
\end{lemma}
To proceed we need the following non-vanishing result.
Recall for $v\in {\VV_\infty}$ \[[\Pi_v']=\sum_{\underline{i}=\left(i_1,\ldots.,i_{q_{0,v}}\right)}\sum_{\alpha=1}^{\dim E_{\mu_v}^\lor} X_{\underline{i}}^*\otimes \xi_{v,\alpha,\underline{i}}\otimes e_\alpha^\lor. \]
For each $\underline{i}$ write
\[\iota^*\left(X_{\underline{i}}\right)=s\left(\underline{i}\right)Y_1^*\wedge\ldots\wedge Y_{q_0}^*,\]
where $s\left(\underline{i}\right)$ is some complex number.
If $\frac{1}{2}\in \CCrit(\Pi')$, we know that $\mathcal{T}$ exists
and $\zeta_v(s,{}\cdot{})=P(s,{}\cdot{})L(s,\Pi_v')$ for $v\in {\VV_\infty}$ by \ref{T:connL}.
Since $\frac{1}{2}$ is critical, we know that $\zeta_v(\frac{1}{2},{}\cdot{})$ is well defined for all vectors in the Shalika model.
We thus can set
\[c\left(\Pi_v'\right)\coloneqq \sum_{\underline{i}}\sum_{\alpha=1}^{\dim E_{\mu_v}} s\left(\underline{i}\right)\mathcal{T}\left(e_\alpha^\lor\right)\zeta_v\left(\frac{1}{2},\xi_{v,\alpha,\underline{i}}\right)\]
and $c\left(\Pi_\infty'\right)\coloneqq \prod_{v\in {\VV_\infty}}c\left(\Pi_v'\right)$.

For $s=\frac{1}{2}+m\in \CCrit\left(\Pi'\right)$, the $L$-function of $\Pi'\otimes\lvert\det'\lvert^m$ has critical point $\frac{1}{2}$. We set
$\Pi'\left(m\right)=\Pi'\otimes \lvert\det\lvert^m$ and 
$c\left(\Pi_\infty',m\right)\coloneqq c\left(\Pi'\left(m\right)_\infty\right).$
\begin{theorem}[{\cite[ Theorem A.3]{SunII}}]\label{T:Sun}
 For $\Pi'$ as in \ref{I:prop3} and $s=\frac{1}{2}+m\in \CCrit\left(\Pi'\right)$, the expression $c\left(\Pi_\infty',m\right)$ does not vanish.
\end{theorem}
\begin{proof}
We will assume without loss of generality that $s=\frac{1 }{2}$ is critical.
Since $c\left(\Pi_\infty'\right)=\prod_{v\in {\VV_\infty}}c\left(\Pi_v'\right)$ we fix a place $v\in {\VV_\infty}$. We set $H\coloneq\GL_{1}\left(\HH\right)\times \GL_1\left(\HH\right)$ and $G\coloneq\GL_{2}\left(\HH\right)$ with maximal compact subgroup $K'_H$, respectively, $K'$. Since $\frac{1}{2}$ is critical, it follows from \ref{T:global} that the local zeta integral at $v$ gives a functional
\[\zeta_v\left(\frac{1}{2},{}\cdot{}\right)\in \mathrm{Hom}_H\left(\Pi_v',\chi\right),\]
where $\chi=\mathrm{1}_{\GL_{1}\left(\HH\right)}\otimes \det'^{w_v}$. It is nonzero, since 
$\zeta_v(s,{}\cdot{})=P(s,{}\cdot{})L(s,\Pi_v')$ and there exists by \ref{T:connL} $\xi_{\Pi',v}$ such that $P(s,\xi_{\Pi',v})=1$. Since the $L$-factors at infinity are products of Gamma-functions and non-vanishing holomorphic functions, $\zeta_v(s,\xi_{\Pi',v})$ also never vanishes. Thus $\zeta_v(s,{}\cdot{})$ never vanishes and hence $\zeta_v(\frac{1}{2},{}\cdot{})$ is non-zero.
Let $j_2$ be the inclusion
$j_2\colon\mathfrak{h}'/\mathfrak{k}'_H\hookrightarrow \lig'/\mathfrak{k}'$
and consider now the map
\begin{align*}
   \mathrm{Hom}\left(\lig'/\mathfrak{k}',\Pi_v'\otimes E_{\mu_v}^\lor\right)&\rightarrow \mathrm{Hom}\left(\mathfrak{h}'/\mathfrak{k}'_H,\chi\otimes E_{\left(0,-w_v\right)}\right) \\
   f&\longmapsto \left(\zeta_v\left(\frac{1}{2},{}\cdot{}\right)\otimes \mathcal{T}_v\right)\circ f\circ j_2
\end{align*}
By \cite[Theorem A.3]{SunII} the induced map
\[c\colon H^1\left(\lig'_\infty,K',\Pi_v'\otimes E_{\mu_v}^\lor\right)\rightarrow H^1\left(\mathfrak{h}', K'_H, \chi\otimes E_{\left(0,-w_v\right)}\right)\]
does not vanish on the one dimensional space $H^1\left(\lig'_\infty,K',\Pi_v'\otimes E_{\mu_v}^\lor\right)$. Since it is generated by $[\Pi_v']$ we conclude that $c\left(\Pi_v'\right)\neq 0$.
\end{proof}
 We set $c\left(\Pi_\infty,m\right)^{-1}\coloneqq \omega\left(\Pi_\infty,m\right).$
\begin{remark}The proof of \cite[Theorem A.3]{SunII} relies crucially on the numerical coincidence, \emph{i.e.} that either the lowest or highest nonvanishing degree of the $\left(\lig'_\infty,K'_\infty\right)$-cohomology 
$H^*\left(\lig'_\infty,K_\infty', \Pi_\infty'\otimes E_\mu^\lor\right)$
is $\dim_\RR\Sh$.\end{remark}
\begin{theorem}\label{T:pre}
Let $\Pi'$ be a cuspidal irreducible representation of $\GL_{2}'\left(\A\right)$ as in \ref{I:prop3}. Assume $s=\frac{1}{2}\in \CCrit\left(\Pi'\right)$ and let $\xi^0_{\Pi'_f}$ be the vector of \ref{L:nomvec}.
Then \[\int_{\mathbf{S}_{K_f'}^{H'_1}} \mathcal{T}^*\iota^*\Theta_{\Pi',0}\left(\xi_{\Pi'_f}^0\right)\wedge[\eta]=\frac{L\left(\frac{1}{2},\Pi'_f\right)\prod_{v\in S_{\Pi'_f,\psi}}P\left(\frac{1}{2},\xi^0_{\Pi',v}\right)}{\omega\left(\Pi'_f\right)\omega\left(\Pi_\infty'\right) \mathrm{vol}\left(\iota^{-1}\left(K_f'\right)\right)} \]
for every small enough open compact subgroup $K_f'$ of $\GL_{2}'\left(\A_f\right)$.
\end{theorem}
\begin{proof}
The proof of this theorem can be carried out in the same way as the proof of \cite[Theorem 6.7.1]{GroRag}. We only include it for completeness. Recall from \ref{S:2.3} that $c=\mathrm{vol}\left(\KK^\times\bs \A^*/\RR_{>0}^r\right).$ We choose $K_f'$ such that it fixes $\xi_{\Pi'_f}^0$.
Plugging $\xi_{\Pi'_f}^0$ in the definition of the terms of the integral and using the $K_f'$-invariance of $\xi^0_{\Pi'_f}$ we obtain the following identity.
\[\int_{\mathbf{S}_{K_f'}^{H'_1}} \mathcal{T}^*\iota^*\Theta_{\Pi',0}\left(\xi_{\Pi'_f}^0\right)\wedge[\eta]=\]\[= \mathrm{vol}\left(\iota^{-1}\left(K_f'\right)\right)^{-1}c^{-1}\omega\left(\Pi'_f\right)^{-1}\sum_{\underline{i},\alpha}s\left(\underline{i}\right)\mathcal{T}\left(e_\alpha\right) \int_{H'_1\left(\KK\right)\bs H'_1\left(\A\right)/\RR_+^d}\eta\restr{\phi_{\underline{i},\alpha}^0}{H'_1\left(\A\right)}\, \mathrm{d}h,\]
where \[\phi^0_{\underline{i},\alpha}\coloneqq \left(\mS_\psi^\eta\right)^{-1}\left(\bigotimes_{v\in {\VV_\infty} }\xi_{v,\underline{i},\alpha}\otimes \xi_{\Pi'_f}^0\right).\]
We compute now the latter integral over $H'_1\left(\KK\right)\bs H'_1\left(\A\right)/\RR_+^d$ for fixed $\underline{i}$ and $\alpha$. Again plugging in the definitions yields
\[\int_{H'_1\left(\KK\right)\bs H'_1\left(\A\right)/\RR_+^d}[\eta]\restr{\phi_{\underline{i},\alpha}^0}{H'_1\left(\A\right)}\ddd h=\int_{Z'\left(\A\right)H'_1\left(\KK\right)\bs H'_1\left(\A\right)}\int_{Z'\left(\KK\right)\bs Z'\left(\A\right)/\RR_+^d}\]\[\left(\phi^0_{\underline{i},\alpha}\left(\bpm h_1&0\\0&h_2\epm z\right)\eta^{-1}\left(\det'\left(h_2z\right)\right)dz\right)\ddd h_1\ddd h_2.\]
We can now pull the $z=\mathrm{diag}\left(a,a\right)$-contribution out of $\phi_{\underline{i},\alpha}^0$ and $\eta^{-1}\left(\det'\right)$, which yields a factor of $\omega\left(z\right)\eta\left(\det'(a)\right)^{-1}=1$ and hence, the integral simplifies to
\[c\int_{Z'\left(\A\right)H'_1\left(\KK\right)\bs H'_1\left(\A\right)}\phi^0_{\underline{i},\alpha}\left(\bpm h_1&0\\0&h_2\epm z\right)\eta^{-1}\left(\det' (h_2)\right)\ddd h_1\ddd h_2.\]
Recall the equality of \ref{T:global} and the properties of the special vector $\xi^0_{\Pi'_f}$
\[\int_{Z'\left(\A\right)H'_1\left(\KK\right)\bs H'_1\left(\A\right)}\phi^0_{\underline{i},\alpha}\left(\bpm h_1&0\\0&h_2\epm z\right)\left\lvert\frac{\det' (h_1)}{\det' (h_2)}\right\lvert^{s-\frac{1}{2}}\eta^{-1}\left(\det' (h_2)\right)\ddd h_1\ddd h_2=\]\[=\zeta_\infty\left(s,\xi_{\infty,\underline{i},\alpha}^0\right)\zeta_f\left(s,\xi_{\Pi'_f}^0\right)=\zeta_\infty\left(s,\xi_{\infty,\underline{i},\alpha}^0\right)L(s,\Pi'_f)\prod_{v\in S_{\Pi'_f,\psi}}P\left(\frac{1}{2},\xi^0_{\Pi',v}\right)\]
for $\re >>0$. But the integral converges absolutely for all $s$ hence, we obtain the equality for all $s$. Recall that $L(s,\Pi)$ is an entire function and hence, $L\left(\frac{1}{2},\Pi'_f\right)\in \CC$ since $s=\frac{1}{2}$ is critical.
Therefore, 
\[\int_{Z'\left(\A\right)H'_1\left(\KK\right)\bs H'_1\left(\A\right)}\phi^0_{\underline{i},\alpha}\left(\bpm h_1&0\\0&h_2\epm z\right)\eta^{-1}\left(\det' (h_2)\right)\ddd h_1\ddd h_2=\]\[=\zeta_\infty\left(\frac{1}{2},\xi_{\infty,\underline{i},\alpha}^0\right)L(\frac{1}{2},\Pi'_f)\prod_{v\in S_{\Pi'_f,\psi}}P\left(\frac{1}{2},\xi^0_{\Pi',v}\right).\]
Plugging this in the above sum over $\underline{i}$ and $\alpha$, we obtain the desired identity.
\end{proof}
We are now ready to prove our analog of \cite[Theorem 7.1.2]{GroRag}.
\begin{theorem}\label{T:manny}
Let $\DD$ be a quaternion algebra and let $\Pi'$ be a cuspidal irreducible cohomological representation of $\GL_{2}'\left(\A\right)$ which admits a Shalika model with respect to $\eta$. Assume that either $\eta$ is trivial or the $\Aut(\CC)$-orbit of $\Pi'$ admits a unique local Shalika model with respect to $\eta$. Let $\mu$ be the highest weight such that $\Pi'$ is cohomological with respect to $E_\mu^\lor$ and assume that $\jl\left(\Pi'\right)$ is residual, \emph{i.e.} $\jl\left(\Pi'\right)=\mw\left(\Sigma,2\right)$ for $\Sigma$ a cuspidal irreducible cohomological representation of $\GL_2\left(\A\right)$. Let moreover $\chi=\widetilde{\chi}\circ\det'$, where $\widetilde{\chi}$ is a Hecke-character of $\GL_1(\A)$ of finite order.
Then, for $\frac{1}{2}+m\in \CCrit\left(\Pi'\right)$,
\[\frac{L\left(\frac{1}{2}+m,\Pi'_f\otimes\chi_f\right)}{\omega\left(\Pi'_f\right)\mathcal{G}\left(\chi_f\right)^{4}\omega\left(\Pi_\infty',m\right)}\in \QQ\left(\Pi',\chi,\eta\right).\]
\end{theorem}
\begin{proof}
Again the proof can be adapted from \cite{GroRag} to the situation at hand. To show the claim it is enough to show that the ratio stays invariant under all $\sigma\in \Aut(\CC/\QQ(\Pi',\chi,\eta))$.
First assume that $m=0$, $\frac{1}{2}\in \CCrit(\Pi')$ and $\chi=1$.
We are going to compute 
\begin{equation}\label{E:2w}\sigma\left(\int_{\mathbf{S}_{K_f'}^{H'_1}} \mathcal{T}^*\iota^*\Theta_{\Pi',0}\left(\xi_{\Pi'_f}^0\right)\wedge[\eta]\right)\end{equation} for some $\sigma\in\Aut(\CC/\QQ(\Pi',\chi,\eta))$ in two different ways, where $K_f'$ is a sufficiently small open compact subgroup of $\GL(\A_f)$. On the one hand, (\ref{E:2w}) equals by \ref{T:pre} and \ref{L:nomvec} to
\[\sigma\left(\frac{L\left(\frac{1}{2},\Pi'_f\right)\prod_{v\in S_{\Pi'_f,\psi}}P\left(\frac{1}{2},\xi^0_{\Pi',v}\right)}{\omega\left(\Pi'_f\right)\omega\left(\Pi_\infty'\right) \mathrm{vol}\left(\iota^{-1}\left(K_f'\right)\right)} \right)=\]
\[=\sigma\left(\frac{L\left(\frac{1}{2},\Pi'_f\right)}{\omega\left(\Pi'_f\right)\omega\left(\Pi_\infty'\right)}\right)\cdot \frac{\prod_{v\in S_{\Pi'_f,\psi}}P\left(\frac{1}{2},\xi^0_{\Pi',v}\right)}{\mathrm{vol}\left(\iota^{-1}\left(K_f'\right)\right)},\] where we used that $\mathrm{vol}(\iota^{-1}(K_f'))\in\QQ^\times$.
On the other hand , by pulling $\sigma$ into the integral (\ref{E:2w}), we compute 
\[\sigma\left(\int_{\mathbf{S}_{K_f'}^{H'_1}} \mathcal{T}^*\iota^*\Theta_{\Pi',0}\left(\xi_{\Pi'_f}^0\right)\wedge[\eta]\right) \stackrel{(\ref{L:indeed})}{=}\int_{\mathbf{S}_{K_f'}^{H'_1}} \sigma_{H_1'}^*(\mathcal{T}^*\iota^*\Theta_{\Pi',0}\left(\xi_{\Pi'_f}^0\right))\wedge[{}^\sigma\eta]  \stackrel{(\ref{E:sumdia})}{=}\]\[=\int_{\mathbf{S}_{K_f'}^{H'_1}} {}^\sigma \mathcal{T}^*\sigma_{H_1'}^*(\iota^*\Theta_{\Pi',0}\left(\xi_{\Pi'_f}^0\right))\wedge[{}^\sigma\eta]\stackrel{(\ref{E:sumdia2})}{=}\int_{\mathbf{S}_{K_f'}^{H'_1}} {}^\sigma \mathcal{T}^*\iota^*\sigma_{\Sg}^*(\Theta_{\Pi',0}\left(\xi_{\Pi'_f}^0\right))\wedge[{}^\sigma\eta]\stackrel{(\ref{T:2strc})}{=}\]\[=\int_{\mathbf{S}_{K_f'}^{H'_1}}{}^\sigma\mathcal{T}^*\iota^*\Theta_{{}^\sigma\Pi',0}\left( \xi_{{}^\sigma\Pi'_f}^0\right)\wedge[{}^\sigma\eta].\]
By \ref{T:L} ${}^\sigma\Pi'$ admits a Shalika model with respect to ${}^\sigma\eta$ and is cohomological and therefore the last integral equals by \ref{T:pre}
\[\frac{L\left(\frac{1}{2},{}^\sigma\Pi'_f\right)}{\omega\left({}^\sigma\Pi'_f\right)\omega\left({}^\sigma\Pi_\infty'\right)}\cdot \frac{\prod_{v\in S_{\Pi'_f,\psi}}P\left(\frac{1}{2},\xi^0_{\Pi',v}\right)}{\mathrm{vol}\left(\iota^{-1}\left(K_f'\right)\right))},\]
which proves the assertion.

If $\frac{1}{2}+m$ is an arbitrary critical point and $\chi=1$, consider $\Pi'\left(m\right)=\Pi'\otimes \lvert\det'\lvert^m$ and hence, $\frac{1}{2}$ is a critical point for this twisted representation. Recall also that $\mathcal{G}\left(\lvert\det'\lvert_f^n\right)=1$, thus \ref{T:twist} proves the claim. Finally, to obtain the result for $\frac{1}{2}+m$ is an arbitrary critical point and $\chi\neq 1$, we apply \ref{T:twist} again and note that $\Pi_\infty'= \Pi_\infty'\otimes\chi_\infty$, since $\chi$ is of finite order.
\end{proof}
Recall that since $\jl(\Pi')=\mw(\Sigma,2)$, the partial $L$-functions of the discrete series representations $\Pi'$ and $\mw\left(\Sigma,2\right)$ coincide. We therefore obtain a new result on critical values for residual representations of $\GL_4$. 
 Note that for any place $v\in {\VV_f}$,  \[L\left(\frac{1}{2}+m,\Pi_v'\otimes \chi_v\right)\in \QQ\left(\Pi',\chi\right)\] by \ref{L:compat} and by \ref{T:G-T} $\Pi'$ admits a Shalika model with respect to $\omega_\Sigma$. 
 The following is therefore an easy consequence of \ref{T:manny}.
\begin{theorem}
Let $\Pi=\mw\left(\Sigma,2\right)$ be a discrete series representation of $\GL_4\left(\A\right)$ such that there exists a cuspidal irreducible representation $\Pi'$ of $\GL'_2\left(\A\right)$ with $\jl\left(\Pi'\right)=\Pi$, which is cohomological with respect to the coefficient system $E_\mu^\lor$.  Assume that either $\omega_\Sigma^2$ is trivial or the $\Aut(\CC)$-orbit of $\Pi'$ admits a unique local Shalika model with respect to $\omega_\Sigma$. Let $\chi$ be a finite order Hecke-character of $\GL_1\left(\A\right)$ and let $s=\frac{1}{2}+m \in \CCrit\left(\Pi'\right)$. Then
\[\frac{L\left(\frac{1}{2}+m,\Pi_f\otimes \chi_f\right)}{\omega\left(\Pi'_f\right)\mathcal{G}\left(\chi_f\right)^4\omega\left(\Pi_\infty',m\right)}\in \QQ\left(\Pi',\omega_\Sigma,\chi\right).\]
\end{theorem}
\section{Proof of Theorem \ref{T:global} \& Theorem \ref{T:connL}}\label{S:JF}
We will now show how to adapt the proof given in \cite{FriJac} to the situation at hand. Almost all of the arguments remain unchanged and we only include them for completeness.
Throughout this section $\Pi'$ will be a cuspidal irreducible representation of $\GL_{2n}'(\A)$ which admits a Shalika model with respect to $\eta$ and $\phi\in\Pi$ will be a cusp form.

If $H$ is an algebraic subgroup of $\GL'_{2n}$ containing $Z_{2n}'$ we denote by \[H^0\left(\A\right)=\{h\in H\left(\A\right): \lvert\det'\left(h\right)\lvert=1\}.\]
Given a Haar-measure $\ddd h$ on $H(\A)$, there exists a Haar measure $\ddd z$ on the center ${Z_{2n}'}(\KK)\bs {Z_{2n}'}(\A)$ such that for all $s\in \CC$ and $f$ a smooth function on $H(\A)$
\begin{equation}\begin{gathered}\label{E:inteex}\int_{Z_{2n}(\A)H(\KK)\bs H(\A)}f(h)\lvert\det'(h)\lvert^s\ddd h=\\ \int_{{Z_{2n}'}(\KK)\bs {Z_{2n}'}(\A)}\lvert \det'(z)\lvert ^s \int_{H(\KK)\bs H^0(\A)}f(hz)\ddd h \ddd z,\end{gathered}\end{equation}
assuming the first integral converges.
Indeed, this follows since $H^0(\A)\bs H(\A)$ can be identified with ${Z_{2n}^{'0}}(\A)\bs Z_{2n}'(\A)\times {Z_{2n}^{'0}}(\A)\bs Z_{2n}'(\A)$
and that the integral of $\lvert\det'\lvert^s$ over ${Z_{2n}^{'0}}(\A)\bs Z_{2n}'(\A)$ is the same as the integral of $\lvert\det'\lvert^s$ over $Z'_{2n}(\KK)\bs Z_{2n}(\A)$.
We will denote by $\mathrm{1}_n$ the ${n}$-dimensional identity matrix.
Let $q,p\in\mathbb{Z}^+$ be such that $p+q={2n}$, let $U'_{(q,p)}\subseteq P_{(q,p)}'$ be the corresponding unipotent subgroup of $\GL'_{2n}$ and let $A$ be the group of diagonal matrices of $\GL_{{2n}}$ embedded into $\GL'_{2n}$. We identify $U'_{(q,p)}$ from time to time with the linear space of $p\times q$ matrices $ M_{q,p}'$. To each $\beta\in M'_{q,p}(\KK)$ we associate the character $\theta_\beta$ of $U'_{(q,p)}\left(\A\right)$ 
\[u=\bpm \mathrm{1}_p&v\\0&\mathrm{1}_q\epm\mapsto \psi\left(\mathrm{Tr}'\left(v\beta\right)\right).\]
Moreover, let $H=\GL_p'\times \GL_q'$ be the Levi-component of $P_{(q,p)}$.
Then for $\gamma=\bpm \gamma_1&0\\0&\gamma_2\epm\in H\left(\KK\right)$ it is straightforward to see that
$\theta_\beta\left(\gamma^{-1}u\gamma\right)=\theta_{\gamma_2^{-1}\beta\gamma_1}\left(u\right).$
The additive group $M'_{q,p}\left(\A\right)$ is isomorphic to $M_{dq,dp}\left(\A\right)$. It is well known that the additive characters of the latter group are parametrized by the space of linear functionals $\mathrm{Hom}_\KK\left(M_{dq,dp}\left(\KK\right),\KK\right)$ by identifying 
$X\in \mathrm{Hom}_\KK\left(M_{dq,dp}\left(\KK\right),\KK\right)$ 
with the additive character $v\mapsto \psi(X(v))).$ Identifying $M_{dq,dp}\left(\KK\right)$ with $M'_{q,p}\left(\KK\right)$ again, we obtain that all characters of $U'_{(q,p)}\left(\A\right)$ are of the form $\theta_\beta$ and $\theta_\beta=\theta_{\beta'}$ if and only if $\beta'$.
This allows us to consider for a cuspform $\phi\in \Pi'$ its Fourier expansion 
\[\phi\left(g\right)=\sum_{M'_{q,p}\left(\KK\right)}\phi_\beta\left(g\right),\]
where \[\phi_\beta\left(g\right)\coloneqq \int_{U'_{(q,p)}\left(\KK\right)\bs U'_{(q,p)}\left(\A\right)}\phi\left(gu\right)\theta_\beta\left(u\right)du.\]
It is again easy to see that
$\phi_\beta\left(\gamma g\right)=\phi_{\gamma_2^{-1}\beta\gamma_1}\left(g\right)$
and $\phi_0=0$, since $\phi$ is cuspidal.
\begin{lemma}\label{L:conv}
\[\int_{H\left(\KK\right)\bs H^0\left(\A\right)}\sum_{\beta\in M'_{q,p}\left(\KK\right)}\lvert\phi_\beta\left(h\right)\lvert \ddd h<\infty\]
\end{lemma}
\begin{proof}
It suffices to show that the integral is finite over a standard Siegel set of $H^0$, \emph{i.e.} let $H_{2n}$ be the Cartan subgroup of $\GL_{2n}'$ consisting of the diagonal matrices with entries in a fixed maximal subfield $\mathbb{E}\subseteq \DD$, let $\Omega$ be a compact subset of $\GL'_{2n}\left(\A\right)$, let $C$ be a positive constant and let $S\left(C\right)$ be the connected component of $\mathrm{1}_{2n}$ of the diagonal matrices $$a=\mathrm{diag}\left(a_1,\ldots,a_p,a_{p+1},\ldots,a_{{2n}}\right)$$ with $a\in H_{2n}$ satisfying $\left\lvert\frac{a_i}{a_{i+1}}\right\lvert\ge C$ for $i\neq p,{2n}$ and $\prod_{i=1}^pa_i=\prod_{i=p+1}^{{2n}}a_i=1$, see \cite[Theorem 4.8]{PlaRap}.
Hence, we have to show that there exists a constant $D$ such that
\[\sum_{\beta\in M'_{q,p}}\lvert\phi_\beta\left(a\omega\right)\lvert<D\]
for all $a\in S\left(C\right)$ and $\omega\in \Omega$.
We consider the function $u\mapsto\phi\left(ua\omega\right),\, u\in U_{(q,p)}'(\A)$ as a smooth, periodic function in $u$ for fixed $a$ and $\omega$. Then its Fourier series is also smooth and converges absolutely. To prove that this convergence is uniform, \emph{i.e.} independent of $a$ and $\omega$, it suffices to show like in the proof of \cite[Lemma 2.1]{FriJac} that firstly, there exists a compact open subgroup $U_f\subseteq U_{(q,p)}'(\A_f)$ such that
$\phi\left(uu'a\omega\right)=\phi\left(ua\omega\right)$
for all $u'\in U_f$ and secondly, there exists a constant $D'$ independent of $a$ and $\omega$ such that for any $X$ of the enveloping universal algebra of $\mathfrak{u}_{(q,p),\infty},$ the Lie algebra of $\prod_{v\in {\VV_\infty}}U_{(q,p)}(\KKv)$,
\[\lvert\lambda\left(X\right)\phi\left(ua\omega\right)\lvert<D'.\]
Here we denote by $\lambda$ the left action of $U_\infty$ and $\rho$ its right action. The existence of $U_f$ as above follows immediately from the smoothness of $\phi$, since $\Omega$ is compact and $S(C)$ normalizes $U_{(q,p)}'(\A_f)$.
To prove the second claim, we fix $v\in {\VV_\infty}$, a root $\alpha$ of $H_{2n}$ in $U'_{(q,p)}$ and a root vector $X_\alpha$ of $\alpha$ in the Lie algebra of $\mathfrak{u}_{(q,p),\infty}$. Recall that since $H_{2n}$ is a Cartan subgroup, such root vectors span $\mathfrak{u}_{(q,p),\infty}$. Then
\[\lambda\left(-X_\alpha\right)\phi\left(ua\omega\right)=\alpha\left(a_v\right)^{-1}\rho\left(\mathrm{ad}(\omega^{-1}) X_\alpha\right)\phi\left(ua\omega\right).\]
 Now $\mathrm{ad}(\omega^{-1}) X_\alpha$ is a linear combination of basis elements of $\lig'_v$, whose coefficients are bounded, because $\Omega$ is compact. Since $a\in S\left(C\right)$, $\alpha\left(a_v\right)^{-1}$ is bounded by a constant multiple of $\lvert a_p\lvert^{-M}\lvert a_{{2n}}\lvert^M$ for some $M\ge 0$.

Therefore, $\lambda\left(-X_\alpha\right)\phi\left(ua\omega\right)$ is bounded above by
\[\sum_j \lvert a_p\lvert^{-M_j}\lvert a_{2n}\lvert^{M_j}\lvert\phi_j\left(ua\omega\right)\lvert,\]
for $\phi_j\in \Pi$.
The following lemma will be useful in this and the following proof.
\begin{lemma}[{\cite[Lemma 2.2]{FriJac}}]\label{L:Lemma22}
Let $\phi$ be a cusp form of a reductive group $G$, which is invariant under the split component of the center of $G$. Let $R$ be a maximal proper parabolic subgroup of $G$, let $\delta_R$ be the module of the group $R\left(\A\right)$ and let $\Omega$ be a compact subset of $G\left(\A\right)$. Then for every $M\ge 0$ there exists a constant $D$ such that
$\delta_R\left(r\right)^M\lvert\phi\left(r\omega\right)\lvert\le D$
for all $r\in R\left(\A\right)$ and $\omega\in \Omega$.
\end{lemma}
Since $ua$ is contained in $P_{(q-1,p+1)}'(\A)$ and $P_{(q+1,p-1)}'$ with respective modules
\[\delta_{P_{(p-1,q+1)}'}\left(ua\right)=\lvert a_p\lvert^{-2nd},\, \delta_{P_{(p+1,q-1)}'}\left(ua\right)=\lvert a_{p+1}\lvert^{2nd},\]
we deduce that $\lvert a_p\lvert^{-M_j}\lvert a_{2n}\lvert^{M_j}\lvert\phi_j\left(ua\omega\right)\lvert$ is bounded above. This finishes the proof.
\end{proof}
The next step is to observe that even though we are dealing with matrices over a division algebra, Gauss elimination still holds true in $M'_{q,p}\left(\KK\right)$.
Therefore, the $H(\KK)=\GL_p'\left(\KK\right)\times \GL_q'\left(\KK\right)$-orbits on $M'_{q,p}\left(\KK\right)$ under the action $\gamma\cdot \beta=\gamma_2^{-1}\beta\gamma_1$ are precisely given by the possible ranks of the matrices.
To be more precise, we say a matrix $\beta$ has rank $r$ if it is in the orbit of \[\beta_r\coloneqq \bpm \mathrm{1}_r &0\\0&0\epm.\]
The stabilizer $H_{\mathrm{1}_r}\left(\KK\right)$ of the matrix
$\beta_r$ is the subgroup of matrices of the form
\begin{equation}\label{E:stab}\bpm g_1&0\\0&g_2\epm \text{ with } g_1=\bpm a_1&b_1\\0&d_1\epm,\, g_2=\bpm a_2&0\\c_1&d_1\epm,\end{equation}
where $d_1$ is a square matrix of dimension $r$, $a_1$ is a square matrix of dimension $q-r$ and $a_2$ is a square matrix of dimension $p-r$.
 Now we can write 
\[\phi\left(g\right)=\sum_{r=1}^{\min(q,p)}\sum_{\gamma\in H_{\mathrm{1}_r}\left(\KK\right)\bs H\left(\KK\right)} \phi_{\beta_r}\left(\gamma g\right).\]
 Next the generalization of \cite[Proposition 2.1]{FriJac}, which goes throw exactly like in split case..
\begin{prop}\label{P:A1}
Let $q>p$. Then for any cusp form $\phi\in \Pi'$,
\[\int_{G_q'\left(\KK\right)\bs G_q'^0\left(\A\right)}\phi \left(\bpm g_1&0\\0& \mathrm{1}_p\epm\right)\ddd g_1=0.\]
\end{prop}
\subsection{ Proof of Theorem \ref{T:global}}
 Before we begin the proof, let us remark the following. 
\begin{remark}\label{R:bounded}
 If $\Pi'$ admits a Shalika model with respect to $\eta$, $\Pi_v'$ admits a Shalika model with respect to $\eta_v$, \emph{i.e.} a continous intertwining map
\[\left(\Pi_v'\right)^{\infty}\rightarrow \mathrm{Ind}_{\mS\left(\KKv\right)}^{\GL_{2n}'\left(\KKv\right)}\left(\psi_v\otimes \eta_v\right)\] if $v\in {\VV_\infty}$ and a morphism of $\GL(\KKv)$-representations
\[\Pi_v'\rightarrow \mathrm{Ind}_{\mS\left(\KKv\right)}^{\GL_{2n}'\left(\KKv\right)}\left(\psi_v\otimes \eta_v\right)\] if $v\in {\VV_f}$.
In both cases Frobenius reciprocity gives us a continuous morphism 
\[\lambda_v\in \mathrm{Hom}_{\mS\left(\KKv\right)}\left(\left(\Pi_v'\right)^\infty,\psi_v\otimes \eta_v\right)\text{, respectively, }\lambda_v\in \mathrm{Hom}_{\mS\left(\KKv\right)}\left(\Pi_v',\psi_v\otimes \eta_v\right).\]
If $v\in {\VV_\infty}$ the so obtained map is a priori just an intertwiner of group actions, but not necessarily continuous. However, the space of smooth vectors satisfies the Heine-Borel property, \emph{i.e.} a subset of $\left(\Pi_v'\right)^\infty$ is compact if and only if it is bounded on bounded sets and closed.
Since a linear map of Fr\'echet spaces is continuous if and only if it is bounded, the claim follows. If $v\in {\VV_f}$ we obtain $\lambda_v$ without any extra steps.

For a cuspform $\phi=\bigotimes_{v\in \VV}\phi_v\in \Pi'$ we have that $\lvert\phi\left(g\right)\lvert\le \beta\left(\phi\right)$, where $\beta$ is a semi-norm on $(\Pi_\infty')^\infty\otimes \Pi_f^{K_f'}$ for any open compact subgroup $K_f'$, since cusp forms are of rapid decay.
Letting $\lambda$ be the Shalika functional associated via Frobenius reciprocity to our Shalika model, we
obtain that $g\mapsto \lambda\left(\Pi'\left(g\right)\phi\right)$ is bounded. Thus, if we restrict $\lambda$ to smooth vectors we obtain so the local Shalika functionals $\lambda_v,\, v\in \VV$,
we also have that $g_v\mapsto \lambda_v\left(\Pi_v'\left(g_v\right)\phi_v\right)$ is bounded for all $v\in \VV$.
\end{remark}
\begin{theorem}\label{T:Aglobal}
Let $\Pi'$ be a cuspidal irreducible representation of $\GL_{{2n}}'\left(\A\right)$.
Assume $\Pi'$ admits a Shalika model with respect to $\eta$ and let $\phi\in\Pi'$ be a cusp form. Consider the integrals
\[\Psi\left(s,\phi\right)\coloneqq \int_{Z'_{2n}\left(\A\right)H_n'\left(\KK\right)\bs H_n'\left(\A\right)}\phi\left(\bpm h_1&0\\0&h_2\epm\right)\left\lvert \frac{\det'\left(h_1\right)}{\det'\left(h_2\right)}\right\lvert^{s-\frac{1}{2}}\eta\left(h_2\right)^{-1}\ddd h_1\ddd h_2,\]
\[\zeta\left(s,\phi\right)\coloneqq \int_{\GL'_n\left(\A\right)}\mS^\eta_\psi\left(\phi\right)\left(\bpm g_1&0\\0&1\epm\right)\lvert\det'\left(g_1\right)\lvert^{s-\frac{1}{2}}\ddd g_1.\]
Then $\Psi\left(s,\phi\right)$ converges absolutely for all $s$ and $\zeta\left(s,\phi\right)$ converges absolutely if $\re >>0$. Moreover, if $\zeta\left(s,\phi\right)$ converges absolutely,
$\Psi\left(s,\phi\right)=\zeta\left(s,\phi\right)$.
\end{theorem}
\begin{proof}
We apply \ref{L:Lemma22} to the case $R'=P_{(n,n)}'\subseteq \GL_{2n}'$ to see that\[{\phi\left(\bpm h_1&0\\0&h_2\epm\right)\left\lvert\frac{\det' \left(h_1\right)}{\det' \left(h_2\right)}\right\lvert}^{M}\]
is bounded above for any $M$, hence, $\Psi(s,\phi)$ converges absolutely.
Indeed, recall that $Z_n'(\A)\GL_n'(\KK)\bs \GL_n'(\A)$ has finite volume and hence \[Z'_{2n}(\A)H_n'(\KK)\bs H_n'(\A)=(1_n\times Z_{n}'(\A))\Omega,\] where $\Omega$ has finite volume. Since above $M$ can be chosen arbitrarily small, the claim follows.
For a suitable measure $\ddd z$ on $Z_{2n}'\left(\KK\right)\bs Z'_{2n}\left(\A\right)$ we have by (\ref{E:inteex})
\[\Psi\left(s,\phi\right)=\int_{Z_{2n}'(\KK)\bs Z'_{2n}(\A)}\lvert\det'\left(z\right)\lvert^{s-\frac{1}{2}} \]\[\int_{\GL'_n\left(\KK\right)\bs \GL^{'0}_n(\A)}\int_{\GL'_n\left(\KK\right)\bs \GL^{'0}_n\left(\A\right)} \phi\left(\bpm h_1z&0\\0&h_2\epm\right)\eta\left( h_2\right) \ddd h_1 \ddd h_2\ddd z.\]
Inserting the Fourier series we see that the contribution of the matrices with rank $r<n$ is $0$ by \ref{P:A1} and hence, 
\begin{equation}\label{E:sumin}\begin{gathered}\int_{\GL'_n\left(\KK\right)\bs \GL^{'0}_n(\A)}\int_{\GL'_n\left(\KK\right)\bs \GL^{'0}_n\left(\A\right)} \phi\left(\bpm h_1z&0\\0&h_2\epm\right)\eta\left( h_2\right) \ddd h_1 \ddd h_2=\\=\int_{\GL'_n\left(\KK\right)\bs \GL^{'0}_n\left(\A\right)\times \GL'_n\left(\KK\right)\bs \GL^{'0}_n\left(\A\right)}\sum_{\gamma_1\in \GL'_n\left(\KK\right)}\\\phi_{\beta_n}\left(\bpm\gamma_1h_1z&0\\0&h_2\epm \right) \eta\left( h_2\right) \ddd h_1\ddd h_2.\end{gathered}\end{equation}
Contracting the sum and the integral and performing a change of variables, it follows that (\ref{E:sumin}) is equal to 
\[\int_{\GL^{'0}_n\left(\A\right)}\mS^\eta_\psi\left(\phi\right)\left(\bpm gz&0\\0&\mathbf{1}_n\epm \right)\eta\left(g\right)\ddd g.\]
Thus $\Psi\left(s,\phi\right)$ and $\zeta(s,\phi)$ are equal to 
\begin{equation}\label{E:convne} \int_{Z'_{2n}\left(\KK\right)\bs Z'_{2n}\left(\A\right)}\lvert\det'\left( z\right)\lvert^{s-\frac{1}{2}}\int_{\GL^{'0}_n\left(\A\right)}\mS^\eta_\psi\left(\phi\right)\left(\bpm gx&0\\0&\mathrm{1}_n\epm\right)\eta(g)\ddd g\ddd z\end{equation}
where the last equation is valid only once we show that $\zeta\left(s,\phi\right)$ converges absolutely for $\re>>0$. To show the convergence we use the Dixmier-Malliavin theorem.
\begin{theorem}[Dixmier-Malliavin theorem]\label{L:DLapMinII}
Suppose $G$ to be a Lie group and $\pi$ a continuous representation of $G$ on a Fr\'echet space $V$.
Then every smooth vector $v\in V^\infty$ can be represented as a finite sum
$v=\sum_k \pi\left(f_k\right)v_k,$
with $v_k\in V$, $f_k$ a smooth, compactly supported function on $G$ and
$\pi\left(f\right)w\coloneqq \int_G f\left(x\right)\pi\left(x\right)w\ddd x$
for some fixed Haar measure on $G$.
\end{theorem}
\begin{remark}
Note that if $G$ is a reductive group over $\KK$ and $(\Pi'_f,V_f)$ is a smooth representation of $G(\A_f)$, we can write $v=\Pi'_f(v)\coloneqq \int_{G(\A_f)} \phi\left(x\right)\Pi'\left(x\right)v\ddd x$ for some smooth, \emph{i.e.} locally constant, function $\phi$ as every vector in $V_f$ is fixed by some open compact subgroup.
\end{remark}
We consider the action of $\GL_{{2n}}'(\A)$ on $\mS_\psi^\eta\left(\Pi'\right)$. The cusp form $\phi$ is a smooth vector in $\Pi'$, where we consider $\Pi'$ as a proper $\GL_{2n}'(\A)$-subrepresentation of the corresponding $L^2$-space.
Applied to our case this yields that $\mS_\psi^\eta\left(\phi\right)\left(g\right)$ can be written as a finite sum
\[\sum_k \int_{U'_{(n,n)}\left(\A\right)}\xi_k\left(g\bpm \mathrm{1}_n&u\\0&\mathrm{1}_n\epm \right)\phi_k\left(u\right)\ddd u,\]
where the $\phi_k$ are compactly supported, smooth functions on $U'_{(n,n)}\left(\A\right)$ and $\xi_k\in \mS_\psi^\eta\left(\Pi\right)$. Moreover, all $\xi_k$ satisfy the equivariance property under $\eta$ and $\psi$ and are therefore bounded by the remark of \ref{R:bounded}. Applying this, we deduce
\[\mS_\psi^\eta\left(\bpm g_1&0\\0&\mathrm{1}_n\epm\right)=\sum_k \xi_k\left(\bpm g_1&0\\ 0&\mathrm{1}_n\epm\right)\widehat{\phi_k}\left(g_1\right),\]
where $\widehat{\phi_k}$ is the Fourier transform of $\phi_k$. 
Recalling the definition of $\zeta(s,\phi)$, we obtain
\[\zeta(s,\phi)=\int_{\GL'_n\left(\A\right)}\sum_k \xi_k\left(\bpm g_1&0\\ 0&\mathrm{1}_n\epm\right)\widehat{\phi_k}\left(g_1\right)\lvert\det'\left(g_1\right)\lvert^{s-\frac{1}{2}}\ddd g_1\]
Since the $\xi_k$ are bounded, $\zeta(s,\phi)$ is thus bounded by a multiple of 
\[\sum_k \int_{\GL'_n\left(\A\right)}\widehat{\phi_k}\left(g_1\right)\lvert\det g_1\lvert^{s-\frac{1}{2}}dg_1,\]
which converges absolutely for $s$ with real part sufficiently large by \cite[Theorem 13.8]{GodJac} and thus $\zeta(s,\phi)$ converges for $\re>>0$.
\end{proof}
\subsection{ Proof of Theorem \ref{T:connL}}
\begin{theorem}\label{T:Aconnl}
Let $\Pi'$ be a cuspidal irreducible representation of $\GL_{2n}'\left(\A\right)$
and assume $\Pi'$ admits a Shalika model with respect to $\eta$. Then for each place $v\in \VV$ and $\xi_{v}\in\mS_{\psi_v}^{\eta_v}\left(\Pi_v'\right)$ there exists an entire function $P\left(s,\xi_{v}\right)$, with $P\left(s,\xi_{v}\right)\in \CC[q_v^{s-\frac{1}{2}},q_v^{\frac{1}{2}-s}]$ if $v\in {\VV_f}$, such that
\[\zeta_v\left(s,\xi_{v}\right)=P\left(s,\xi_{v}\right)L\left(s,\Pi_v'\right)\]
and hence, $\zeta_v\left(s,\xi_{v}\right)$ can be analytically continued to $\CC$. Moreover, for each place $v$ there exists a vector $\xi_{v}$ such that $P\left(s,\xi_{v}\right)=1$. If $v$ is a place where neither $\Pi'$ nor $\psi$ ramify this vector can be taken as the spherical vector $\xi_{\Pi_v'}$ normalized by $\xi_{\Pi_v'}\left(\mathrm{id}\right)=1$.
\end{theorem}
We follow closely the proofs of \cite[Proposition 3.1, Proposition 3.2]{FriJac}.
We denote by $S\left(M_{s,t}\right)$, respectively, $S\left(M_{s,t}'\right)$ the space of Schwartz-functions on $M_{s,t}\left(\KKv\right)$, respectively, $M_{s,t}'\left(\KKv\right)$.
\begin{proof}
The first step is to prove the following lemma
\begin{lemma}\label{L:bigth}
    There exists, depending on $\xi_v$, a positive Schwartz-function $\Theta\in S\left(M_{n,n}\right)$, such that
\[\left\lvert\xi_v\left(\left(\bpm g_1&0\\ 0&1\epm g\right)\right) \right\lvert\le \Theta\left(b^{-1}g_1a\right)\]
for the Iwasawa decomposition \[g=u\bpm a&0\\0&b\epm k,\, u\in U_{(n,n)}'\left(\KKv\right),\,\bpm a&0\\0&b\epm\in H_n'\left(\KKv\right),\, k\in K_v' .\]
\end{lemma}
\begin{proof}
We first assume that we are in the archimedean case. Using the Dixmier-Malliavin theorem, it is enough to prove the claim in the case $\xi_v$ being of the form
\[\int_{\GL'_{2n}\left(\KKv\right)}\xi_{v,1}\left(gh\right)\Psi\left(h\right)\dv h,\]
where $\Psi$ is smooth function of $\GL'_{2n}\left(\KKv\right)$ with compact support. Write $g$ as 
\[h=\bpm \mathrm{1}_n&u_1\\0&\mathrm{1}_n\epm\bpm a_1&0\\0&b_1\epm k_1,\,g=\bpm \mathrm{1}_n&u_2\\0&\mathrm{1}_n\epm\bpm a_2&0\\0&b_2\epm k_2.\] We compute
\[\xi_{v}\left(\left(\bpm g_1&0\\ 0&\mathrm{1}_n\epm g\right)\right)=\]\[=\psi\left(\tr\left(g_1u_2\right)\right)\int_{\GL'_n(\KKv)\times \GL'_n\left(\KKv\right)\times K_v'}\xi_{v,1}\left(\bpm g_1a_2a_1&0\\0&b_2b_1\epm k_1\right)\]
\[\Xi\left(b_2^{-1}g_1a_2;k,k_2,a_1,b_1\right)\lvert\det'_v\left( a_1b_1^{-1}\right)\lvert^{-nd}\dv a_1\dv b_1\dv k_1,\]
where $\Xi\left(v;k_1,k_2,a_1,b_2\right)$ is the Fourier transform of \[u_1\mapsto \Psi\left(k_2^{-1}\bpm \mathrm{1}_n&u_1\\0&\mathrm{1}_n\epm \bpm a_1&0\\0&b_1\epm k_2\right)\]
This function and its derivatives have compact support, which is independent of the parameters $k_2,a_1,b_1,k_1$. Thus, the respective Fourier transform are contained in a bounded set in the space of Schwartz-functions on $U_{(n,n)}'\left(\KKv\right)$.
Hence, there exists a positive Schwartz-function $\Theta_1$ and a function $\Theta_2$ with compact support such that \[\lvert\Xi\left(v;k_1,k_2,a_1,b_1\right)\lvert\le \Theta_1\left(v\right)\Theta_2\left(a_1,b_1\right).\] This is enough to show the majorization, since $\xi_{v,1}$ is bounded by the remark of \ref{R:bounded}.
In the case where $\KKv$ is non-archimedean we do not need the Dixmier-Malliavin lemma, since we automatically can write $\xi_v$ in integral form by the remark after \ref{L:DLapMinII}.
\end{proof}
In the next step we let $v\in \VV$ be a place and consider integrals of the form
\[Z\left(\xi_v,\Psi,s\right)\coloneqq \int_{\GL'_{2n}\left(\KKv\right)}\xi_v\left(g\right)\Psi\left(g\right)\lvert\det'_v\left( g\right)\lvert^{s-\frac{{2n}d-1}{2}}\dv g\]
for $\xi_v\in \mS_{\psi_v}^{\eta_v}\left(\Pi_v'\right)$ and $\Psi\in S\left(M'_{{2n},{2n}}\right)$.
Since $\xi_v$ is bounded, this integral converges for $\re>>0$, see for example the proof of \cite[Theorem, 3.3]{GodJac}.
\begin{lemma}\label{L:helpls}
The function \[\frac{Z\left(\xi_v,\Psi,s\right)}{L\left(s,\Pi_v'\right)}\] is meromorphic and if $v\in {\VV_f}$ it is an element of $\CC[q_v^{s-\frac{1}{2}},q_v^{\frac{1}{2}-s}].$ Moreover, there exists $\xi_{v,j}\in \mS_{\psi_v}^{\eta_v}\left(\Pi_v'\right)$, $\Psi_j\in S\left(M'_{{2n},{2n}}\right)$ such that we can write the local $L$-factor as a finite sum of the form \[L\left(s,\Pi_v'\right)=\sum_j Z\left(\xi_{v,j},\Psi_j,s\right).\]
\end{lemma}
\begin{proof}
We first assume that $\KKv$ is non-archimedean. 
Let $I(\Pi_v')$ be the $\CC$ vector-space spanned by the integrals of the form
\[Z\left(f,\Psi,s\right)\coloneqq \int_{\GL'_{2n}\left(\KKv\right)}\Psi\left(g\right)f\left(g\right)\lvert\det'_v\left( g\right)\lvert^{s-\frac{{2n}d-1}{2}}\dv g,\]
where $f$ is a smooth matrix coefficient of $\Pi_v'$ and $\Psi\in S\left(M'_{{2n},{2n}}\right)$.
To be more precise, the integrals converge for $\re>>0$ and admit a meromorphic continuation. By 
\cite[Theorem 3.3]{GodJac}
$I(\Pi_v')$ is a $\CC[q_v^{s-\frac{1}{2}},q_v^{\frac{1}{2}-s}]$-ideal in $\CC(q_v^{s-\frac{1}{2}})$ generated by $L(s,\Pi_v')$.

We will now show that the $\CC$-vector space spanned by the $Z(\xi_v,\Psi,s)$ is $I(\Pi_v')$, which consequently will show the claims of the lemma.
To do so we introduce the space $\mathscr{U}$ consisting of smooth matrix coefficients of the form \[g\mapsto \int_{K_v'} \xi_v\left(k^{-1}g\right)e\left(x\right)\dv k,\,g\in \GL(\KKv),\, \xi_v\in\mS_{\psi_v}^{\eta_v}\left(\Pi_v'\right),\]
where $e$ is an idempotent under the usual convolution product on the functions supported on $K_v'$.
Given $\xi_v$ and $\Psi$, we define \[g\mapsto f\left(g\right)\coloneqq \int_{K_v'} \xi_v\left(x^{-1}g\right)e\left(k\right)\dv k,\, g\in \GL(\KKv),\] which is a smooth matrix coefficient of $\Pi_v'$ and hence, for such $f$ \[Z\left(\xi_v,\Psi,s\right)=\int_{\GL'_{2n}\left(\KKv\right)}f\left(g\right)\Psi\left(g\right)\lvert\det'_v\left( g\right)\lvert^{s-\frac{{2n}d-1}{2}}\dv g\in \mathscr{U}.\]
On the other hand, for every $f\in\mathscr{U}$ and Schwartz-function $\Psi$ there exists $\xi_v\in  \mS_{\psi_v}^{\eta_v}\left(\Pi_v'\right), \Psi'\in S\left(M'_{{2n},{2n}}\right)$ such that $Z\left(f,\Psi,s\right)=Z\left(\xi_v,\Psi',s\right)$. Indeed, 
\[Z(f,\Psi,s)=\int_{\GL'_{2n}\left(\KKv\right)}\int_{K_v'} \xi_v\left(k^{-1}g\right)e\left(k\right)\Psi\left(g\right)\lvert\det'_v\left( g\right)\lvert^{s-\frac{{2n}d-1}{2}}\dv k\dv g=\]
\[=\int_{\GL'_{2n}\left(\KKv\right)} \xi_v\left(g\right)\int_{K_v'}\left(e\left(k\right)\Psi\left(kg\right)\right\lvert\det'_v\left( g\right)\lvert^{s-\frac{{2n}d-1}{2}}\dv k\dv g=Z(\xi_v,\Psi',s),\]
where $\Psi'(g)\coloneqq \int_{K_v'}e\left(k\right)\Psi\left(kg\right)\dv k$.
This shows that the space spanned by the $Z(\xi_v,\Psi, s)$ is the space spanned by $\mathscr{U}$. It therefore suffices to show that the span of $\mathscr{U}$ is $I(\Pi_v')$. But since $\mathscr{U}$ is closed under right translations under $\GL(\KKv)$ and $\Pi_v'$ is irreducible, any smooth matrix coefficient $f$ of $\Pi_v'$ can be written as a finite sum \[f\left(g\right)=\sum_i f_i\left(g_ig\right)\] for some suitable $g_i\in \GL(\KKv)$ and $f_i\in \mathscr{U}$. Therefore, the final claim follows because then
\[Z(f,\Psi,s)=\int_{\GL\left(\KKv\right)}\sum_{i}f_{i}\left(g_ig\right)\Psi\left(g\right)\lvert\det g\lvert^{s-\frac{{2n}d-1}{2}}\dv g=\]\[=\int_{\GL\left(\KKv\right)}\sum_{i}f_{i}\left(gg_i\right)\Psi\left(g_i^{-1}gg_i\right)\lvert\det g\lvert^{s-\frac{{2n}d-1}{2}}\dv g,\] where the last expression is of the desired form.
In the case where $\KKv$ is archimedean, we argue as above, replacing the action of $\GL'_{2n}\left(\KKv\right)$ by the action of the Lie algebra and $K_v'$ and using the Dixmier-Malliavin lemma.
\end{proof}
We return now to the proof of \ref{T:Aconnl} and assume from now on that $v$ is archimedean, since the non-archimedean case can be dealt with analogously. We start with the second assertion.
We will only prove the archimedean case, since the non-archimedean case follows analogously. Let $\SL_{2n}'\coloneqq \{g\in \GL_{2n}':\det'_v\left(g\right)=1\}$ and $K_0\coloneqq \SL_{2n}'(\KKv)\cap K_v'$.
Using the Iwasawa decomposition we can write
\[ Z\left(\xi_v,\Psi,s\right)=\int_{H_n'\left(\KKv\right)\times U'_{(n,n)}\left(\KKv\right)\times K_0}\xi_v\left(\bpm a&x\\0&b\epm k\right)\Psi\left(\bpm a&x\\0&b\epm k\right)\]\[\lvert\det'_v\left(a\right)\lvert^{s-\frac{1}{2}}\lvert\det'_v\left( b\right)\lvert^{s-\frac{{n}d-1}{2}}\dv a\dv b\dv x\dv k.\]
We introduce the function
\begin{equation}\label{E:Xi}\Xi\left(u,t,w;k\right)\coloneqq \int_{H_n'\left(\KKv\right)}\Psi\left(\bpm x&y\\0&w\epm k\right)\left(\mathrm{Tr}'\left(yt\right)-\mathrm{Tr}'\left(xu\right)\right)\dv x\dv y.\end{equation}

If we put issues of convergence aside for a moment, the Fourier inversion formula and a change of variables imply that
\begin{equation}\label{E:strange}Z\left(\xi_v,f,s\right)=\int_{\SL_{2n}'\left(\KKv\right)\times \GL'_n\left(\KKv\right)}\xi_v\left(\bpm a&0\\0&\mathrm{1}_n\epm x\right)\lvert\det'_v\left( a\right)\lvert^{s-\frac{1}{2}} \ddd\mu_\Psi\left(x\right)\dv a,\end{equation}
where we define the measure $\mu_\Psi$ on $\SL_{2n}'\left(\KKv\right)$ by
\[\int_{\SL_{2n}'\left(\KKv\right)}\Psi\left(x\right)\ddd \mu_\Psi\left(x\right)\coloneq\int_{\GL'_n\left(\KKv\right)\times U'_{(n,n)}\left(\KKv\right)\times K_v'}\Xi\left(u,b,b^{-1};k\right)\lvert\det'_v\left( b\right)\lvert^{nd}\]\[f\left(\bpm b^{-1}&0\\0&\mathrm{1}_n\epm \bpm \mathrm{1}_n &u\\0&\mathrm{1}_n\epm \bpm \mathrm{1}_n&0\\0&b\epm k\right)\dv b\dv u\dv k.\]
Let us now argue how to put the issues of convergence to rest in the integral of (\ref{E:strange}).
Following \cite{FriJac} we consider the unimodular subgroup $Q$ of $\GL'_{2n}$ consisting of matrices of the form
\[Q=\left\{\bpm b^{-1}&u\\0&b\epm:b\in \GL_n',\, u\in M_{n,n}'\right\}.\]
Thus, $\ddd\mu_\Psi=\mu_\Psi\left(q,k\right)\dv q\,\dv k$, where $\mu_\Psi$ is a smooth function on $Q(\KKv)\times K_v'$ and it and its derivatives are rapidly decreasing, \emph{i.e.} \[\lvert\lvert q\lvert\lvert^N\lvert\mu_\Psi(q,k)\lvert\]
is bounded for all natural numbers $N$, where $\lvert\lvert q\lvert\lvert$ denotes the usual height of $q$.
Recall the majorization $\alpha$ of the beginning of the proof and the remark before \ref{T:Aglobal} and that we obtained from \ref{L:bigth} that
\[\int_{\SL_{2n}'\left(\KKv\right)\times \GL'_n\left(\KKv\right)}\left\lvert\xi_v\left(\bpm a&0\\0&\mathrm{1}_n\epm x\right)\lvert\det'_v\left( a\right)\lvert^{s-\frac{1}{2}}\right\lvert \ddd\mu_\Psi\left(x\right)\dv a\le \]
\[\le\int_{\GL_n'(\KKv)\times Q(\KKv)}\Theta(b^{-1}ab^{-1})\lvert \det'_v(a)\lvert^{\re-\frac{1}{2}}\lvert\mu_\Psi(q,k)\lvert \dv  q\dv k\dv  a\] for a suitable Schwartz-function $\Theta$.
After changing $a\mapsto bab$, we can bound this integral for $\re>>0$ by a multiple of 
\[\int_{Q(\KKv)}\lvert \det'_v(b)\lvert^{2\re-1}\lvert\mu_\Psi(q,k)\lvert \dv  q\dv k\dv  a,\] which converges since $\mu_\Psi$ is rapidly decreasing.
Thus, we justified the rewriting of the integral (\ref{E:strange}) and showed that the operator
\begin{equation}\label{E:preservation}\int_{Q\times K'_v}\Pi_v'\left(qk\right)\mu_\varphi\left(q,k\right)\dv  q\dv k\end{equation}preserves the local Shalika model, since a priori the operator does not preserve smoothness.

By collecting the results so far, we can prove the following.
By \ref{L:helpls} we find $\xi_{v,j}\in \mS_{\psi_v}^{\eta_v}(\Pi_v'),\, \Psi_j\in S\left(M'_{{2n},{2n}}\right)$ such that
\[L\left(s,\Pi_v'\right)\stackrel{(\ref{L:helpls})}{=}\sum_j\int_{\GL'_{2n}\left(\KKv\right)}\xi_{v,j}\left(g\right)\Psi_j\left(g\right)\lvert\det'_v\left(g\right)\lvert^{s-{nd-\frac{1}{ 2}}}\dv g\stackrel{(\ref{E:strange})}{=}\]
\[=\sum_j\int_{\GL'_n\left(\KKv\right)}\xi_{v,j}'\left(\bpm g_1&0\\0&\mathrm{1}_n\epm\right)\lvert\det'_v\left( g\right)\lvert^{s-\frac{1}{2}}\dv g_1,\]
where \[\xi_{v,j}'\left(g\right)=\int_{\SL_{2n}'\left(\KKv\right)}\xi_{v,j}\left(gx\right)\mu_{\Psi_j}\left(x\right)\dv x.\]
Since we showed that (\ref{E:preservation}) preserves $\mS_{\psi_v}^{\eta_v}(\Pi_v')$, we have $\xi_{v,j}'\in \mS_{\psi_v}^{\eta_v}(\Pi_v')$ and therefore we proved the second claim of \ref{T:Aconnl}.

Next, we show the first claim of \ref{T:Aconnl}.
We apply the Dixmier-Malliavin lemma to $Q\times K_v'$ and write
\[\xi_v\left(g\right)=\sum_j\int_{\GL'_n\left(\KKv\right)\times U'_{(n,n)}\left(\KKv\right)\times K_v'}\xi_{v,j}\left(g\bpm b^{-1}&0\\0&\mathrm{1}_n\epm \bpm \mathrm{1}_n&u\\0&\mathrm{1}_n\epm \bpm \mathrm{1}_n&0\\0&b\epm k\right)\]\[\Gamma_j\left(u,b,k\right)\lvert\det'_v\left(b\right)\lvert^{nd}\dv b\,\dv u\,\dv k,\]
where $\Gamma_j$ are smooth functions with compact support on $\GL'_n\left(\KKv\right)\times U'_{(n,n)}\left(\KKv\right)\times K_v'$. Let $\Lambda_j$ be the projection of the support of $\Gamma_j$ to 
$U'_{(n,n)}\left(\KKv\right)$ and
let $\Psi\in S\left(M'_{n,n}\right)$ be such that $\Psi\left(b^{-1}\right)=1$ for $b\in \bigcup_j \Lambda_j$, where we identify $U_{(n,n)}'$ with $M_{n,n}'$.
Define
\[\Gamma_j'\left(\bpm a&x\\0&b\epm; k\right)\coloneqq \int_{U'_{(n,n)}\left(\KKv\right)\times U'_{(n,n)}\left(\KKv\right)}\Gamma_j\left(u,b,k\right)\Psi\left(v\right)\psi\left(\mathrm{Tr}\left(xu-yv\right)\right)\dv u\dv v.\]
Then $\zeta_v\left(s,\xi_v\right)$ can be written as
\begin{equation}\label{E:strangeu}\begin{gathered}
    \sum_j\int_{H_n'\left(\KKv\right)\times U'_{(n,n)}\left(\KKv\right)\times K'_v}\xi_{v,j}\left(\bpm a&x\\0&b\epm k\right)\Gamma'_j\left(\bpm a&x\\0&b\epm; k\right)\\\lvert\det'_v\left( a\right)\lvert^{s+nd-\frac{1}{2}}\lvert\lvert\det'_v\left( b\right)\lvert^{s+nd-\frac{1}{2}}\dv a\dv b\dv x\dv k.\end{gathered}
\end{equation}
Let \[\Omega_1\coloneqq \{\left(a,b\right)\in M_{n,{2n}}': \left(a,b\right):\DD^{2n}\rightarrow \DD^n\text{ surjective}\}.\] The group $\GL'_n(\KKv)$ acts from the left and the group $K_v'$ from the right on $\Omega_1(\KKv)$. The resulting action of $\GL'_n(\KKv)\times K_v'$ is transitive. The stabilizer of $\left(0,\mathrm{1}_n\right)$ is the group $\left(k_2^{-1},k\right)$, where \[k=\bpm k_1&0\\0&k_2\epm\in K_v'\cap P_{(n,n)}'(\KKv)\] for some $k_1$. Let  
\[ \Omega\coloneq\left\{\bpm a&b\\c&d\epm\in M_{{2n},{2n}}': \left(c,d\right)\in \Omega_1\right\}.\] Let $\mathscr{S}\left(\Omega\right)$ be the space of smooth functions $\varphi\colon \Omega(\KKv)\ra \CC$ such that 
\begin{enumerate}
    \item $\lvert a\lvert^{2n}\lvert b\lvert^{2n}\varphi(g),\, g=\bpm a&b\\c& d\epm$ is bounded for all $n\in \mathbb{ZZ}$,
    \item The projection of the support of $\phi$ to $\Omega_1(\KKv)$ is compact,
    \item If $D$ is a differential operator which commutes with additive changes in $\left(a,b\right)$ then $D\varphi\in \mathscr{S}\left(\Omega\right)$.
\end{enumerate}
Analogously we define the space
\[\Omega_0\coloneqq \left\{
\bpm a&b\\0&d\epm\in M_{{2n},{2n}}':\det'_v\left(d\right)\neq 0\right\}\]
and $\mathscr{S} \left(\Omega_0\times K'_v\right)$. The natural map
\[r\colon \Omega_0(\KKv)\times K_v'\rightarrow \Omega(\KKv),\, \left(p,k\right)\mapsto pk\]
is surjective, proper, and a submersion and the inverse image of $pk$ is \[r^{-1}(pk)=\{\left(pk'^{-1},k'k\right):k'\in K_v'\cap P_{(n,n)}'\left(\KKv\right)\}.\]
\begin{lemma}\label{L:last}
Let $\varphi\in \mathscr{S}\left(\Omega_0\times K'\right)$. Then
\[\varphi_*\left(pk\right)=\int_{K'_v\cap P_{(n,n)}(\KKv)}\varphi\left(pk'^{-1},k'k\right)\dv k'\]
belongs to $\mathscr{S}\left(\Omega\right)$.
\end{lemma}
Once the lemma is established, the remaining claims of \ref{T:Aconnl} can be shown exactly like in \cite{FriJac}. 
\begin{proof}[Proof of \ref{L:last}]
It is easy to check that $\varphi_*$ is well defined and that the first two properties are satisfied, so it remains to check the third. Let $D$ be a differential operator of order $1$ on $\Omega$, which is independent of $\left(a,b\right)$. Since $r$ is submersive, there exists a pullback differential operator $D^*$ on $\Omega_0(\KKv)\times K'_v$ such that $\left(D^*\phi\right)_*=D\phi_*$, hence, it is enough to show that $D^*$ leaves $\mathscr{S}\left(\Omega_0\times K'\right)$ invariant. 
Assume that $D$ is an operator in $\left(a,b\right)$, hence, without loss of generality it acts on a function $\varphi'$ by
\[\frac{d}{dt}\restr{\varphi'\left(\bpm a+tX&b+tX\\ c&d\epm\right)}{t=0}\]
at a matrix $X$. Then we can choose $D^*$ such that it acts on a function $\varphi$ by
\[\frac{d}{dt}\restr{\varphi\left(\bpm x+tXk^{-1}&y+tYk^{-1}\\ 0&m\epm;k\right)}{t=0},\]
which is a differential operator in the variables $x,y$, whose coefficients depend only on $k$. Therefore, the obtained function stays in $\mathscr{S}\left(\Omega_0\times K_v'\right)$.
The second possibility is that $D$ is a differential operator on $\Omega_1$. Since any such operator is the linear combination of operators defined by invariant vector fields on $\GL'_n(\KKv)$ and $K'_v$. First assume that $D$ acts on $\varphi'$ by
\[\frac{d}{dt}\restr{\varphi'\bpm a&b\\ \exp\left(tX\right)c&\exp\left(tX\right)d\epm}{t=0}\]
where $X$ is an element of the Lie algebra of $\GL'_n(\KKv)$.
Then we can choose again $D^*$ such that it acts on $\varphi$ by
\[\frac{d}{dt}\restr{\varphi\left(\bpm x&y\\ 0&\exp\left(tX\right)m\epm ;k\right)}{t=0},\]
which clearly leaves $\mathscr{S}\left(\Omega_0\times K'\right)$ invariant.
Finally, for an element $Y\in\mathfrak{k}_v'$, the value of $D$ on $\varphi'$ is the difference, of the two operators $D_1$ and $D_2$ applied to $\varphi'$, given by
\[\frac{d}{dt}\restr{\varphi'\left( \bpm a&b\\ c&d\epm\exp\left(tY\right) \right)}{t=0}\text{ and }\frac{d}{dt}\restr{\varphi'\bpm a\exp\left(tY\right)&b\exp\left(tY\right)\\ c&d\epm}{t=0}.\]
We can then choose $D_1^*$ to act as $\frac{d}{dt}\restr{\phi\left(p,k\exp\left(tY\right)\right)}{t=0},$ which preserves $\mathscr{S}\left(\Omega_0\times K'\right)$. By the first case we considered, $D_2*$ does so too, since it is a differential operator in $\left(a,b\right)$ with polynomial coefficients.
\end{proof}
The last step in the proof of \ref{T:Aconnl} concerns the special case of $\Pi_v'$ and $\psi_v$ being unramified.
Thus, assume $\Pi_v'$ is unramified and let $\xi_{v,0}$ be the corresponding vector in the Shalika model, \emph{i.e.} the one fixed by $\GL_{2d_vn}\left(\OO_v'\right)$. Then, using \cite[Lemma 6.10]{GodJac}, we know that
\[L\left(s,\Pi_v'\right)=\int_{\GL_{2d_vn}\left(\OO_v'\right)}f\left(g\right)\lvert\det'_v\left( g\right)\lvert^{s-{nd-\frac{1}{ 2}}}\dv g,\]
where $f$ is a spherical function attached to $\Pi_v'$, \emph{i.e.} the matrix coefficient of $\Pi_v'$ is of the form $g\mapsto v_0^\lor(\Pi_v'(g)v_0)$, where $v_0$ and $v_0^\lor$ are non-zero vectors of $\Pi_v'$ and $\Pi_v'^\lor$ fixed by by the maximal open compact subgroup $\GL_{2d_vn}\left(\OO_v'\right)$. Let $\Psi_v$ be the characteristic function of $\GL_{2d_vn}\left(\OO_v'\right)$.
Then following the proof of \ref{L:helpls} shows that \[L\left(s,\Pi_v'\right)=\int_{\GL'_{2n}\left(\KKv\right)}\xi_{v,0}\left(g\right)\Psi_v\left(g\right)\lvert\det'_v\left( g\right)\lvert^{s-{nd-\frac{1}{ 2}}}\dv g\]
Recall, that $\zeta\left(s,\xi_{v,0}\right)$ can by (\ref{E:strange}) also be written as
\[\zeta_v\left(s,\xi_{v,0}\right)=\]\[=\int_{H_n'\left(\KKv\right)\times U'_{(n,n)}\left(\KKv\right)\times K_v'}\xi_{v,0}\left(\bpm a&0\\0&\mathrm{1}_n\epm \bpm b^{-1}&0\\0&\mathrm{1}_n\epm\bpm \mathrm{1}_n&u\\0&\mathrm{1}_n\epm\bpm \mathrm{1}_n&0\\0&b\epm k \right)\]\begin{equation}\label{E:last}\Xi\left(u,b,b^{-1};k\right)\lvert\det'_v\left( a\right)\lvert^{s-\frac{1}{2}}\lvert\det'_v\left( b\right)\lvert^{nd}\dv a\dv b\dv u\dv k,\end{equation}
where we plugged in the definition of $\mu_{\Psi_v}$ and $\Xi$ is defined by (\ref{E:Xi}).
It is easy to see that $\Xi$ vanishes unless $u$, $b$ and $b^{-1}$ have entries in $\OO_v'$, since the conductor of $\psi_v$ is $\OO_v$. Therefore, (\ref{E:last}) equals to 
\[\int_{\GL'_{2n}\left(\KKv\right)}\xi_{v,0}\left(g\right)\Psi_v\left(g\right)\lvert\det'_v\left(g\right)\lvert^{s-{nd-\frac{1}{ 2}}}\dv g,\]
which proves the claim.
\end{proof}

\end{document}